\documentclass[a4paper,twoside,11pt,reqno]{amsart}    

\usepackage{a4wide} %%Smaller margins = more text per page.
\usepackage{amsmath,bbm}
\usepackage{amssymb}
\usepackage{graphicx}
\usepackage{color}
\usepackage{latexsym}
\usepackage[english]{babel}   
\usepackage[utf8]{inputenc}
\usepackage[T1]{fontenc}
\usepackage{stmaryrd,comment}
\usepackage{cite}
\usepackage{url}
\usepackage[plainpages=false,pdfpagelabels,hypertexnames=false]{hyperref}

\usepackage[usenames,dvipsnames]{xcolor}
\definecolor{celestialblue}{rgb}{0.29, 0.59, 0.82}

\usepackage{tikz}
\usetikzlibrary{decorations.pathreplacing}
\usetikzlibrary{calc}

\usepackage{booktabs}

\usepackage[normalem]{ulem}
\usepackage{soul}

\catcode`\@=11
\def\section{\@startsection{section}{1}%
 \z@{.7\linespacing\@plus\linespacing}{.5\linespacing}%
 {\normalfont\bfseries\scshape\centering}}

\def\subsection{\@startsection{subsection}{2}%
 \z@{.5\linespacing\@plus\linespacing}{.5\linespacing}%
 {\normalfont\bfseries\scshape}}

\def\subsubsection{\@startsection{subsubsection}{3}%
 \z@{.5\linespacing\@plus\linespacing}{-.5em}%{.5\linespacing}%
 {\normalfont\bfseries\itshape}}

\definecolor{darkgreen}{rgb}{0,0.4,0}
\definecolor{BrickRed}{rgb}{0.65,0.08,0}
\hypersetup{colorlinks=true,linkcolor=blue,citecolor=red,filecolor=BrickRed,urlcolor=darkgreen}
\newcommand{\beq}{\begin{equation}}
\newcommand{\eeq}{\end{equation}}

\newcommand{\gf}{generating function}
\newcommand{\gfs}{generating functions}
\newcommand{\fps}{formal power series}
\newcommand{\Maple}{{\sc Maple}}

\newcommand{\Pzero}{U_0}
\newcommand{\Pun}{U_1}
\newcommand{\Ptwo}{U_2}
\newcommand{\num}{N}

\newcommand{\zt}{\tilde z}

\DeclareMathOperator{\Rat}{Rat}
\DeclareMathOperator{\numsmall}{num}
\DeclareMathOperator{\den}{den}

\DeclareMathOperator\OS{OS}
\newcommand{\ns}{{\mathbb N}} % Natural numbers
\newcommand{\zs}{{\mathbb Z}} % Integers
\newcommand{\qs}{{\mathbb Q}}  % Rationals
\newcommand{\rs}{{\mathbb R}} % real numbers
\newcommand{\LandauO}{\mathcal{O}}

\newcommand{\Bc}{\mathcal{B}}
\newcommand{\Cc}{\mathcal{C}}
\newcommand{\Qc}{\mathcal{Q}}

\newcommand{\Lc}{\mathcal{L}}

\newcommand{\Sc}{\mathcal{S}}

\newcommand{\C}{\mathbb{C}}

\newcommand{\Q}{\mathbb{Q}}
\newcommand{\Z}{{\mathbb Z}}

\newcommand{\tQ}{\widetilde Q}
\newcommand{\tC}{\widetilde C}
\newcommand{\KK}{\mathbb{K}} 
\newcommand{\vareps}{\varepsilon}

\newtheorem{theo}{Theorem}[section]
\newtheorem{lemma}[theo]{Lemma}
\newtheorem{prop}[theo]{Proposition}
\newtheorem{coro}[theo]{Corollary}

\newtheorem{conjecture}[theo]{Conjecture}

\theoremstyle{remark}
\newtheorem*{remark}{Remark}

\newcommand{\id}{{\rm id} }
\newcommand{\GA}{\mathbb{A}}
\newcommand{\GK}{\mathbb{K}}
\newcommand{\cS}{\mathcal S}
\newcommand{\spol}{S}
\newcommand{\bx}{\bar{x}}
\newcommand{\by}{\bar{y}}
\newcommand{\bu}{\bar{u}}
\newcommand{\bv}{\bar{v}}

\newcommand{\bzeta}{\bar\zeta}  %root of unity

\newcommand{\xzero}{u}
\newcommand{\av}{v}
\newcommand{\aw}{w}

\newcommand{\DSA}{B_1}
\newcommand{\DSB}{B_2}

\newcommand{\tz}{ \tilde z}
\newcommand{\St}{\tilde{S}}
\newcommand{\T}{\tilde w}
\newcommand{\Sh}{\hat{S}}
\newcommand{\Rh}{\hat{R}}

\DeclareMathOperator\Pol{Pol}

\DeclareMathOperator\Alg{Alg}
\def\emm#1,{{\em #1}}

\newcommand*\circled[1]{\tikz[baseline=(char.base)]{
            \node[shape=circle,draw,inner sep=1pt] (char) {\footnotesize #1};}}
 
\newcommand{\hookSEarrow}{\mathrel{\rotatebox[origin=c]{-45}{$\hookrightarrow$}}}
\newcommand{\hookNEarrow}{\mathrel{\rotatebox[origin=c]{45}{$\hookrightarrow$}}}

\begin{document}

\author{Mireille Bousquet-M\'elou \and Michael Wallner}

\thanks{MBM was partially supported by the ANR projects DeRerumNatura (ANR-19-CE40-0018) and Combiné (ANR-19-CE48-0011). MW was supported by an Erwin Schr\"odinger Fellowship and a Stand-Alone Project of the Austrian Science Fund (FWF): J~4162 and P~34142.}

\date{\today}

\title{Walks avoiding a quadrant and the reflection principle}

\begin{abstract}
  We continue the enumeration of plane lattice walks
   with small steps avoiding the negative quadrant, initiated by the first author in 2016. 
  We solve in detail a new case, namely the king model where all eight nearest neighbour steps are allowed.  The associated \gf\ is proved to be the sum of  a simple, explicit D-finite series (related to the number
  of walks confined to the first quadrant), and an algebraic one. This was already the case for the two models solved 
  by the first author  in 2016.
  The principle of the approach is also the same,
  % as in~\cite{Bousquet2016},
  but challenging theoretical and computational difficulties arise as we now handle algebraic series of larger degree.

We expect a similar algebraicity phenomenon to  hold for the seven \emm Weyl,\ step sets, which are those for which walks confined to the first quadrant can be counted using the reflection principle. With this paper, this is now proved for three of them. For the remaining four, we predict the D-finite part of the solution, and in three of the four cases, give evidence for the algebraicity of the remaining part.
\end{abstract}

\maketitle

% \tableofcontents

%%%%%%%%%%%%%%%%%%%%%%%%%%%%%%%%%%%%%%%%%%%%%%%%%%%%%%%%%%%
\section{Introduction}
\label{sec:intro}
%%%%%%%%%%%%%%%%%%%%%%%%%%%%%%%%%%%%%%%%%%%%%%%%%%%%%%%%%%

Over the last two decades, the enumeration of walks in the non-negative quadrant
\[
  \Qc:= \{ (i,j) : i \geq 0 \text{ and } j \geq 0 \}
\]
has attracted a lot of attention  and  established its own scientific community with {close to a hundred research papers; see, e.g., \cite{bomi10} and citing papers.
One of its {attractive} features is the diversity of the used tools, such as
algebra on \fps~\cite{bomi10,Mishna-jcta}, bijective approaches~\cite{Bern07,chyzak-yeats},
computer algebra\cite{BoKa08,KaKoZe08}, 
complex analysis\cite{raschel-unified,BeBMRa-17},
probability theory\cite{BoRaSa14,DenisovWachtel15}, and difference
Galois theory\cite{DHRS-17}. Most of the attention has focused on  walks \emm with small steps,, that is, taking their steps in a fixed subset $\cS$ of $\{-1,0, 1\}^2\setminus{(0,0)}$. For each such step set~$\cS$  (often called a \emm model, henceforth), one considers a trivariate \gf\ $Q(x,y;t)$ defined by
\beq\label{Q-def}
  Q(x,y;t)= \sum_{n \ge 0}\sum_{i,j \in \Qc} q_{i,j}(n) x^i y^j t^n,
\eeq
where $q_{i,j}(n)$ is the number of quadrant walks with steps in $\cS$, starting from $(0,0)$, ending at $(i,j)$, and having in total $n$ steps. For each $\cS$, one now knows whether and where this series fits in the following classical hierarchy of series:
\[ 
  \hbox{ rational } \subset \hbox{ algebraic } \subset  \hbox{ D-finite }
        \subset \hbox{ D-algebraic}.
\] 
Recall  that a series (say $Q(x,y;t)$ in our case) is \emm rational, if it is the ratio of two polynomials, \emm algebraic, if it satisfies a polynomial equation (with coefficients that are polynomials in the variables), \emm D-finite, if it satisfies three \emm linear, differential equations (one in each variable), again with polynomial coefficients, and finally \emm D-algebraic,
if it satisfies three \emm polynomial, differential equations. It has been known since the $1980$s~\cite{gessel-factorization} that the \gf\ of walks confined to a half-plane is algebraic. This explains why $Q(x,y;t)$ is algebraic in some cases, for instance when $\cS=\{\rightarrow, \uparrow, \leftarrow\}$: indeed, confining walks to the first quadrant is then equivalent to confining them to the right half-plane $i\ge 0$. It was shown in~\cite{bomi10} that exactly $79$ (essentially distinct)
quadrant problems with small steps  are \emm not, equivalent to any half-plane problem.
One central result in the classification of these $79$ models  is that  $Q(x,y;t)$ is D-finite if and only if a certain group, which is easy to construct from the step set $\cS$, is finite~\cite{bomi10,BoRaSa14,KuRa12,BoKa08,Mishna-Rechni,melczer-mishna-sing}. 

Since any strictly convex closed cone can be deformed into the first quadrant, the enumeration of walks confined to $\Qc$  captures all such counting problems (provided we consider all possible step sets $\cS$, not only small steps).
Similarly, any non-convex closed cone in two dimensions can be deformed into  the three-quadrant plane 
\beq\label{Ccone-def}
	\Cc := \{ (i,j) : i \geq 0 \text{ or } j \geq 0 \},
\eeq
and in 2016, the first author initiated  the enumeration of  lattice paths confined to $\Cc$~\cite{Bousquet2016}. 
Therein, the two most natural models of walks were studied:
\emm simple walks, with steps in $\{\rightarrow,  \uparrow, \leftarrow, \downarrow \}$, 
and \emm diagonal walks, with steps in $\{\nearrow, \nwarrow, \swarrow,  \searrow \}$.
In both cases, the generating function
\beq\label{C-def}
  C(x,y;t)=\sum_{n\ge 0 } \sum_{i,j\in \Cc}c_{i,j}(n) x^i y^j t^n
\eeq
defined analogously to $Q(x,y;t)$ (see~\eqref{Q-def})
was proved to differ  from the series
\beq\label{Q-combin}
\frac{1}{3}\left( Q\left(x,y;t\right) -  \frac{1}{x^2}\,Q\!\left(\frac{1}{x},y;t\right) - \frac{1}{y^2}\,Q\!\left(x,\frac{1}{y};t\right) \right)
\eeq
by an \emm algebraic,  one. In both cases, the underlying group is finite, hence $Q(x,y;t)$ is D-finite and $C(x,y;t)$ is  D-finite as well.

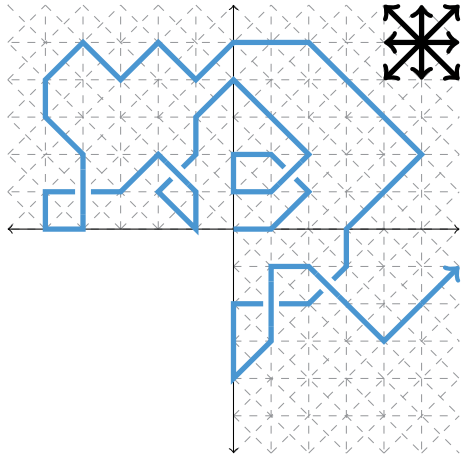
\begin{figure}[hb]
	\centering
	\scalebox{0.99}{%\documentclass[crop]{standalone}% 'crop' is the default for v1.0, 
%
%\usepackage[dvipsnames]{xcolor}
%\definecolor{celestialblue}{rgb}{0.29, 0.59, 0.82}
%
%\usepackage{tikz}
%\usetikzlibrary{decorations.pathreplacing}
%\usetikzlibrary{calc}
%
%\begin{document}

\begin{tikzpicture}[scale=0.5]
	\newcommand*{\quadw}{6}
	\newcommand*{\quadh}{6}
	\pgfmathsetmacro{\dd}{\quadh-\quadw}
	
	\draw[step=1cm, Gray,very thin,dashed] (-0,-\quadh+0.1) grid (\quadw-0.1,0);
	\draw[step=1cm, Gray,very thin,dashed] (-\quadw+0.1,-0) grid (\quadw-0.1,\quadh-0.1);
	
	%draw diagonal grid lines
	\foreach \i in {0,...,\quadh}
  {		
		\draw[-,Gray,very thin,dashed] (-\quadw,\i) -- (-\i+\dd,\quadh);
		\draw[-,Gray,very thin,dashed] (0,-\i) -- (\quadw,\quadh-\i-\dd);
	}		
	\pgfmathsetmacro{\ll}{\quadw-1}
	\foreach \i in {1,...,\ll}
  {				
		\pgfmathsetmacro{\x}{min(\quadw-\i+\dd,\quadw)}
		\draw[-,Gray,very thin,dashed] (-\i,0) -- (\x,\quadh);
		\draw[-,Gray,very thin,dashed] (\i,-\quadh) -- (\quadw,-\i-\dd);
	}
	
	\pgfmathsetmacro{\ll}{\quadw-1}
	\foreach \i in {-\quadw,...,\ll}
  {		
		\draw[-,Gray,very thin,dashed] (\i,\quadh) -- (\quadw,\i);
	}
	\pgfmathsetmacro{\ll}{\quadw-1}
	\foreach \i in {1,...,\ll}
  {		
		\draw[-,Gray,very thin,dashed] (-\i,0) -- (-\quadw,\quadh-\i);
		\draw[-,Gray,very thin,dashed] (0,-\i) -- (\quadw-\i,-\quadh);
	}

	%draw axis
	\draw [<->] (0,-\quadh) -- (0,\quadh);
	\draw [<->] (-\quadw,0) -- (\quadw,0);

	%draw path	  
	\def\xe{0}
	\def\ye{0}
	\coordinate (current point) at (0,0);
	\draw[-,color=celestialblue,line width=2] (0,0) 
	\foreach \x/\y in {1/0,
	                   1/1, 
	                   -1/1,
	                   -1/0,
	                   0/-1,
										 1/0,
	                   1/1,
	                   -1/1,
	                   -1/1,
	                   -1/-1,
										 0/-1,
	                   -1/-1,
										 1/-1,
										 0/1,
	                   -1/1,
	                   -1/-1,
	                   -1/0,
										 -1/0,
										 0/-1,
										 1/0,
										 0/1,
										 0/1,
										 -1/1,
										 0/1,
										 1/1,
										 1/-1,
										 1/1,
										 1/-1,
										 1/1,
										 1/0,
										 1/0,
										 1/-1,
										 1/-1,
										 1/-1,
										 -1/-1,
										 -1/-1,
										 0/-1,
										 -1/-1,
										 -1/0,
										 -1/0,
										 0/-1,
										 0/-1,
										 1/1,
										 0/1,
										 0/1,
										 1/0,
										 1/-1,
										 1/-1,
										 1/1
										 }
	{	
		-- ++(\x,\y)
%		%% new start is old end
%		\pgfmathsetmacro{\xs}{\xe}
%		\pgfmathsetmacro{\ys}{\ye}		
%		
%		%% next dot
%		\pgfmathsetmacro{\xe}{\xs+\x}
%		\pgfmathsetmacro{\ye}{\ys+\y}
%		
%		%% draw new step 
%		\draw [-,line width=1] (\xs,\ys) -- (\xe,\ye);
%		
%		%% save end for next step
%		\xdef\xe{\xe}
%		\xdef\ye{\ye}
	}	;
	
	\newcommand*{\dis}{0.2}
	\newcommand{\drawcrossing}[2]{
	\coordinate (a) at #1;
	\coordinate (b) at #2;
	\draw[-,white,decorate,double=celestialblue,line width=2, double distance=2] {($(a)!0.2!(b)$) -- ($(a)!0.8!(b)$)};
	};
	
	\drawcrossing{(1,1)}{(2,2)};
	\drawcrossing{(-1,1)}{(-2,2)};
	\drawcrossing{(-4,0)}{(-4,2)};
	\drawcrossing{(2,-1)}{(3,-2)};
	\drawcrossing{(1,-3)}{(1,-1)};
		
	%% draw last step 
		\draw [->,line width=2, color=celestialblue] (5,-2) -- (6,-1);

	\coordinate (cross) at (5,5);
	\foreach \x/\y in { 1/0, 1/1, 0/1, -1/1, -1/0, -1/-1, 0/-1, 1/-1 }
	{	
		\draw[->, line width=2] (cross) -- ($(cross)+(\x,\y)$);
	}
\end{tikzpicture}

%\end{document}
}
	\caption{A king walk in the three-quadrant plane $\Cc$. The associated generating function is D-finite and transcendental (i.e., non-algebraic).}
	\label{fig:king}
      \end{figure}

It became then natural to explore more three-quadrant problems, in
particular to understand whether the D-finiteness of $C(x,y;t)$ was
again related to the finiteness of the associated group -- at least
for the $74$ three-quadrant problems that are not equivalent to a
half-plane problem; see Section~\ref{sec:interestingmodels}. Using an asymptotic argument, Mustapha quickly proved that the~$51$ three-quadrant problems associated with an infinite group have, like their quadrant counterparts, a non-D-finite solution~\cite{Mustapha2019Walks}. Regarding exact solutions,  Raschel and Trotignon obtained in~\cite{RaschelTrotignon2018Avoiding} sophisticated integral expressions for eight step sets. Four of them have a finite group (namely $\{\rightarrow, \uparrow, \leftarrow, \downarrow\}$, $\{\nearrow, \leftarrow,  \downarrow \}$, $\{\rightarrow, \uparrow, \swarrow \}$, and $\{\rightarrow, \nearrow, \uparrow, \leftarrow, \swarrow, \downarrow \}$), and these expressions imply that they are D-finite (at least in $x$ and $y$).
In fact, the latter three  are now known  to be
algebraic~\cite{mbm-tq-kreweras}. The other four have an infinite group
and have been further studied by Dreyfus and Trotignon: one of them is
D-algebraic, the other three are
not~\cite{dreyfus-trotignon}. Furthermore, the remarkable results of
Budd~\cite{Budd2020Winding} and Elvey
Price~\cite{ElveyPrice2020Winding} on the winding number of various
families of plane walks provide explicit D-finite expressions for
several \gfs\ of three-quadrant walks  starting and ending close to
the origin, in particular  for step sets
$\{\rightarrow,\nearrow, \leftarrow, \swarrow\}$
and
$\{\rightarrow,  \nwarrow, \leftarrow, \searrow \}$
in Budd's paper, as well as
$\{\uparrow,  \leftarrow, \searrow \}$
and
$\{\rightarrow, \uparrow, \nwarrow, \leftarrow, \downarrow, \searrow\}$
in Elvey Price's\footnote{In both papers, when steps $\nwarrow$ or $\searrow$ are allowed, one includes in the enumeration walks using jumps from $(-1,0)$ to $(0,-1)$ and vice-versa. Such jumps are forbidden in this paper, but we show in the last section that allowing them in king walks does not significantly modify the form of our results.}.

\medskip
\noindent{\bf Main results.} 
In this paper we enrich the collection of completely solved cases with
the {\emph{king walks}}, in which all eight nearest neighbour steps $\rightarrow, \nearrow, \uparrow, \nwarrow, \leftarrow, \swarrow,
\downarrow, \searrow $ are allowed; see Figure~\ref{fig:king}. This is again a finite group model, and the series $Q(x,y;t)$ is a well-understood D-finite series~\cite{bomi10}.
Here we determine $C(x,y;t)$, and show that the algebraicity phenomenon of~\cite{Bousquet2016} persists: the series  $C(x,y;t)$ differs from the linear combination~\eqref{Q-combin} by an algebraic series, this time of degree $216$.
For the simple and diagonal walks of~\cite{Bousquet2016} this
algebraic series was of degree $72$ ``only''. The \gf \ $C_{i,j}$ of
walks ending at a prescribed position $(i,j)$ differs from a series 
of the form $\pm Q_{k,\ell}/3$
by an algebraic series of degree at most $24$ (while
this degree was bounded by $8$ in  the two models
of~\cite{Bousquet2016}).

Moreover, we explain why we expect a similar property to hold  (with variations on the linear combination~\eqref{Q-combin} of series~$Q(\cdot, \cdot)$)
for the seven models of Table~\ref{tab:weyl}. These are precisely the models  for which the quadrant
problem can be solved using the reflection principle~\cite{gessel-zeilberger}, and for this reason we call them \emm Weyl models,. We predict  the relevant linear combination of series~$Q$, and we give evidence of the algebraicity phenomenon for the three rightmost models of Table~\ref{tab:weyl}.
However, we also expect the effective solution of these models to be challenging in computational terms,  because the relevant algebraic series will most likely have  very large degrees.

 \begingroup
\setlength{\tabcolsep}{4pt} 

\begin{table}
  \centering
  \begin{tabular}{l|c|c|c|c}
    model $\cS$&
 \begingroup
                 \setlength{\tabcolsep}{1pt}
                 \begin{tabular}{cccc}
              \includegraphics[height=0.8cm]{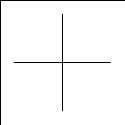}&\includegraphics[height=0.8cm]{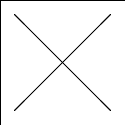} &\includegraphics[height=0.8cm]{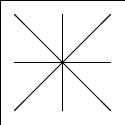}&\includegraphics[height=0.8cm]{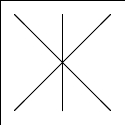}\\
              \small{simple}& \small{diagonal}& \small{king} &\small{diabolo }
                                       \end{tabular}
                                       \endgroup
           &
             \begin{tabular}{cc}
               \includegraphics[height=0.8cm]{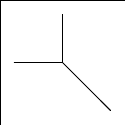}&\includegraphics[height=0.8cm]{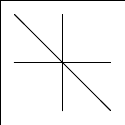} \\
                \small{tandem} &\small{double-tandem}
\\
             \end{tabular}
          &
            \begin{tabular}{c}
     \includegraphics[height=0.8cm]{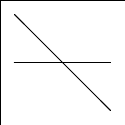} \\
              \small{Gouyou-Beauchamps}
                      \end{tabular}
  \\
       \hline
    group $G$ & \begin{minipage}{3cm}
      $(x,y)$,$(\bar x, y)$,\\
      $(\bar x, \bar y)$,$(x, \bar y)$
        \end{minipage} & 
 \begin{minipage}{3cm}                                                                                  $(x,y), (\bar x y,y)$,\\  
$(\bar x y, \bar x), (\bar y, \bar x)$,\\
$(\bar y, x\bar y ), (x,x\bar y )$  
\\
\end{minipage}                                                                                                                       &  \begin{minipage}{3cm}
\mbox{}\\$(x, y), (\bar x y, y),$\\ $(\bar x y, {\bar x}^2 y), (\bar x, {\bar x}^2y)$,\\
$ (\bar x, \bar y), (x\bar y, \bar y),$\\$(x\bar y, x^2\bar y), (x,x^2\bar y )$\\
\end{minipage}\\
    new steps &
     \includegraphics[height=0.8cm]{10101010}\hskip 3mm\includegraphics[height=0.8cm]{01010101} \hskip 2mm\includegraphics[height=0.8cm]{11111111}\hskip 3mm\includegraphics[height=0.8cm]{11011101}& \includegraphics[height=0.9cm]{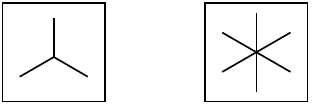}&  \includegraphics[height=0.8cm]{10101010}\\
   \begin{minipage}{2cm}
 \vskip -36mm    Weyl\\ chambers  
    \end{minipage}
  &\scalebox{0.7}{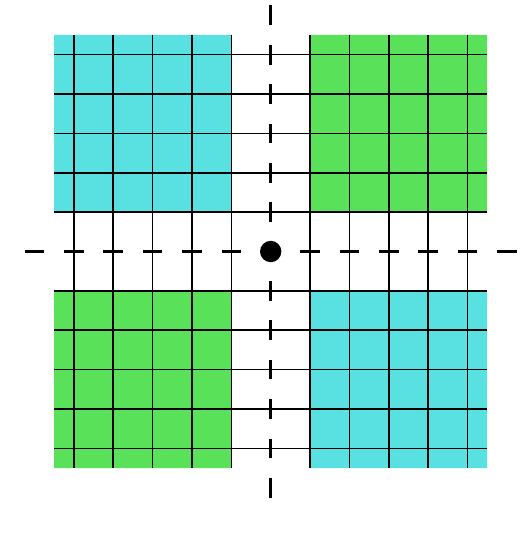} & \scalebox{0.7}{\begin{picture}(0,0)%
\includegraphics{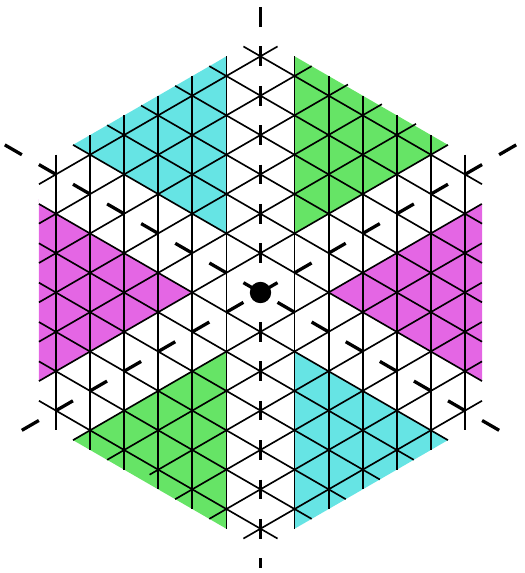}%
\end{picture}%
\setlength{\unitlength}{4144sp}%
\begingroup\makeatletter\ifx\SetFigFont\undefined%
\gdef\SetFigFont#1#2#3#4#5{%
  \reset@font\fontsize{#1}{#2pt}%
  \fontfamily{#3}\fontseries{#4}\fontshape{#5}%
  \selectfont}%
\fi\endgroup%
\begin{picture}(2382,2609)(836,-2063)
\put(2612,-1884){\makebox(0,0)[lb]{\smash{{\SetFigFont{14}{16.8}{\familydefault}{\mddefault}{\updefault}{\color[rgb]{0,0,0}$A_2$}%
}}}}
\end{picture}%
} &  \scalebox{0.7}{\begin{picture}(0,0)%
\includegraphics{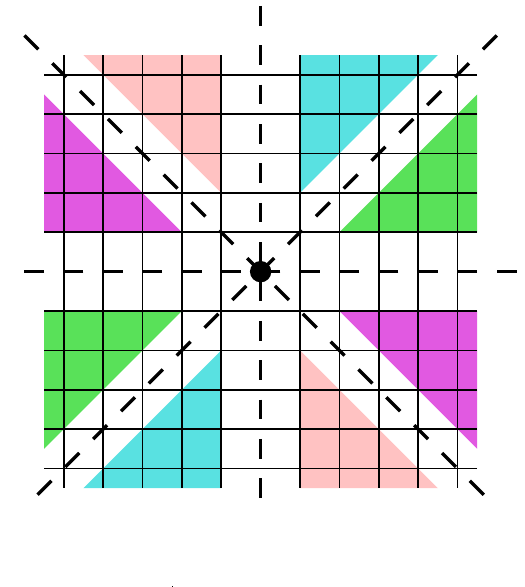}%
\end{picture}%
\setlength{\unitlength}{4144sp}%
\begingroup\makeatletter\ifx\SetFigFont\undefined%
\gdef\SetFigFont#1#2#3#4#5{%
  \reset@font\fontsize{#1}{#2pt}%
  \fontfamily{#3}\fontseries{#4}\fontshape{#5}%
  \selectfont}%
\fi\endgroup%
\begin{picture}(2384,2678)(1239,-1952)
\put(3151,-1771){\makebox(0,0)[lb]{\smash{{\SetFigFont{14}{16.8}{\familydefault}{\mddefault}{\updefault}{\color[rgb]{0,0,0}$B_2$}%
}}}}
\end{picture}%
} 
  \end{tabular}
  \vskip 3mm
  \caption{The seven Weyl models, with their usual names and  groups (defined in Section~\ref{sec:group}). The next rows show how to deform the steps and the plane so that walks in the first quadrant correspond to walks in a Weyl chamber, and walks avoiding the negative quadrant to walks avoiding a Weyl chamber.}
  \label{tab:weyl}
\end{table}
\endgroup

\medskip\noindent
{\bf Outline of the paper.} 
We begin in Section~\ref{sec:general} with  generalities on the enumeration  of walks with small steps confined to the three-quadrant cone $\Cc$, and on the related functional equations. We describe the $74$ (essentially distinct) models of interest -- those that are not equivalent to a half-plane model -- and define the group associated with a model. In Section~\ref{sec:OS} we state and justify our conjecture on the form of $C(x,y;t)$ in the seven Weyl cases. The next three sections are devoted to the solution of the king model in three quadrants: in Section~\ref{sec:KingResults} we state our results in details,  and we prove them in  Sections~\ref{sec:KingEquation} and~\ref{sec:quadratic}. In Section~\ref{sec:combi}, we give combinatorial proofs, based on the reflection principle, of some identities obtained so far via functional equations, and generalize them to all seven Weyl models. In Section~\ref{sec:final} we conclude with a few comments, in particular about what happens to the king model when one allows steps between $(0,-1)$ and $(-1,0)$.

This paper is the full version of an extended abstract~\cite{mbm-wallner-AofA} that was published in the proceedings of the \emph{Analysis of Algorithms} Conference in 2020.

%%%%%%%%%%%%%%%%%%%%%%%%%%%%%%%%%%%%%%%%%
\section{Enumeration in the three-quadrant plane: basic tools}
\label{sec:general}
%%%%%%%%%%%%%%%%%%%%%%%%%%%%%%%%%%%%%%%%%%

Let us begin with  some definitions and notation on \fps.
Let $\GA$ be a commutative ring and $x$ an indeterminate. We denote by
$\GA[x]$ (resp.~$\GA[[x]]$) the ring of polynomials (resp.~\fps) in $x$
with coefficients in $\GA$. If $\GA$ is a field,  then $\GA(x)$ denotes the field
of rational functions in $x$, and $\GA((x))$ the
field of Laurent series in $x$, that is, series of the form
$ \sum_{n \ge n_0} a_n x^n$,
with $n_0\in \zs$ and $a_n\in \GA$.  This notation is generalized to polynomials, fractions, and series in several indeterminates.
 The coefficient of $x^n$ in a   series $F(x)$ is denoted by
 $[x^n]F(x)$. We denote partial derivatives with indices: for instance, for a series $F$ involving the indeterminate $x$, we write $F_x$ for $\partial F/\partial x$.

 We denote {with} bars the reciprocals of variables: that is, $\bx=1/x$,
so that $\GA[x,\bx]$ is the ring of Laurent polynomials in $x$ with
coefficients in $\GA$.

We will often handle series of $\qs(x)((t))$, and consider $\qs(x)$ as a subring of $\qs((x))$ (that is, we expand rational functions in $x$ around $x=0$). For $F(x;t) \in \qs((x))((t))$, of the form
\[
F(x;t)= \sum_{n\ge n_0} t^n  \sum_{i \ge i_0(n)} a(n, i) x^i,
\]
the \emm non-negative part of $F$ in, $x$ is the
following series in $t$ and $x$:
\[
  [x^{\geq}]F(x;t)= \sum_{n\ge n_0} t^n  \sum_{i \ge \max(0, i_0(n))} a(n, i) x^i.
\]
We define analogously the \emm positive part, of $F$, denoted by $[x^>]F$. The \emm negative part, of $F$ is $[x^<]F:=F-[x^\ge ]F$. Observe that, if $F(x;t) \in \qs(x)((t))$, this convention makes the roles of $x$ and $1/x$ non-symmetric: the negative part of $F(x)$ is  not always obtained by inverting $x$ in the positive part of $F(\bx)$. For instance, if we take $F(x)=1/(1+x)$, its negative part is $0$, but $F(\bx)=1/(1+\bx)=x/(1+x)$ has a non-trivial positive part. However, if $F(x;t) \in \qs[x, \bx]((t))$, then the expected symmetry holds.

If $\GA$ is a field,  a power series $F(x) \in \GA[[x]]$
  is \emm algebraic, (over $\GA(x)$) if it satisfies a
non-trivial polynomial equation $P(x, F(x))=0$ with coefficients in
$\GA$. Otherwise it is \emm transcendental,. It is \emm differentially finite, (or \emm D-finite,) if it satisfies a non-trivial linear
differential equation with coefficients  in $\GA(x)$. For
multivariate series, D-finiteness requires the
existence of a differential equation \emm in each variable,.  {We
  refer to~\cite{Li88,lipshitz-df} for general results on D-finite series.}

 We usually omit the dependency in $t$ of our series, writing for instance $C(x,y)$ for $C(x,y;t)$.
For a series $F(x,y) \in \qs[x, \bx, y, \by][[t]]$ and two integers
$i$ and $j$, we denote by $F_{i,j}$ the coefficient of $x^i y^j$ in $F(x,y)$. 
This is a series in $\qs[[t]]$. We also denote
\beq \label{Fhv}
  F_{-,0}(\bx) =  \sum_{i<0}F_{i,0}\, x^i \quad\text{and} \quad
           F_{0,-}(\by) =  \sum_{j<0}F_{0,j}\, y^j .
\eeq
These two series lie in $\bx\qs[\bx][[t]]$ and $\by\qs[\by][[t]]$, respectively.

%=======================================================
\subsection{A functional equation}
%=======================================================

We fix a subset $\cS$ of $\{-1,0, 1\}^2\setminus\{(0,0)\}$ and we consider
walks with steps in $\cS$ that start from 
$(0,0)$ and  remain in the cone $\Cc$ defined by~\eqref{Ccone-def}.
By this, we mean that not only must every vertex of the walk
lie in $\Cc$, but also every edge: a walk containing a step from
$(-1,0)$ to $(0,-1)$ (or vice versa) does not lie in $\Cc$. We often say
for short that our walks \emm avoid the negative quadrant,. The \emm step polynomial, of $\cS$ is defined by
\beq\label{eq:steppolynomial}
  S(x,y)=\sum_{(i,j)\in \cS} x^i y^j%
  = \by H_-(x)+ H_0(x) +y H_+(x)
  = \bx V_-(y) + V_0(y) +x V_+(y),
\eeq
for some Laurent polynomials $H_-, H_0, H_+$ and $V_-, V_0, V_+$
% (of maximal degree $1$ and minimal valuation $-1$)
recording horizontal and vertical displacements, respectively.
%The letters should be mnemonics for ``horizontal'' and ``vertical''. 
We denote by $C(x,y) \equiv C(x,y;t)$ the
\gf~\eqref{C-def} of walks confined to $\Cc$. In the expression~\eqref{C-def},  $c_{i,j}(n)$ is the number of walks of length $n$ that go from
$(0,0)$ to $(i,j)$ and
are confined to $\Cc$. 

Constructing walks confined to $\Cc$ step by step gives the following functional equation:
\[
C(x,y)= 1 +tS(x,y) C(x,y) -t\by H_-(x) C_{-,0}(\bx) -t\bx V_-(y)C_{0,-}(\by) -t\bx\by
 C_{0,0} \mathbbm 1_{(-1,-1)\in \cS},
\]
where we have used the notation~\eqref{Fhv}. Note that the series $C_{-,0}(\bx)$ and $C_{0,-}(\by)$ count walks ending on  the horizontal and vertical boundaries of $\Cc$ (but not at $(0,0)$).
 On the right-hand side, the term $1$ accounts for the
 empty walk, the next term describes the extension of a walk in $\Cc$
 by one step of $\cS$, and each of the other three terms corresponds to
 a ``bad'' move, either starting from the negative $x$-axis, or from
 the negative $y$-axis, or from $(0,0)$. Equivalently,
  \beq\label{eqfunc-gen}
 K(x,y)C(x,y)= 1 -t\by H_-(x) C_{-,0}(\bx) -t\bx V_-(y)C_{0,-}(\by) -t\bx\by
 C_{0,0} \mathbbm 1_{(-1,-1)\in \cS},
 \eeq
 where $K(x,y):=1-tS(x,y)$ is the \emm kernel, of the equation.

As recalled in the introduction, the enumeration of walks confined to the non-negative quadrant $\Qc$ has been studied intensively over the last $20$ years.
 The associated \gf\ $Q(x,y)\equiv Q(x,y;t)\in \qs[x,y][[t]]$  defined in~\eqref{Q-def} satisfies a similarly looking equation~\cite[Lem.~4]{bomi10}:
 \beq\label{eqfunc-gen-qu}
 K(x,y)Q(x,y)= 1 -t\by H_-(x) Q(x,0) -t\bx V_-(y)Q(0,y) +t\bx\by
 Q(0,0) \mathbbm 1_{(-1,-1)\in \cS}.
 \eeq

  \begin{remark} In  two recent references dealing with the winding number of plane lattice walks~\cite{Budd2020Winding,ElveyPrice2020Winding}, it seems more natural to count walks in which all vertices lie in $\Cc$, but not necessarily all edges: this means that there may be steps form $(-1,0)$ to $(0,-1)$, and vice-versa. Counting these walks would add two terms to the right-hand side of~\eqref{eqfunc-gen}, namely
 \beq\label{ajout}
   t \by C_{-1,0} \mathbbm 1_{(1,-1)\in \cS} + t \bx C_{0,-1} \mathbbm 1_{(-1,1)\in \cS}.
   \eeq
   We discuss in the final section of the paper
   the enumeration of king walks in $\Cc$ when these two steps are allowed. The results are qualitatively the same as when they are forbidden.
    \end{remark}
 
%=======================================================
\subsection{Interesting step sets}
\label{sec:interestingmodels}
%=======================================================

As in the quadrant case~\cite{bomi10}, we can decrease the number of
step sets  that are worth being considered thanks to a few simple observations (\emm a priori,, there are $2^8$ of them):
\begin{itemize}
\item Since the cone $\Cc$ (as well as the quarter plane $\Qc$) is $x/y$-symmetric, the counting problems defined by $\cS$ and by its mirror image $\overline \cS:=\{(j,i): (i,j) \in \cS\}$ are equivalent; the associated \gfs\ are related by $\overline C(x,y)=C(y,x)$.
\item If all steps of $\cS$ are contained in the right half-plane $\{(i,j): i\ge 0\}$, then \emm all, walks with steps in $\cS$ lie in $\Cc$, and the series $C(x,y)=1/(1-tS(x,y))$ is simply rational. The series $Q(x,y)$ is known to be algebraic in this case~\cite{banderier-flajolet,bousquet-petkovsek-recurrences,Duchon98,gessel-factorization}.
\item If all steps of $\cS$ are contained in the left half-plane $\{(i,j): i\le 0\}$, then confining a walk to $\Cc$ is equivalent to confining it to the upper half-plane: the associated \gf\ is then algebraic, and so is $Q(x,y)$.
\item If all steps of $\cS$ lie (weakly) above the first diagonal
  ($i=j$), then confining a walk to~$\Cc$ is again equivalent to confining it to the upper half-plane: the associated \gf\ is then algebraic, and so is $Q(x,y)$.
\item  If all steps of $\cS$ lie (weakly) above the second diagonal ($i+j=0$), then all walks with steps in $\cS$ lie in $\Cc$, and $C(x,y)=1/(1-tS(x,y))$ is simply rational. In this case however, the series $Q(x,y)$ is not at all trivial~\cite{bomi10,Mishna-Rechni}. Such step sets are sometimes called \emm singular, in the framework of quadrant walks.
  \item 
   Finally, if all steps of $\cS$ lie (weakly) below the second diagonal, then a walk confined to~$\Cc$ moves for a while along the second diagonal, and then either stops there or leaves it into the NW or SE quadrant using 
    a South, South-West, or West step. 	It cannot leave the chosen quadrant anymore and behaves therein like a half-plane walk.  
	By polishing this observation, one can prove that $C(x,y)$ is algebraic (while $Q(x,y)=1$).
  \end{itemize}
  By combining these arguments, one finds that there are 
  $74$ essentially distinct models of walks avoiding the negative quadrant that are worth studying: the $79$ models considered for quadrant walks (see~\cite[Tables~1--4]{bomi10}) except the $5$ ``singular'' models for which all steps of~$\cS$ lie weakly above the diagonal $i+j=0$.

%=======================================================
\subsection{The group of the model}
\label{sec:group}
%=======================================================

One important tool in the systematic approach to quadrant walks
%  -- and, as we shall see, to walks confined to $\Cc$ as well -- 
is a certain group $G$ of birational transformations associated with the step set $\cS$. It was introduced in~\cite{bomi10}, and is an algebraic variant of a group introduced much earlier in the study of \emm random, walks in the quadrant~\cite{fayolle-livre,flatto-hahn,malyshev}.
 
 We assume from now on that $\cS$ contains positive and negative steps in the horizontal and vertical directions (otherwise the problem degenerates, as explained above).  We  define  two bi-rational transformations $\phi$ and $\psi$, acting
on pairs  $(u,v)$ of coordinates (which will be, typically, rational functions of $x$ and $y$):
\[
  \phi: (u,v)\mapsto \left(\bu\, \frac {V_{-}(v)}{V_{+}(v)}, v\right)
  \qquad \hbox{and} \qquad
  \psi: (u,v)\mapsto \left(u, \bv\, \frac {H_{-}(u)}{H_{+}(u)}\right),
\]
where $H_-$, $H_+$, $V_-$, and $V_+$ are defined by~\eqref{eq:steppolynomial}. Each transformation fixes one coordinate, and transforms the other \emph{so as to leave the step polynomial $S(u,v)$,  defined by~\eqref{eq:steppolynomial}, unchanged.}
Note that  $\phi$ and $\psi$ are both involutions. The group $G$ is
the group generated by these two transformations. It is isomorphic to
a dihedral group of order $2n$, with $n\in\ns\cup
\{\infty\}$. The \emm length, of $g\in G$, denoted $\ell(g)$,
  is the smallest $\ell$ such that $g$ can be written as a product of
  $\ell$ generators $\phi$ and $\psi$. The \emm sign, of $g\in G$,
denoted $\vareps_g$, is defined by $\vareps_g = (-1)^{\ell(g)}$.
Note that for any $g \in G$, we have $S(g(x,y))=S(x,y)$.

Among the $74$ interesting models identified in the previous subsection, exactly $23$ have a finite group; see~\cite{bomi10}. For the remaining $51$ models, an asymptotic argument implies that the series $C_{0,0}$ that counts  walks ending at $(0,0)$ is not D-finite, which implies that $C(x,y;t)$ is not D-finite; see~\cite{Mustapha2019Walks}.  Among the $23$ models with a finite group,
\begin{itemize}
\item $16$ have a vertical symmetry (say) and a group of order $4$,
\item $5$ have a group of order $6$, and
\item $2$ have a group of order $8$.
\end{itemize}
These models and groups are listed in~\cite[Tables~1--3]{bomi10}. Another classification of these groups distinguishes the $7+4$ models with a \emm monomial, group (meaning that for every $g\in G$, the pair $g(x,y)$ consists of two Laurent monomials in $x$ and $y$), shown in Tables~\ref{tab:weyl} and~\ref{tab:zero} of this paper, from the $12$ non-monomial ones (Table~\ref{tab:non-monomial}).
 
\begin{table}[htb]
  \centering
  \begin{tabular}{c|c|c|c|c|c|c}
 $\cS$ &  \includegraphics[height=0.8cm]{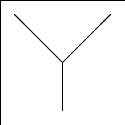} \includegraphics[height=0.8cm]{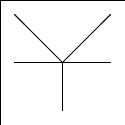}  &\includegraphics[height=0.8cm]{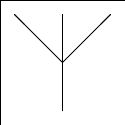} \includegraphics[height=0.8cm]{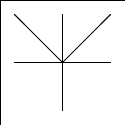} & \includegraphics[height=0.8cm]{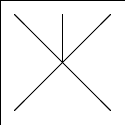} \includegraphics[height=0.8cm]{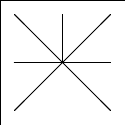}
    & \includegraphics[height=0.8cm]{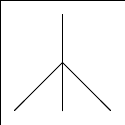} \includegraphics[height=0.8cm]{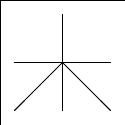}&\includegraphics[height=0.8cm]{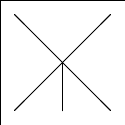} \includegraphics[height=0.8cm]{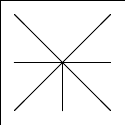}& \includegraphics[height=0.8cm]{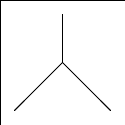} \includegraphics[height=0.8cm]{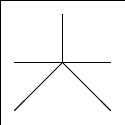}\\
    \hline
$G$ &   \begin{minipage}{20mm}
  $(x,y),(\bar x,y)$,\\ $(\bar x, \bar y \frac{1}{x+\bar x})$,\\ $(x,  \bar y \frac{1}{x+\bar x}) $        
\end{minipage}&
                \begin{minipage}{20mm}
$(x,y),(\bar x,y)$,\\
$(\bar x,\bar y\,\frac {1}{x+1+\bar x})$,\\
$(x,\bar y\,\frac {1}{x+1+\bar x}) $
\end{minipage}&
                \begin{minipage}{20mm}
$(x,y),(\bar x, y)$,\\  
$(\bar x, \bar y\, \frac {x+\bar x}{x+1+\bar x})$,\\ 
$(x, \bar y\, \frac {x+\bar x}{x+1+\bar x})$
\end{minipage}&
                \begin{minipage}{28mm}
$(x,y),(\bar x,y)$,\\
$(\bar x,\bar y\,(x+1+\bar x))$,\\
$(x,\bar y\,(x+1+\bar x))$\end{minipage}&
\begin{minipage}{20mm}$(x,y), (\bar x,y)$,\\ 
$(\bar x,\bar y\, \frac{x+1+\bar x}{x+\bar x})$,\\
$(x,\bar y\, \frac{x+1+\bar x}{x+\bar x})$\end{minipage}&
                                                          \begin{minipage}{22mm}$(x,y),(\bar{x},y)$,\\
$(\bar x,\bar{y}(x+\bar{x}))$,\\
$(x,\bar{y}(x+\bar{x}))$
\end{minipage}                   
  \end{tabular}
  \vskip 4mm
  \caption{The $12$ models with a finite non-monomial group.}
  \label{tab:non-monomial}
\end{table}

%%%%%%%%%%%%%%%%%%%%%%%%%%%%%%%%%%%%%%%%%%%%%%%%%ù
\section{A conjectural  form of \texorpdfstring{$C(x,y)$}{C(x,y)} for Weyl models}
\label{sec:OS}
%%%%%%%%%%%%%%%%%%%%%%%%%%%%%%%%%%%%%%%%%%%%%%%%%%%%%%

It is shown in~\cite{bomi10} that  $19$ of   the $23$ quadrant problems with a finite group can be solved in a uniform manner by considering the \emm orbit sum, of $xy$,
where for any series $F(x,y)$ in $\qs(x,y)[[t]]$, we define the \emm orbit sum, of $F$ by
  \[
  \OS (F(x,y)):=  \sum_{g\in G}\varepsilon_g\, g(F(x,y)),
  \]
  with  $g(F(x,y)):=F(g(x,y))$. To explain the relevance of  orbit sums
  observe that the quadrant equation~\eqref{eqfunc-gen-qu}, once multiplied by $xy$, reads 
\[ %  \beq\label{RS}
K(x,y) xy Q(x,y)=    xy-R(x)-S(y),
\] %  \eeq
  where the series $R(x)$ does \emm not, involve $y$ and the series $S(y)$  does  not involve $x$. Recall moreover that for every $g \in G$ we have $K(g(x,y))=1-tS(g(x,y))=1-tS(x,y)=K(x,y)$. By linearity this gives
  \begin{align} 
    K(x,y)   \sum_{g\in G} \varepsilon_g \, g(xyQ(x,y))&=  \OS (K(x,y)xyQ(x,y)) \nonumber\\
                                                &= \OS(xy) -\sum _{g\in G} \varepsilon_g \,g(R(x))-\sum _{g\in G} \varepsilon_g \,g(S(y)) \nonumber\\
    & = \OS(xy), \label{OS-Q}
  \end{align}
  since $\vareps_{g \circ \psi}=\vareps_{g \circ \phi}=- \vareps_g$ while  $g(R(x)) = (g \circ \psi)(R(x))$ and $g(S(y)) = (g \circ \phi)(S(y))$.
  Analogously, for walks confined to $\Cc$, the form of the functional equation~\eqref{eqfunc-gen} implies that:
 \beq\label{OS-C}
    K(x,y)   \sum_{g\in G} \varepsilon_g\,  g(xyC(x,y))=  \OS (xy).
    \eeq

    \begin{remark}
      If we decide to allow steps between $(-1,0)$ and $(0,-1)$ in the walks that we count, thus adding the terms~\eqref{ajout} to the right-hand side of the functional equation~\eqref{eqfunc-gen}, the orbit sum of $xy C(x,y)$ is still $\OS(xy)/K(x,y)$.
    \end{remark}
    
%================================================
    \subsection{Vanishing orbit sums and algebraicity}
%================================================
    \begin{table}[b]
      \centering
      \begin{tabular}{l|c|c|}
 model       &     \begin{tabular}{ccc}
              \includegraphics[height=0.8cm]{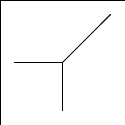}&\includegraphics[height=0.8cm]{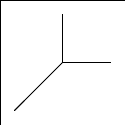} &\includegraphics[height=0.8cm]{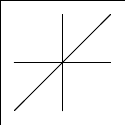}\\
              \small{Kreweras}& \small{reverse Kreweras}& \small{double-Kreweras}
                   \end{tabular}&
                                  \begin{tabular}{c}
                                    \includegraphics[height=0.8cm]{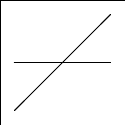}\\ \small{Gessel}
                                    \end{tabular}
        \\
        \hline
        group &\begin{minipage}{3cm}
$(x,y), (\bar x \bar y,y)$,\\ 
$(\bar x \bar y, x)$, $(y,x)$,\\
$(y,  \bar x \bar y),   (x,\bar x\bar y)$
 \end{minipage}                                                       &
    \begin{minipage}{3cm}
\mbox{}\\$(x, y), (\bar x \bar y, y),$\\ $(\bar x \bar y, x^2y), (\bar x, x^2y)$,\\
$ (\bar x, \bar y), (xy, \bar y),$\\ $(xy, {\bar x}^2 \bar y), (x, \bar y {\bar x}^2)$\\
\end{minipage}      \\
        new steps &    \includegraphics[height=0.9cm]{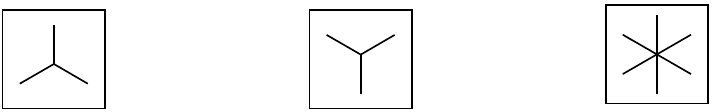} &      \includegraphics[height=0.8cm]{10101010}\\
   \begin{minipage}{2cm}
 \vskip -36mm    quadrant\\ walks
    \end{minipage}
 & \includegraphics[scale=0.7]{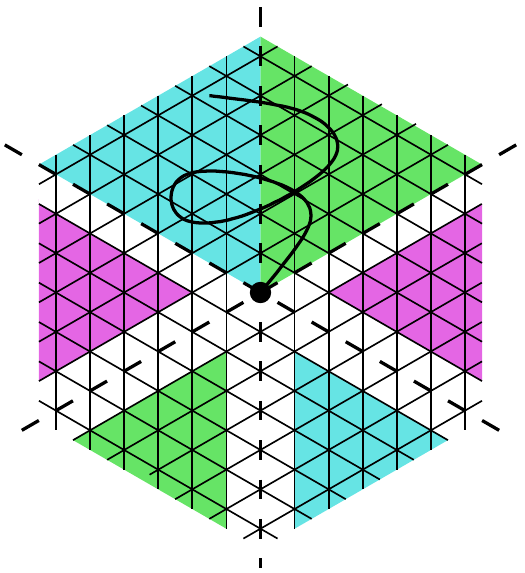} &  \includegraphics[scale=0.7]{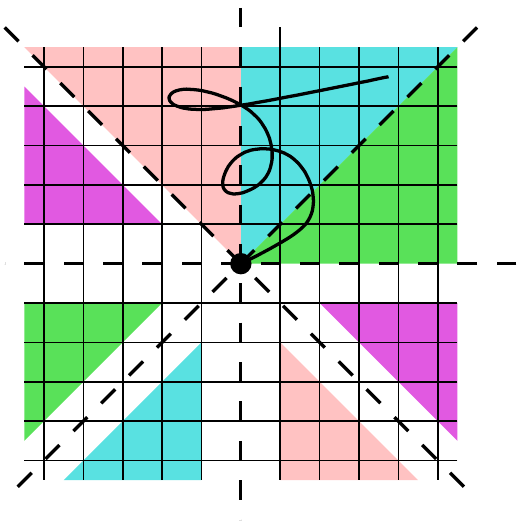}
      \end{tabular}
      \vskip 3mm
      \caption{The four models for which $\OS(xy)=0$, with their names and groups. After deformation, the first quadrant becomes a union of two/three
        % (translated)
        Weyl chambers. The three-quadrant plane corresponds to the complement of this region.}
       \label{tab:zero}
    \end{table}

    As was first observed in~\cite{bomi10}, the orbit sum $\OS(xy)$ is zero for exactly four of the $23$ models with a finite group. The four models are shown in Table~\ref{tab:zero}. For each of them, one  has:
\[ %\beq\label{OS-0-Q}
      \sum_{g\in G} \varepsilon_g \, g(xyQ(x,y))=0=\sum_{g\in G} \varepsilon_g \, g(xyC(x,y)).
\] %\eeq
That is, the orbit sums of $xyQ(x,y)$  and $xyC(x,y)$ vanish. 
It is known that these four models are precisely those for which $Q(x,y)$ is algebraic.
One can derive the algebraicity from the fact that $\OS(xy)=0$~\cite{BeBMRa-16,BeBMRa-17,BostanKurkovaRaschel2016}. These derivations strongly suggest that, more generally, for any finite group model and any point $(a,b) \in \Qc$ such that $\OS(x^{a+1} y^{b+1})=0$, the \gf\ for walks in $\Qc$ starting from $(a,b)$ is algebraic. This is proved in some cases beyond the case $(a,b)=(0,0)$; see the discussion in~\cite[Sec.~7.2]{BeBMRa-17}. Note that the corresponding \gf\ $\tQ(x,y)$ is defined by
\[
  K(x,y)\tQ(x,y)= x^a y^b -t\by H_-(x) \tQ(x,0) -t\bx V_-(y)\tQ(0,y) +t\bx\by
   \tQ(0,0) \mathbbm 1_{(-1,-1)\in \cS}.
 \]
 By linearity, it is thus expected that for any polynomial $I(x,y) \in \qs[x,y]$ satisfying $\OS(xyI)=0$, the series $\tQ(x,y) \in \qs[x,y][[t]]$ defined by
 \[
  K(x,y)\tQ(x,y)= I(x,y) -t\by H_-(x) \tQ(x,0) -t\bx V_-(y)\tQ(0,y) +t\bx\by
   \tQ(0,0) \mathbbm 1_{(-1,-1)\in \cS}
 \]
 is algebraic.

Could a similar algebraicity phenomenon hold for three-quadrant problems? That is, could it be that, for any finite group model, and any Laurent polynomial $I(x,y) \in \qs[x,\bx,y,\by]$ having its support in $\Cc$ and satisfying $\OS(xyI)=0$, the series $\tC(x,y)$ defined by
 \beq\label{eq-tC}
  K(x,y)\tC(x,y)= I(x,y) -t\by H_-(x) \tC_{-,0}(\bx) -t\bx V_-(y)\tC_{0,-}(\by) -t\bx\by
   \tC_{0,0} \mathbbm 1_{(-1,-1)\in \cS}
 \eeq
 is algebraic? (Again, if $I(x,y)$ is reduced to a single monomial $x^ay^b$, then $\tC$ counts walks in~$\Cc$ starting from $(a,b)$.) Several results and guesses support this belief:
 \begin{enumerate}
 \item  It was conjectured in\cite{Bousquet2016} that for the four models of Table~\ref{tab:zero}, for which $\OS(xy)=0$,  the series $C(x,y)$ that counts walks in $\Cc$ starting from $(0,0)$ is algebraic. This was mostly based  on a guessed polynomial equation satisfied by $C_{0,0}$. The algebraicity of $C(x,y)$ (and $C_{0,0}$) is now proved for the first three of these models~\cite{mbm-tq-kreweras}, with explicit algebraic expressions. Moreover, Theorem~23 in~\cite{Budd2020Winding}, taken with $\alpha=0$, $\beta_-=-\pi/2$, and $\beta_+=3\pi/4$, proves that $C_{0,0}$ is algebraic as well for the fourth model.
   
 \item For the simple square lattice model  $\cS=\{\rightarrow,\uparrow, \leftarrow, \downarrow\}$, the orbit of $(x,y)$ under the action of $G$ is $\{(x,y), (\bx, y), (x, \by),(\bx,\by)\}$. The orbit sum $\OS(xy)$ is non-zero, and the series $C(x,y)$ is not algebraic. But it was proved that for walks starting at $(-1,0)$, that is, for $I(x,y)=\bx$, the \gf\ of walks confined to $\Cc$ \emm is, algebraic~\cite[Thm.~6]{Bousquet2016}, and one observes that $\OS(xyI)=\OS(y)=0$. Moreover, the \gf\ of walks in $\Cc$ starting from $(-1,b)$ and ending at $(0,0)$ is also algebraic~\cite[Cor.~2]{Bousquet2016},
   and for $I(x,y)= \bx y^b$ it also holds that $\OS(xyI)=0$.
   
 \item\label{item3} Still for the square lattice model, the heart of the derivation of $C(x,y)$ in\cite{Bousquet2016} is to prove that the series $\tC(x,y)$ defined by~\eqref{eq-tC} with $I(x,y)= (2+\bx^2+\by^2)/3$ is algebraic. Observe that $\OS(xyI)=0$.
   
 \item Similar results hold for the diagonal model $\cS=\{\nearrow, \nwarrow, \swarrow, \searrow \}$, which has the same group as the square lattice model. More precisely, for $I(x,y)=(2+\bx^2+\by^2)/3$ the series $\tC(x,y)$ is algebraic~\cite[Thm.~4]{Bousquet2016}, and moreover walks confined to $\Cc$ starting at $(-1,b)$ and ending at $(0,0)$ have an algebraic \gf~\cite[Cor.~5]{Bousquet2016}.
   
   \item Finally, it is conjectured by Raschel and
     Trotignon~\cite[p.~9]{RaschelTrotignon2018Avoiding} that for any
     finite group model, walks in $\Cc$ starting from $(-1,b)$ (or
     $(b,-1)$)  have an algebraic \gf. This series
     satisfies~\eqref{eq-tC} with again $I(x,y)= \bx y^b$, and thus $\OS(xyI)=0$.
    \end{enumerate}
   We could add to this list more algebraicity results for walks with
   a fixed endpoint confined to certain cones;
   see~\cite[Thm.~23]{Budd2020Winding}
   and~\cite[Cor.~4]{ElveyPrice2020Winding}. In these two papers
   steps between $(-1,0)$ and $(0,-1)$ are allowed (when they belong
   to $\Cc$), and thus the series under consideration obey slightly
   different functional equations. We also refer
     to~\cite{buchacher-kauers-trotignon} for more exotic series with
      zero orbit sums  that are \emm not, algebraic.

%================================================================
\subsection{Vanishing orbit sums for three-quadrant walks}
% ================================================================
 As mentioned in item~\eqref{item3} above, the heart of the derivation of $C(x,y)$ for the simple square lattice is the solution of~\eqref{eq-tC} with $I(x,y)= (2+\bx^2+\by^2)/3$.  Recall the associated series $xy\tC$ has orbit sum zero. What is then the connection between $\tC(x,y)$ and the series $C(x,y)$ we are interested~in?

 It is not hard to construct, for any finite group model,  a series $\tC(x,y)$ related to $C(x,y)$ such that  $xy\tC$  has orbit sum zero. In sight of the functional equations~\eqref{OS-Q} and~\eqref{OS-C}, an obvious choice is $\tC(x,y):=C(x,y)-Q(x,y)$. However, we would also like $\tC(x,y)$ to be characterized by a functional equation resembling~\eqref{eq-tC},
 in which every unknown series  is explicitly described as a sub-series of $\tC(x,y)$. But this is not the case for the above choice of $\tC$.  Indeed, if for instance we take $\cS=\{\rightarrow, \uparrow, \leftarrow, \downarrow\}$, we have
\beq\label{eqQ-square}
  K(x,y) Q(x,y) = 1-t\by Q(x,0) -t\bx Q(0,y) ,
\eeq
\beq\label{eqC-square}
  K(x,y) C(x,y) = 1-t\by C_{-,0}(\bx) -t\bx C_{0,-}(\by) ,
\eeq
and, if we take $\tC(x,y):=C(x,y)-Q(x,y)$,  we obtain, by extracting terms of the forms $x^iy^0$ and $x^0y^i$ with $i<0$:
\[
  C_{-,0}(\bx)=  \tC_{-,0}(\bx),  \qquad C_{0,-}(\by) =\tC_{0,-}(\by),
\]
so that
\[
  K(x,y) \tC(x,y)= -t\by \tC_{-,0}(\bx) -t\bx \tC_{0,-}(\by) +t\by Q(x,0) +t\bx Q(0,y) ,
\]
which does not look very encouraging. However, we can get more leeway as follows.  Observe that for any model $\cS$ with a  finite group $G$, any $h \in G$, and any series $F\in \qs(x,y)[[t]]$, we have $\OS (h(F ))=\varepsilon_h \OS  ( F)$. Therefore, for any $h\in G$, the series $ \varepsilon_h h(xyQ(x,y))$ has the same orbit sum as $xy Q(x,y)$ and $xy C(x,y)$. Consequently, for any collection of real numbers  $\lambda_h, h \in G$ such that $\sum_h \lambda_h=1$, the series 
\beq\label{A-def-gen}
        \tC(x,y) := C(x,y) - \bx\by \sum _{h \in G} \varepsilon_h \lambda_h h( xyQ(x,y))
\eeq
is such that $xy \tC$ has vanishing orbit sum. Can we choose the $\lambda_h$ such that $\tC(x,y)$ is defined by an equation not involving $Q$? Returning to the square lattice case,  if we choose
\[
  \tC(x,y):=C(x,y) - \frac{\bx\by}3\big( xy Q(x,y)- \bx y Q(\bx,y) -x\by Q(x,\by)\big),
\]
we obtain
\[
  C_{-,0}(\bx)=  \tC_{-,0}(\bx) - \frac {\bx^2} 3\, Q(\bx,0), \quad
    C_{0,-}(\by)=  \tC_{0,-}(\by) - \frac {\by^2} 3\, Q(0,\by), \quad
  \]
  and the combination of~\eqref{eqQ-square} (written for $(x,y)$, $(\bx,y)$, and $(x,\by))$ and~\eqref{eqC-square} results in the following simple equation:
  \[
    K(x,y) \tC(x,y)= \frac{2+\bx^2+\by^2}3 -t\by \tC_{-,0}(\bx) -t\bx \tC_{0,-}(\by),
  \]
  which involves no specialization of $Q$. The following proposition tells us that a similar choice exists for any of the Weyl models of Table~\ref{tab:weyl}. In fact, it is easy to check that this choice is always unique.

\begin{prop}\label{prop:A-def-gen}
  Let $\cS$ be one of the Weyl models shown in Table~\ref{tab:weyl}. Let $2d$ be the order of the associated group $G$. Then $d=2, 3$ or $4$. Let $\omega:=\phi\psi\phi \cdots$ (with $d$ generators) be the only element of length $\ell(\omega)=d$ in $G$.
Define
\begin{align}
      A(x,y) &= C(x,y) - \frac{\bx \by} {2d-1} \sum_{h\in G\setminus\{\omega\}} \varepsilon_h\, h(xyQ(x,y)) \label{Aexpr-1}\\
      &= C(x,y) -\frac{\bx \by} {2d-1} \left( \frac{\OS(xy)}{K(x,y)} - \varepsilon_\omega\,  \omega(xyQ(x,y))\right).\label{Aexpr-2}
\end{align}
Then $xyA(x,y)$ has orbit sum zero, and is characterized by the following equation:
\begin{align}
  \label{A-func-eq}
  \begin{aligned}
    K(x,y) A(x,y)&= 1 -\frac {\bx\by}{2d-1} \left(
  \OS(xy)-(-1)^d\bx\by\right)\\
& \quad -t\by H_-(x) A_{-,0}(\bx) -t\bx V_-(y)A_{0,-}(\by) -t\bx\by A_{0,0} \mathbbm 1_{(-1,-1)\in \cS}.
\end{aligned}
\end{align}
\end{prop}

\begin{proof}
  First, the equivalence between the two expressions of $A(x,y)$ comes from~\eqref{OS-Q}. Then, the orbit sum of $xyA$ vanishes because, as noticed above, $xyC$ and each $\vareps_h \, h(xyQ(x,y))$ have the same orbit sum.

  Now we want to write an equation for $A(x,y)$, using the defining equations of $C$ and $Q$. Let us first express $C_{-,0}(\bx)$ in terms of $A$ and $Q$, by extracting  from~\eqref{Aexpr-1} terms of the form $x^iy^0$, with $i<0$. By examination of the three possible groups, detailed in Table~\ref{tab:weyl}, we see  that only one element $h$ contributes, namely $\omega^-=\phi\psi\cdots$ with $d-1$ generators. More explicitly,
  \beq\label{Cm0}
    C_{-,0}(\bx)=A_{-,0}(\bx)+\frac{ (-1)^{d-1}}{2d-1}
    \begin{cases}
      \bx^d Q(\bx,0) & \text{if } d=2 \text{ or } 4,\\
      \bx^d Q(0,\bx) & \text{if } d=3.
    \end{cases}
  \eeq
  Analogously, when we extract from~\eqref{Aexpr-1} terms of the form $x^0 y^j$, with $j<0$, the only group element that contributes is $\omega^+=\phi\psi\cdots $ with $(d+1)$ generators, and
  \beq\label{C0m}
    C_{0,-}(\by)= A_{0,-}(\by)+\frac {(-1)^{d+1}}{2d-1}
    \begin{cases}
      \by^m Q(0,\by) & \text{if } d=2 \text{ or } 4,\\
      \by^3 Q(\by,0) & \text{if } d=3,
%       \by^3 Q(0,\by) & \text{if } d=4.
    \end{cases}
  \eeq
where $m=2$ for $d=2$ and $m=3$ for $d=4$.  Finally, the only element $h$ that contributes to the coefficient of $x^0y^0$ in $C(x,y)$ is the identity, and
 \beq\label{C00}
    C_{0,0}= A_{0,0} + \frac 1 {2d-1} Q(0,0).
\eeq
  
  We now start from~\eqref{Aexpr-2} to write an equation defining $A(x,y)$.
   By examining again the three possible groups we see that   $\omega(x,y)=(\bu,\bv)$ with
  \beq\label{uv}
    (u,v)=
    \begin{cases}
      (x, y) & \text{if } d=2 \text{ or } 4,\\
      (y,x) &\text{if } d=3.
    \end{cases}
  \eeq
 Let us finally denote $\delta= \mathbbm 1_{(-1,-1)\in \cS}$. Then
  \begin{align*}
    K(x,y)   A(x,y) &=   K(x,y)C(x,y) - \frac{\bx \by}{2d-1}\left( {\OS(xy)} -(-1)^d  K(x,y) \omega(xy Q(x,y))\right) \\
    &= 1 -t\by H_-(x) C_{-,0}(\bx) -t\bx V_-(y)C_{0,-}(\by) -t\delta \bx\by
      C_{0,0}
      - \frac{\bx \by}{2d-1} {\OS(xy)}\\
                    & \quad + \frac{(-1)^d \bx^2 \by ^2}{2d-1} \Big(
                       1 -tv H_-(\bu) Q(\bu,0) -tu V_-(\bv)Q(0,\bv) +t\delta uv
 Q(0,0)\Big).
      \end{align*}
      Here, we have used~\eqref{eqfunc-gen} to express $K(x,y)C(x,y)$, and~\eqref{eqfunc-gen-qu} to express $K(x,y) \omega(Q(x,y))=K(\bu, \bv) Q(\bu, \bv)$, with $(u,v)$ given by~\eqref{uv}.

      We now express $C_{-,0}(\bx)$, $C_{0,-}(\by)$, and $C_{0,0}$ thanks to~\eqref{Cm0}--\eqref{C00},
      and examine separately the cases $d=2,4$ and $d=3$.

      If $d=2$ or $4$, we have
      \begin{align*}
        K(x,y)   A(x,y) &=   1 -t\by H_-(x) A_{-,0}(\bx) -t\bx V_-(y)A_{0,-}(\by) -t\delta \bx\by      A_{0,0}   - \frac{\bx \by ({\OS(xy)-\bx\by})}{2d-1} \\
                        & \quad - t\delta \frac{\bx \by}{2d-1} \big( Q(0,0)-\bx\by x y Q(0,0) \big) \\
                        & \quad + \frac {t\by}{2d-1} Q(\bx, 0)\left( \bx^d H_-(x)-\bx^2 H_-(\bx)\right)
        \\
        & \quad + \frac{t\bx}{2d-1} Q(0, \by) \Big( \by^m V_-(y)-\by^2 V_-(\by)\Big),
      \end{align*}
      and the announced equation follows by observing that for each of the 5 models under consideration (shown in the first and third columns of Table~\ref{tab:weyl}),
      \[
        \bx^{d-2} H_-(x)= H_-(\bx)        \quad \text{and} \quad \by^{m-2} V_-(y)= V_-(\by).
      \]

      For the two models such that $d=3$ (central column in Table~\ref{tab:weyl}), we have $\delta=0$, and
      \begin{align*}
        K(x,y)   A(x,y) &=   1 -t\by H_-(x) A_{-,0}(\bx) -t\bx V_-(y)A_{0,-}(\by)   - \frac{\bx \by({\OS(xy)+ \bx\by})}{2d-1} \\
                        &\ \quad - \frac {t\by}{2d-1} Q(0,\bx)\left( \bx^d H_-(x)-\bx^2 V_-(\bx)\right)
        \\
        &\ \quad - \frac{t\bx}{2d-1} Q(\by,0) \Big( \by^3 V_-(y)-\by^2 H_-(\by)\Big),
      \end{align*}
      and now the announced equation follows from the fact that
      \[
        \bx H_-(x)=V_-(\bx). \qedhere
      \]
\end{proof}

\begin{remark}
   Let us explain why we only consider the seven Weyl models in Proposition~\ref{prop:A-def-gen}.
\begin{itemize}
\item For the four models of Table~\ref{tab:zero}, we have seen that $\OS(xy)=0$. Hence $xy C$ has orbit sum zero, and we can simply take $A(x,y)=C(x,y)$.
\item The remaining $12$ models (Table~\ref{tab:non-monomial}) have a vertical symmetry, and a group of order~$4$. The orbit of $(x,y)$ consists of the pairs $(x,y),(\bx,y)=\phi(x,y)$, and two pairs $(x, \by r(x))=\psi(x,y)$ and $(\bx,\by r(x))=\phi\psi(x,y)$ where $r(x)=r(\bx)$ is a rational function in $x$, \emm not reduced to a monomial,.  Let us consider a series $\tC(x,y)$ of the form~\eqref{A-def-gen}. If at least one of the coefficients $\lambda_{\psi}$ or $\lambda_{\phi\psi}$ is non-zero, then it is not clear how to express $C_{0,-}(\by)$ in terms of $Q$ and $A$, as we did in~\eqref{C0m}. If these two coefficients are taken to be $0$, then
  \[
    \tC(x,y)=C(x,y)-\lambda Q(x,y) +(1-\lambda) \bx^2 Q(\bx,y),
  \]
  and the reader can check that the equation for $\tC$ still involves specializations of $Q$, for any choice of $\lambda$.
\end{itemize}
\end{remark}

Let us thus return to the seven Weyl models. The reason for the construction of $A(x,y)$ (and of the notation change $\tC \rightarrow A$) is the following, partially proved, conjecture.

\begin{conjecture}\label{conj:Weyl}
  For any of the seven Weyl models of Table~\ref{tab:weyl}, the series $A(x,y)$ defined in Proposition~\ref{prop:A-def-gen} is algebraic. In particular, $C(x,y)$ is D-finite.

  For any of the four models of Table~\ref{tab:zero}, the series $C(x,y)$ is algebraic (as $Q(x,y)$ itself).
\end{conjecture}

Let us recall that the conjecture is proved for the first two models of Table~\ref{tab:weyl} in\cite{Bousquet2016}, and in this paper for the third (king steps). The second part of the conjecture is proved for the first three models of Table~\ref{tab:zero} in\cite{mbm-tq-kreweras} (D-finiteness was established earlier in\cite{RaschelTrotignon2018Avoiding}).
This leaves us with five models for which the conjecture is open: four Weyl models, and the (conjecturally algebraic) Gessel model.  Based on the solved cases, we believe that algebraicity should hold in a strong sense, and in particular, that each series $A_{i,j}$ (for Weyl models) or $C_{i,j}$ (for the models of Table~\ref{tab:zero}) should be algebraic. For the four remaining Weyl models, we have tried to guess (using the \texttt{gfun} package~\cite{gfun} in {\sc Maple}) a polynomial equation for the series $A_{-1,0}$, which coincides with the \gf\ $C_{-1,0}$ of walks ending at $(-1,0)$.
We could not guess anything for the diabolo model (using the counting sequence up to length $n=4000$), but we discovered equations of degree $24$ for each of the next three.
The degree of $C_{-1,0}$ is $4$ (resp.\ $8$, $24$) for the three solved Weyl models. For Gessel's model, $C_{0,0}$ was conjectured to be algebraic of degree $24$   in~\cite{Bousquet2016}, and algebraicity was proved since then in~\cite[Thm.~23]{Budd2020Winding} (taken with $\alpha=0$, $\beta_-=-\pi/2$, and $\beta_+=3\pi/4$).

%%%%%%%%%%%%%%%%%%%%%%%%%%%%%%%%%%%%%%%%%%%%%%%%%%%%
\section{The king walks: statement of the results}
\label{sec:KingResults}
%%%%%%%%%%%%%%%%%%%%%%%%%%%%%%%%%%%%%%%%%%%%%%%%%%%%
We now fix the step set to be $\cS=\{-1,0,1\}^2\setminus\{(0,0)\}$. We still denote by $Q(x,y)$ the \gf\ of walks confined to the first quadrant, and by $C(x,y)$ the \gf\ of walks avoiding the negative quadrant. The orbit of $(x,y)$ under the action of $G$ is $\{(x,y), (\bx, y), (x, \by), (\bx, \by)\}$. Recall from~\cite{bomi10} that $Q(x,y)$ can be expressed in terms of a simple rational function:
  \[
  xyQ(x,y)= [x^> y^>] \frac{\OS(xy)}{K(x,y)}
  = [x^> y^>] \frac{(x-\bx)(y-\by)}{1 - t( x + xy + y + \bx y + \bx + \bx \by + \by + x \by )}.
\]
From Proposition~\ref{prop:A-def-gen} we get an expression of $C(x,y)$ of the form
  \[
    C(x,y)= A(x,y) + \frac{1}{3}\big( Q(x,y) -  \bx^2\,Q\!\left(\bx,y\right) - \by^2\,Q\!\left(x,\by\right) \big),
  \]
 where,  as announced,  $A(x,y)$ is algebraic. More precisely, we write 
  \beq\label{A-PM}
    A(x,y)=P(x,y)+\bx M(\bx, y)+\by M(\by, x),
  \eeq
  where $P(x,y)$ and $M(x,y)$ belong to $\qs[x,y][[t]]$, and prove that $P$ and $M$ are algebraic.

\begin{theo}  [{\bf The GF of king walks}]
  \label{theo:king}
  The \gf\ of  king walks starting from $(0,0)$,
confined to $\Cc$, and ending in the first quadrant (resp.~at a negative
abscissa) is
\beq\label{sol-king}
  \frac 1  3 Q(x,y) + P(x,y), \qquad \left(\hbox{resp.} -\frac 1 3
    \bx^2 Q(\bx,y) + \bx M(\bx,y)\right),
\eeq
where $P(x,y)$ and $M(x,y)$ are algebraic of degree $216$ over $\qs(x,y,t)$. 
The generating function of walks ending at a negative ordinate follows using the $x/y$-symmetry of the step set.

The series $P$ can be expressed in terms of $M$ by:
 \beq\label{Psol-king}
   P(x,y)=\bx \big( M(x,y)-M(0,y)\big) +\by \big(
     M(y,x)-M(0,x)\big),
 \eeq
and $M$ is defined by the following equation:
  \begin{align}
	\begin{aligned}
	K(x,y) &\left(2M(x,y)-M(0,y)\right) = 
		\frac{2x}{3} 
		-2t \by(x+1+\bx)M(x,0)
		+t \by(y+1+\by)M(y,0) \\
		&\qquad
		+t(x-\bx)(y+1+\by)M(0,y)
		-t\left(1 + \by^2 - 2\bx\by \right)%
		M_{0,0}
			-t\by 
		M_{1,0},
	\end{aligned}
	\label{eqMMM-king}
\end{align}
where $K(x,y) = 1 - t( x + xy + y + \bx y + \bx + \bx \by + \by + x \by )$.
The specializations $M(x,0)$ and $M(0,y)$ are algebraic 
of degree $72$ over $\qs(x,t)$ and $\qs(y,t)$, respectively, and $M_{0,0}$ and
$M_{1,0}$ have degree $24$ over $\qs(t)$. 
\end{theo}

We give in the following two subsections
a complete algebraic description of all the series
needed to reconstruct $P(x,y)$ and $M(x,y)$ from~\eqref{Psol-king}
and~\eqref{eqMMM-king}, namely, the univariate series $M_{0,0}$ and
$M_{1,0}$ (Section~\ref{sec:univariate}) 
%(of degree 24),
and the bivariate series $M(x,0)$ and $M(0,y)$ (Section~\ref{sec:bivariate}).
% (of degree 72).

A combinatorial proof of~\eqref{Psol-king} is given in Section~\ref{sec:combi}, together with a generalization to other starting points and other Weyl models.

%==============================================================
\subsection{Univariate series}
\label{sec:univariate}
%=============================================================
We define in three steps an extension of $\Q(t)$ of degree $24$, schematized by 
\[
  \Q(t) \stackrel{4}{\hookrightarrow} \Q(t,\xzero) \stackrel{3}{\hookrightarrow} \Q(t,\av) \stackrel{2}{\hookrightarrow} \Q(t,\aw),
\]
where $\xzero, \av, \aw \in \qs[[t]]$ and the numbers %above the arrows 
give the degrees of the successive extensions.
First, let $\xzero=t+t^2+\LandauO(t^3)$ be the only power series in $t$ satisfying the quartic equation
\beq\label{u-def}
  (1-3\xzero)^3 (1+\xzero)t^2+(1 + 18\xzero^2
  -27\xzero^4)t-\xzero=0.
  \eeq
  Equivalently,
  \beq\label{u-def-alt}
  \frac {\xzero}{(1+\xzero)(1-3\xzero)^3}= \frac {t(1+t)}{1-8t}.
  \eeq
%  with $\tilde t=t/(1-2t)$:  \[    3\,{\tilde t}^{\,2}+\tilde t ={\frac      {\xzero}{1-4\,\xzero+18\,{\xzero}^{2}-27\,{\xzero}^{4}}}.  \]
 Second, let $\av = t+3t^2+\LandauO(t^3)$ be the only series with constant term zero satisfying the cubic equation
\beq\label{v-def}
  (1+3\av-\av^3)\xzero -\av(\av^2+\av+1)=0.
  \eeq
 Clearly, it holds that $\xzero\in \qs(v)$, and hence we have $\qs(t,\xzero,\av)=\qs(t,\av)$. The minimal equation of $\av$ over $\Q(t)$ can be written as follows:
    \beq\label{algv}
   {\frac {\av \left( {\av}^{2}+\av+1 \right)  \left( {\av}^{3}-3\,\av-1 \right) ^{
3}}{ \left( {\av}^{2}+4\,\av+1 \right)  \left( 4\,{\av}^{3}+3\,{\av}^{2}-1
 \right) ^{3}}}= \frac {t(1+t)}{1-8t}.
\eeq
   Third, define
\beq\label{w-def}
  \aw = \sqrt{1 + 4\av - 4\av^3 - 4\av^4} = 1+2t+4t^2+\LandauO(t^3).
\eeq
One can check that  $\aw$ has degree $24$ over $\qs(t)$. Hence the extension $\qs(t,\aw)$ contains $v$.

\smallskip
We can now make the series $M_{0,0}$ and $M_{1,0}$ occurring in~\eqref{eqMMM-king} explicit. Note that    by~\eqref{sol-king}, the series $M_{0,0}$ coincides with the series $C_{-1,0}$ that counts 
  walks in $\Cc$ ending at   $(-1,0)$. It is 
  \beq\label{Cm10-expr}
M_{0,0} =C_{-1,0}= \frac{1}{2t} \left( \frac{\aw(1+2\av)}{1+4\av-2\av^3}-1\right)
= t+2t^2+17t^3+80t^4+536t^5+\LandauO(t^6).
\eeq
Analogously, we have
\[
C_{0,0}= \frac 1 3 Q_{0,0} +P_{0,0} \qquad \text{and} \qquad  C_{-2,0}= -\frac 1 3 Q_{0,0} +M_{1,0},
\]
where $ P_{0,0} = 2M_{1,0}$ (by~\eqref{Psol-king}) and
\begin{align}\label{Mm10-expr}
  M_{1,0}&= 
	\frac{1}{6t^2}\left( 1+2t + \frac{(1-2t)(1+2\av)(16\av^6+24\av^5+7\av^4-24\av^3-30\av^2-10\av-1)}{\aw(\av^4+8\av^3+6\av^2+2\av+1)(1+4\av-2\av^3)} \right). %\\
%	&= \frac{1}{3}+3t^2+8t^3+\frac{212}{3}t^4+\frac{1084}{3}t^5+\frac{7610}{3}t^6+15878t^7+O(t^8).
\end{align}

More generally, we have the following counterpart of~\cite[Cor.~2 and
Cor.~5]{Bousquet2016}. 

\begin{prop}[\bf Walks ending at a prescribed position]
  \label{cor:coeffs}
  Let $\aw$ be the above defined series in $t$. 
 For $j\ge 0$, the series $C_{-1,j}$ belongs  to $ \qs(t,\aw)$, and
 is thus algebraic.
More generally, for $i\ge 1$ and $j\ge 0$, the series $C_{-i,j}$ is
D-finite of the form
 $$
- \frac 1 3 Q_{i-2,j} +  \Rat(t,\aw)
$$
for some rational function $\Rat$. It is transcendental as soon as $i\ge 2$.

For $i\ge 0$ and $j\ge 0$,  the series $C_{i,j}$  is D-finite and transcendental of the form
 $$
 \frac 1 3 Q_{i,j} +  \Rat(t,\aw).
$$
\end{prop}

Another series of interest is $C(1,1)$, which counts all walks in $\Cc$, regardless of their endpoint. It reads
  \[
    C(1,1)=A(1,1) - \frac 1 3 Q(1,1),
  \]
  and we prove that $A(1,1)$ has degree $24$ over $\Q(t)$. A rational
  expression for $A(1,1)$  in terms of $v$ and $w$ is given in
  Proposition~\ref{prop:univariate_series}. However, $Q(1, 1)$ is
  transcendental~\cite{Bostan-etal-2017-qp} hence $C(1, 1)$ is transcendental too.

We also obtain detailed asymptotic results, which refine general results of Denisov--Wachtel~\cite{DenisovWachtel15} and Mustapha~\cite{Mustapha2019Walks}
(who only obtained estimates up to a multiplicative constant).

\begin{coro}
\label{coro:asy}
The number  $c(n)$ of $n$-step king walks confined to  $\Cc$ and
ending anywhere, and
the number  $c_{0,0}(n)$  of such walks in $\Cc$ ending at the origin satisfy for $n \to \infty$:
\begin{align*} 
	c(n) &= \left(\frac{2^{32} K}{3^7}\right)^{\! \! 1/6} \frac{1}{\Gamma(2/3)} \frac{8^n}{n^{1/3}} 
	- \frac{8}{9 \pi} \frac{8^n}{n} + \LandauO\left(\frac{8^n}{n^{4/3}}\right)
               ,
  \\
  	c_{0,0}(n) &= \left(\frac{2^{29} K}{3^7} \right)^{\! \! 1/3} \,\frac{\Gamma(2/3)}{\pi} \frac{8^n}{n^{5/3}} 
	- \left( \frac{2^{62} L}{3^{31}} \right)^{1/6} \frac{1}{\Gamma(2/3) n^{7/3}} + \LandauO\left(\frac{8^n}{n^{8/3}}\right)
	,
\end{align*}
where 
$K$ and $L$ are the unique real roots of 
\[101^6 K^3 - 601275603 \, K^2 + 92811 \, K - 1\]
%
% 601275603 = 3*200425201
% 92811 = 3*30937
and 
\begin{align*}
	101^{18} L^3 &- 342130847546623941461342020714770 \, L^2 
	\\&
	+ 25258724190403343220341683641 \, L - 5078^6.
\end{align*}
\end{coro}

%=====================================================
\subsection{Bivariate series}
\label{sec:bivariate}
%=====================================================

It remains to describe the series $M(x,0)$ and $M(0,x)$ involved in~\eqref{eqMMM-king}. Both are cubic over $\Q(t,w,x)$, and we express them explicitly in terms of a parametrizing series $\Pun$ that satisfies a reasonably compact cubic equation over $\Q(t,v,x)$. Details are given in Proposition~\ref{prop:RhatShat}. We also refer to Figure~\ref{fig:algSR} on page~\pageref{fig:algSR} for the structure of all series involved in the paper.
 
In Table~\ref{tab:degrees},  we  compare the degrees of several relevant algebraic series in the king's model and the simple and diagonal models solved in~\cite{Bousquet2016}\footnote{The details on the series $M(1,0)$, $M(0,1)$, $M(1,1)$, and $A(1,1)$ are not stated in~\cite{Bousquet2016}, but they can be found in the \Maple\ sessions accompanying this paper on the author's webpage.}. This 
gives a hint of the technical difficulties that arise in the solution of the king's model.

\renewcommand{\arraystretch}{1.5}
\begin{table}[hb]
  \centering
  \begin{tabular}{c|ccccccc}
    \toprule
    Series &    $M(x,y)$ & $M(x,0)$ & $M(0,y)$ & $M(1,1)$& $M(0,1)$ &$M_{0,0}$&
                                                                                         \begin{minipage}{23mm}
$ M(1,0),M_{1,0}$, \\ $A(1,1)$, $A_{0,0}$                                                           \end{minipage}
    \\
    \hline
    Simple/Diag. & 72 & 24 & 12 & 16 & 8  & 4/-- &8\\ \hline
    King & 216 & 72 & 72 & 48 & 48 &  24& 24 \\
    \bottomrule
  \end{tabular}
  \vskip 3mm
  \caption{A comparison of the degrees of various algebraic series for the simple (or diagonal) model~\cite{Bousquet2016} and for the king's model (this paper). The series $M_{0,0}$ is zero for the diagonal model.}
  \label{tab:degrees}
\end{table}
\renewcommand{\arraystretch}{1.0}

%%%%%%%%%%%%%%%%%%%%%%%%%%%%%%%%%%%%%%%%%%%%%%%%%%%%
\section{The king walks: an equation with only one catalytic variable}
\label{sec:KingEquation}
%%%%%%%%%%%%%%%%%%%%%%%%%%%%%%%%%%%%%%%%%%%%%%%%%%%%

Our starting point is the functional equation~\eqref{eqfunc-gen}, specialized to 
 \begin{align*}
   \spol(x,y) &= (x +1+\bx)(y+1+\by)-1 \\
   &=x + xy + y + \bx y + \bx + \bx \by + \by + x \by.
\end{align*}
We use the $x/y$-symmetry of $S(x,y)$, which induces a
  bijection between walks ending on the negative $x$- and $y$-axis,
  and implies  that
\[
  C_{-,0}(\bx)= C_{0,-}(\bx)=: C_-(\bx).
\]
This gives
\begin{align}
\label{eq:kernelC1}
	K(x,y) C(x,y) &= 1 - t \by(x+1+\bx) C_{-}(\bx) - t \bx(y+1+\by) C_{-}(\by) - t \bx \by C_{0,0},
\end{align}
where as usual,  the kernel is
$
	K(x,y) = 1 - t\spol(x,y).
$
 Multiplying by $xy$ gives
\begin{align*}
	xy K(x,y) C(x,y) &= xy - t (x^2+x+1) C_{-}(\bx) - t (y^2+y+1) C_{-}(\by) - t C_{0,0}.
\end{align*}
As observed before, the \gf\ $Q(x,y)$ of quadrant walks satisfies similarly:
\[ %\beq\label{eq:kernelQ1} 
  xy K(x,y) Q(x,y) = xy - t (x^2+x+1) Q(x,0) - t (y^2+y+1) Q(0,y)+ t Q_{0,0}.
\] %\eeq 

The subsequent solution follows the same steps as for the simple walk and the diagonal walk in~\cite{Bousquet2016}. 
But in practise, the king model turns out to be much heavier, and raises %very
serious computational difficulties.
In what follows, we
focus on the points of the derivation that differ
from~\cite{Bousquet2016}. We have performed all computations with the
computer algebra system \Maple. The corresponding sessions are available on the authors' webpages.

%===================================================================
\subsection{A series \texorpdfstring{$A(x,y)$}{A(x,y)} with orbit sum zero}
%===================================================================

As discussed in Section~\ref{sec:OS}, and summarized in Proposition~\ref{prop:A-def-gen}, it makes sense to introduce a new series $A(x,y)$ defined by
\beq\label{AC-def-king}
  C(x,y)= A(x,y) + \frac 1 3 \left(  Q(x,y) - \bx^2 Q(\bx,y) - \by^2 Q(x,\by)\right).
\eeq
Note that any monomial $x^iy^jt^n$ that occurs in $A(x,y)$ is such that $(i,j)\in \Cc$. Then $xy A(x,y)$ has orbit sum zero, meaning that
\beq\label{eq:orbitA}
	xy A(x,y) - \bx y A(\bx,y) + \bx \by A(\bx,\by) - x \by A(x,\by) = 0.
\eeq
Moreover, $A(x,y)$ is defined by the functional equation~\eqref{A-func-eq}, which reads:
\begin{align}
\label{eq:kernelA1}
	K(x,y) A(x,y) &= 
		\frac{2+\bx^2+\by^2}{3} 
		-t\by(x+1+\bx)A_{-}(\bx)
		-t\bx(y+1+\by)A_{-}(\by) 
		-t\bx \by A_{0,0}.
\end{align}
We now focus on the determination of $A(x,y)$, which should be algebraic according to our final Theorem~\ref{theo:king}.  The next step is to split the series $A(x,y)$ into three parts, which involve  polynomials in $x$ and $y$ instead of Laurent polynomials.

%======================================================================
\subsection{Reduction to a quadrant-like problem for \texorpdfstring{$M(x,y)$}{M(x,y)}}
\label{sec:red_quadrant:king}
%===================================================================

We now separate in $A(x,y)$ the contributions of the three quadrants, again using the $x/y$-symmetry of the step set: 
\begin{align}
\label{eq:Asep}
	A(x,y) = P(x,y) + \bx M(\bx,y) + \by M(\by,x),
\end{align}
where $P(x,y)$ and $M(x,y)$ lie in $\Q[x,y][[t]]$.
Note that this identity defines $P$ and $M$ uniquely in terms of $A$. 
Replacing $A$ by this expression, and extracting the positive part in $x$ and $y$ from the orbit equation~\eqref{eq:orbitA} relates the series $P$ and $M$ by
\begin{align}\label{PM-king}
	xyP(x,y) &= y \left(M(x,y)-M(0,y)\right) + x \left(M(y,x) - M(0,x)\right),
\end{align}
which is Equation~\eqref{Psol-king} in Theorem~\ref{theo:king}, and also  the same as \cite[Eq.~(22)]{Bousquet2016}.
For a combinatorial proof of this equation see Section~\ref{sec:combi}.

We could now follow the lines of proof of~\cite[Sec.~2.3]{Bousquet2016} to obtain the functional equation~\eqref{eqMMM-king} for $M(x,y)$.
However, we prefer to describe a slightly different -- and more combinatorial -- way to derive this equation.
Clearly, $A(x,y)$ counts walks confined to $\Cc$, starting either from $(0,0)$, $(-2,0)$, or $(0,-2)$, with a weight $2/3$ in the first case and $1/3$ in each of the other two cases. In sight of  the splitting~\eqref{eq:Asep} of $A(x,y)$, the series $P(x,y)$ counts such walks ending in the first quadrant, and $\bx M(\bx,y)$ those ending at a negative abscissa. By combining these two observations and constructing these walks step by step, we can write directly a pair of equations for $P$ and $M$:
\begin{align*}
  K(x,y)P(x,y) &= \frac 2 3 -t\by(x+1+\bx) P(x,0) -t\bx(y+1+\by)P(0,y)+t\bx \by P_{0,0}\\
               &\quad+ t(x+1+\bx) M(0,x)
                 -t\bx M_{0,0}+ t(y+1+\by)M(0,y)-t \by M_{0,0},
\end{align*}
\begin{align}
\label{eqM-king}
\begin{aligned}
    K(x,y) \bx M(\bx,y) &= \frac 1 3 \bx^2 -t\by(x+1+\bx) \bx M(\bx,0)
    -t(y+1+\by)M(0,y)+t \by M_{0,0}\\
    &\quad+ t\bx(y+1+\by)P(0,y)-t\bx \by P_{0,0}.
\end{aligned}
\end{align} 
In the first equation for instance, the term $t(y+1+\by)M(0,y)-t \by M_{0,0}$ counts walks that come from the NW quadrant and enter the non-negative quadrant through the $y$-axis. We will in fact ignore the first equation and replace it by  the link~\eqref{PM-king} between $P$ and  $M$. Extracting the coefficient of $x^1$ in~\eqref{PM-king} gives
  \[
    P(0,y)=  M_x(0,y)+ \by\left( M(y,0)-M_{0,0}\right).
  \]
  Extracting now the coefficient of $y^0$ gives
  \[
    P_{0,0}=2 M_x(0,0) =2M_{1,0}.
  \]
  We plug these two identities into~\eqref{eqM-king}: upon replacing $x$ by $\bx$ and then dividing by~$x$, we find:
  \begin{align}
	\label{eqM-prelim-king}
	\begin{aligned}
        K(x,y)  M(x,y) &= \frac{1}{3} x  -t\by(x+1+\bx)  M(x,0)
    -t\bx(y+1+\by)M(0,y)+t \bx\by M_{0,0}\\
    &\quad+ t(y+1+\by)
\left(M_x(0,y)+ \by\left( M(y,0)-M_{0,0}\right)\right)
-2t \by M_{1,0}.
	\end{aligned}
	\end{align}
This is not yet~\eqref{eqMMM-king}, as there is one more series involved here, namely $M_x(0,y)$. However, by extracting the coefficient of $x^0$ in the above equation, we find one more relation:
\[
t\by(y+1+\by) M(y,0)+(ty+t\by-1) M(0,y)+2t(y+1+\by)M_x(0,y)-t\by(\by+2+y)M_{0,0}-3t\by M_{1,0} =0 .
\]
Combined with~\eqref{eqM-prelim-king}, this now gives~\eqref{eqMMM-king}.

%=================================================================
\subsection{Cancelling the kernel: an equation between bivariate series}
%=================================================================

Next we will cancel the kernel $K$. As a polynomial in $y$, 
the kernel admits only one root that is a formal power series in $t$:
\[ 
Y(x) = \frac{1 - t(x+\bx) - \sqrt{ (1-t(x+\bx))^2-4t^2(x+1+\bx)^2 }}{ 2t(x+1+\bx)}
	  = (x+1+\bx)t + \LandauO(t^2).
\] 
Note that $Y(x)=Y(\bx)$.
%We follow the lines of proof of \cite[Sec.~2.4]{Bousquet2016}: 
We specialize~\eqref{eqMMM-king} to the pairs $(x,Y(x)), (\bx,Y(x)), (Y(x),x)$, and $(Y(x), \bx)$ (the left-hand side vanishes for each specialization since $K(x,y)=K(y,x)$, yet this symmetry is not part of the group of the model), and we eliminate $M(0,Y)$, $M(Y,0)$, and $M(\bx,0)$  from the four resulting equations. We obtain:
\begin{align}
\label{eq:3Ms}
\begin{aligned}
	(x+1+\bx)\left(Y(x)-\frac{1}{Y(x)}\right) & \left( xM(0,x) - 2\bx M(0,\bx) \right) 
	+ 3(x+1+\bx)M(x,0) 
	\\ & \qquad
	-\frac{2\bx Y(x)}{t} + 3M_{1,0}
	+ (2Y(x) -x - \bx) M_{0,0} = 0.
\end{aligned}
\end{align}
We have now eliminated the trivariate series $M(x,y)$. We are left with three bivariate series, namely $M(0,x)$, $M(0, \bx)$, and $M(x,0)$. 
In the next section we eliminate  the term $M(x,0)$, so as to end with two specializations of the series $M(0,x)$.

%==========================================================================
\subsection{An equation between \texorpdfstring{${M(0,x)}$}{M(0,x)} and \texorpdfstring{${M(0,\bx)}$}{M(0,1/x)}}
%==========================================================================
Let us denote the discriminant occurring in $Y(x)$ by
\beq\label{Delta-def}
  \Delta(x) := (1-t(x+\bx))^2-4t^2(x+1+\bx)^2
  = (1-t(3(x+\bx)+2))(1+t(x+\bx+2)),
\eeq
and  introduce the notation
\begin{align} 
\label{eq:RSdef}
\begin{aligned}
	R(x) &:= t^2 M(x,0) = \frac{xt^2}{3} + \left(1+\frac{x^2}{3}\right)t^3 + \LandauO(t^4),  \\
	S(x) &:= txM(0,x) = x(1+x)t^2 + 2x(1+x+x^2)t^3 + \LandauO(t^4).
\end{aligned}
\end{align}
Note that $t^2M_{0,0}=R_0=tS_1$ and $t^2M_{1,0}=R_1$. Then~\eqref{eq:3Ms} reads
\begin{multline}\label{eq:3Ms2}
  	\sqrt{\Delta(x)} \left( S(x) - 2S(\bx) + \frac{xR_0-t}{t(1+x+x^2)} \right) = 
\\		3(x+1+\bx)R(x) + 3R_1  
+ \frac{1 - t(x+\bx)}{t(1+x+x^2)}(xR_0-t)  - (x+\bx)R_0
=: \Rh(x),
\end{multline}
where we defined $\Rh(x)$ as a shorthand for the right-hand side.
Observe that introducing 
\begin{align}
  \label{eq:STcardanosubs}
  \Sh(x): = S(x) - \frac{3R_0/t %S_1
  -2x-\bx}{3(x+1+\bx)},
\end{align} 
% (with $S_1=R_0/t$)
allows us to rewrite the above equation as
\beq\label{Rh-Sh}
  \sqrt{\Delta(x)} \left( \Sh(x) - 2\Sh(\bx) \right) = \Rh(x).
\eeq

Before we go into the details  of the next steps, let us describe their principle.   We consider both sides of~\eqref{eq:3Ms2} as power series in $t$ whose coefficients are Laurent series in $x$. We square Equation~\eqref{eq:3Ms2} and extract the negative part in
$x$, as defined at the beginning of Section~\ref{sec:general}. On the right-hand side, the terms involving 
$R(x)$ (mostly)  disappear as this series involves only non-negative powers of~$x$. On the left-hand side, the terms involving only $S(x)$ mostly disappear as well. There remain terms involving only  $S(\bx)$, as well as the negative part of 
  $\Delta(x)S(x)S(\bx)$. In other words, the result is an expression 
  for the negative part of 
  $\Delta(x)S(x)S(\bx)$ in terms of $S(\bx)$ and univariate series. Using the symmetry of $\Delta(x)$ in~$x$ and~$\bx$, we 
will then express the \emm positive, part of 
$\Delta(x)S(x)S(\bx)$ in terms of $S(x)$ and univariate series. We will thus reconstruct an expression of $\Delta(x)S(x)S(\bx)$ that
does not involve $R(x)$, as in~\cite[Sec.~2.5]{Bousquet2016}.

In order to make the above programme effective, we  need the following lemma, which tells us how to extract the non-negative part of certain series as those that we meet when we square~\eqref{eq:3Ms2}.

\begin{lemma}
\label{lem:pospole}
Let $\zeta=e^{2i\pi/3}$ and $\bar \zeta=e^{-2i\pi/3}$ be the two primitive cubic roots of unity. Let $F(x) \in \C[x]((t))$.
Then,
\[
  [x^{\geq}] \frac{F(\bx)}{1+x+x^2} =\frac 1 {1-\zeta}\, \frac{F(\zeta)}{1-\zeta x}
  +\frac 1 {1-\bzeta}\,\frac{F(\bzeta)}{1-\bzeta x}
\]
and
\begin{multline*}
  [x^{\geq}] \frac{F(\bx)}{(1+x+x^2)^2} = \frac 2 3
  \left(\frac 1 {1-\zeta}\, \frac{F(\zeta)}{1-\zeta x}    +\frac 1 {1-\bzeta}\,\frac{F(\bzeta)}{1-\bzeta x} \right)
  \\  +\frac 1 {(1-\zeta)^2} \left( \frac{\zeta F'(\zeta)} {1-\zeta x}+ \frac{F(\zeta)}{(1-\zeta x)^2}  \right)
   +\frac 1 {(1-\bzeta)^2} \left( \frac{\bzeta F'(\bzeta)} {1-\bzeta x} +  \frac{F(\bzeta)}{(1-\bzeta x)^2} \right).
 \end{multline*}
In fact, the first formula holds for $F(x) \in \bx\C[x]((t))$, and the second for $F(x) \in \bx^3\C[x]((t))$.
\end{lemma}
\begin{proof}
By linearity, it suffices to prove the lemma when $F(x)= x^k$, for $k\ge -1$ in the first part, $k\ge -3$ in the second part. A key ingredient are the following partial fraction expansions:
    \[
      \frac 1 {1+x+x^2} = \frac 1 {(1-\zeta)(1-\zeta x)}
      + \frac 1 {(1-\bzeta)(1-\bzeta x)},
    \]
 \[
      \frac 1 {(1+x+x^2)^2} =\frac 2 3 \left( \frac 1 {(1-\zeta)(1-\zeta x)}
        + \frac 1 {(1-\bzeta)(1-\bzeta x)}\right)
+ \frac 1 {(1-\zeta)^2(1-\zeta x)^2}
      + \frac 1 {(1-\bzeta)^2(1-\bzeta x)^2}   .   
    \]    
    Then we work out each piece separately, first focussing on the case  $k\ge 0$. For instance, 
    \[
      [x^\ge] \frac{\bx^k}{1-\zeta x}= \bx^k \sum_{n\ge k} \zeta^n x^n= \frac{\zeta^k}{1-\zeta x}
    \]
    and
  \[
      [x^\ge] \frac{\bx^k}{(1-\zeta x)^2}= \bx^k \sum_{n\ge k} (n+1)\zeta^n x^n= \zeta^k \sum_{n\ge 0} (k+n+1) \zeta^n x^n= \frac{k\zeta^k}{1-\zeta x} +\frac{\zeta^k}{(1-\zeta x)^2} .
    \]
To complete the proof, we check that the first (resp.\ second) identity of the lemma holds as well if $F(x)=\bx^\ell$ for $\ell=1$ (resp.\ $\ell =1,2,3$).
\end{proof}

By expanding~\eqref{eq:3Ms2} at $x=\zeta$ and $x=\bzeta$, we derive the values of $S(x)$ at these two points, which will be useful in sight of the above lemma:
\begin{align}
	\label{eq:S1}
	&S(\zeta) = S(\bzeta) = - \frac{R_0 + 3 R_1}{1+t} = -t^2 -11t^4 -30t^5 + \LandauO(t^6).
\end{align} 

Now, as already observed, the right-hand side $\Rh(x)$ of~\eqref{eq:3Ms2} is mostly positive in $x$, meaning that the valuation in $x$ of the coefficient of $t^n$ is bounded from below, uniformly in $n$. We now square both sides.  The negative part of (the square of) the right-hand side is easily obtained by an expansion around $x=0$, and found to be
\[
   \left( 2 R_0+t \right) ^{2}{\bx}^{2}+2 \bx \left( 2 R_0+t
 \right)  \left( 2 R_0+6 R_1-1-t \right).
\]
In the square of the left-hand side of~\eqref{eq:3Ms2} some terms are also mostly positive -- in fact all terms that do not involve $S(\bx)$.  Their negative parts can be extracted as above  by an expansion around $x=0$. Some other terms, like $\Delta(x) S(\bx)^2$, are mostly negative, and we subtract their non-negative parts, obtained via an expansion at $x=\infty$ (which is legitimate due to their Laurent polynomial coefficients in $x$). And finally there are two tricky terms:
\[
\Delta(x) S(x) S(\bx) \quad \hbox{ and } \quad   \Delta(x) \frac{S(\bx)(xR_0/t-1)}{1+x+x^2},
\]
which require some care. We leave the first term untouched, since what we want to determine is precisely its negative part.  The numerator of the second term is a series in $t$ with coefficients in $\Q[x,\bx]$. We expand it at infinity, using $S_0=0$ and $S_1=R_0/t$, and obtain, 
\[
  \Delta(x) S(\bx)(xR_0/t-1) = -3 R_0^2 \,x^2+ F(\bx),
\]
for a series $F(x) \in \bx\qs[x]((t))$. Then we divide this by $(1+x+x^2)$. The term $x^2/(1+x+x^2)$ has no negative part, and we apply Lemma~\ref{lem:pospole} to express the negative part of $F(\bx)/(1+x+x^2)$. After having treated all terms, we reach an identity of the form
\[
  [x^<] \big(\Delta(x)S(\bx)S(x) \big)=
  \Delta(x) S(\bx)^2 - \Delta(x) \frac{S(\bx)(xR_0-t)}{t(1+x+x^2)}
  +\frac{\Pol(R_0,R_1,t,x)}{t x^2(1+x+x^2)}
\]
for some polynomial $\Pol$ with rational coefficients. We can now replace $x$ by $\bx$ to obtain an expression of the positive part of 
$\Delta(x)S(\bx)S(x)$ (which is a series in $\qs[x, \bx][[t]]$). We finally denote by $P_0$  the coefficient of $x^0$ in $\Delta(x) S(x)S(\bx)$, 
and obtain an expression of $\Delta(x) S(x)S(\bx)$ in terms of $S(x)$, $S(\bx)$, $P_0$, $R_0$, and $R_1$, which can be written  as:
\begin{align}
\label{eq:2Ms}
\begin{aligned}
	\Delta(x) \left( S(x)^2 + S(\bx)^2 - S(x) S(\bx) + \frac{S(x)(xt-R_0) + S(\bx) ( \bx t - R_0 )}{t(x+1+\bx)} \right) 
	= \\
	(R_0 + 3R_1) \left(
         (2R_0+t) \left( x+\bx+ \frac {1+t}{t(x+1+\bx)} \right)-1-t
        \right) \\
	- (1+4t)(x+\bx)R_0 + (t^2 + tR_0 + R_0^2) (x^2+\bx^2) - P_0.
\end{aligned}
\end{align}
As in~\cite{Bousquet2016} the numerator of the right-hand side as a polynomial in $x$ is not divisible by $\Delta(x)$, nor by any of its factors. 

Observe that~\eqref{eq:2Ms} can also be written in terms of the series $\Sh$ defined by~\eqref{eq:STcardanosubs}, and then takes the following form:
\begin{align}
	\label{eq:2MsSt}
	\Delta(x) \left( \Sh(x)^2 - \Sh(x) \Sh(\bx) + \Sh(\bx)^2 \right) = \frac{{\Pol}(P_0,R_0, R_1, t,x)}{x^4 t^2 (x+1+\bx)^2},
\end{align}	
where ${\Pol}$ is another polynomial with rational coefficients.
This simpler form in  terms of $\Sh$ will guide us in the following final step, in which  we eliminate $S(\bx)$ and obtain an equation in which $S(x)$ is the only bivariate series.

%=========================================================
\subsection{An equation for \texorpdfstring{${M(0,x)}$}{M(0,x)} only}
\label{sec:cat}
% =========================================================

We would like to extract the positive part of~\eqref{eq:2Ms}, but we are stopped by the mixed term $S(x) S(\bx)$. However,
from the structure visible in~\eqref{eq:2MsSt}, we observe that a multiplication by $\Sh(x) + \Sh(\bx)$
% = S(x) + S(\bx) + (x+\bx-2R_0/t)/(x+1+\bx)
eliminates this mixed term, leaving us with the following  cubic
equation in $\Sh$:
\[
  \Delta(x) \left( \Sh(x)^3+ \Sh(\bx)^3 \right) =
  \left( \Sh(x)+ \Sh(\bx)\right)\frac{\Pol(P_0,R_0, R_1, t,x)}{x^4 t^2 (x+1+\bx)^2}.
  \]
We then rewrite this in terms of $S$ rather than $\Sh$, and extract  the non-negative part in $x$, using the same tools  as in the previous subsection.  We refer for full details to  the accompanying \Maple\ worksheet. The terms that are mostly positive or mostly negative in $x$ do not raise any difficulties.  The two tricky terms are those that involve $S(\bx)^2/(1+x+x^2)$ and $S(\bx)/(1+x+x^2)^2$. Their non-negative parts are extracted using Lemma~\ref{lem:pospole}.  When processing the latter term, three additional univariate series occur, namely $S_2$, $S'(\zeta)$, and $S'(\bzeta)$. We find it more convenient to  work with the real and imaginary parts of $\zeta S'(\zeta)$, and to define series $B_1$ and $B_2$ by
\begin{align} \label{eq:B1B2def}
  \begin{aligned}
    (1+t)^2 \zeta S'(\zeta) &= \DSA + i\sqrt{3}\DSB, \\
	(1+t)^2 \bzeta S'(\bzeta) &= \DSA - i\sqrt{3}\DSB.
   \end{aligned}
\end{align}
We use several times $S_1=R_0/t$ and the expressions~\eqref{eq:S1} of $S(\zeta)$ and $S(\bzeta)$. At the end we  obtain a polynomial identity between $S(x)$, $P_0$, $R_0$, $R_1$, $S_2$, $B_1$, $B_2$, $x$, and $t$.

We can reduce to four the number of univariate series involved in this equation as follows. First, we expand the equation around $x=0$ at first order: this gives an expression of $P_0$ in terms of the $5$ other univariate series. We replace $P_0$ by this expression in the functional equation, and now expand at first order around $x=\zeta$: this gives
\[
  3t^2 S_2=- 3tR_0-3t^2-2B_1.
\]
In the end we get a cubic equation in $S(x)$:
\begin{align}
	\label{eq:Pol1}
	\Pol(S(x), R_0, R_1, \DSA, \DSB, t, x) = 0,
\end{align}
for a  polynomial $\Pol(x_0, x_1, x_2, x_3, x_4 ,t, x)$ with rational coefficients.  In the terminology of~\cite{BMJ06}, this is an equation with only one \emm catalytic variable,, namely $x$, as opposed to the original functional equation for $M(x,y)$ that had two catalytic variables, $x$ and $y$.

We can describe  the above polynomial $\Pol$ in a reasonably compact form thanks to some of its properties: first, when we introduce the series $\Sh(x)$ defined by~\eqref{eq:STcardanosubs}, there is no quadratic term (in $\Sh(x)$).
Then, the coefficients of the resulting equation are (almost) symmetric in $x$ and $\bx$, and they become symmetric if we introduce the series $\Sh(x)/(x-\bx)$. Now we can write the equation in terms of a new variable $y:=x+1+\bx$.
Then we observe one more property, namely that the coefficients are (almost) invariant when we replace $y$ by $\by (1+1/t)$. We refer to our \Maple\ worksheet for details.  If we denote
\beq\label{z-def}
z= t(x+1+\bx ) + \frac{1+t}{x+1+\bx }
\eeq
and
%\beq\label{S-tilde-def}
\[
  \St(x)= \frac{x+1+\bx }{x-\bx}\, \Sh(x)= \frac 1{x-\bx} \left( (x+1+\bx ) S(x)- \frac{R_0} t + \frac {2x} 3  + \frac\bx 3\right),
\]
%\eeq
then Equation~\eqref{eq:Pol1} reads:
%%%%%%%%%%%%%%%%%%%%%%%%%%%%%%%%%%%%%%%%%%%%%%%%%%%%%%%%%%%%%%%%%%%%%
%% some commands to shrink the space before and after + and - signs
%% they are only used here
%%%%%%%%%%%%%%%%%%%%%%%%%%%%%%%%%%%%%%%%%%%%%%%%%%%%%%%%%%%%%%%%%%%%%
{
\newcommand{\shrink}[1]{\!#1\!}
\newcommand{\myp}{\shrink{+}}
\newcommand{\mym}{\shrink{-}}
\begin{align}
\label{eqSc}
\begin{aligned}
 0 &= 27 {t}^{2} \left( 2 t\myp z\myp 1 \right)  \left( 10 t\mym 3 z\myp 1 \right) {\tilde S(x)}^{3} \\
 & +\big(  \left( 216 {t}^{2} \mym 27 {z}^{2} \myp 54 t \right) R_0^{2} \myp 27 t \left( 6 R_1 t\mym 6 R_1 z\myp 6 {t}^{2} \myp 2 zt \mym {z}^{2} \myp 2 B_2\myp t\myp z \right) R_0  
 \\
   & \quad - 9 {t}^{2} \left( 27 R_1^{2} \mym 27 R_1 t\myp 9 R_1 z\myp 5 {t}^{2} \mym 2 zt\myp 3 B_1\mym 3 B_2\mym 9 R_1\myp 6 t\mym 2 z\myp 1 \right)  \big) \tilde S(x)\\
   &  +\left( 72 {t}^{2} \mym 9 {z}^{2} \myp 18 B_1\myp 18 t \right) R_0^{2} \myp 9 t \left( 6 R_1 t\mym 6 R_1 z\myp 6 {t}^{2} \myp 2 zt \mym {z}^{2} \myp 2 B_1\myp 2 B_2\myp t\myp z \right) R_0 \\
   & \quad -{t}^{2} \left( 81 R_1^{2} \mym 81 R_1 t\myp 27 R_1 z\mym 5 {t}^{2} \mym 10 zt\myp 3 {z}^{2} \mym 9 B_1\mym 9 B_2\mym 27 R_1\myp 6 t\mym 4 z\myp 2 \right).
\end{aligned}
\end{align}
}

%%%%%%%%%%%%%%%%%%%%%%%%%%%%%%%%%%%%%%%%%
\section{The king walks: algebraicity}
\label{sec:quadratic}
%%%%%%%%%%%%%%%%%%%%%%%%%%%%%%%%%%%%%%%%%%%%%%%%%
In~\cite{BMJ06}, a general method to solve equations in one catalytic variable was developed, proving in particular that their solutions are systematically algebraic (provided the equation is \emm proper, in a certain natural sense). In Section~\ref{sec:system} we first use the results 
of~\cite{BMJ06} to obtain a system of four polynomial equations relating the series  $R_0, R_1, \DSA$, and $\DSB$. Combined with a few initial terms, this system characterizes these four series. 
Unfortunately, it turns out to be too big for us to 
obtain individual equations for each of the four series, be it by bare hand elimination or using Gr\"obner bases: we did obtain  polynomial equations for $R_0$ and  $R_1$, of degree $24$ in each case, but not for the other two series.  
Instead, as detailed in Section~\ref{sec:guess}, we have resorted to a guess-and-check approach,
consisting in \emph{guessing} such equations (of degree $12$ or $24$,
depending on the series), and then \emph{checking}  that they satisfy the
system.

%===================================================
\subsection{A polynomial system relating \texorpdfstring{$R_0$, $R_1$, $B_1$, and $B_2$}{R0, R1, B1, and B2}}
\label{sec:system}
%===================================================

We start from the cubic equation~\eqref{eq:Pol1}. The approach of~\cite{BMJ06} instructs us to consider the  series  $X$ (in $t$, or in a fractional power of $t$)), satisfying
\beq\label{dx0}
\Pol_{x_0}(S(X), R_0, R_1, \DSA, \DSB, t, X)       = 0, 
\eeq
where $\Pol_{x_0}$ stands for the derivative of $\Pol$ with respect to its
first variable.
The number of such series $X$ and their first terms depend only
on the first  terms of the series $S(x)$, $R_0$, $R_1$, $\DSA$, and $\DSB$; see~\cite[Thm.~2]{BMJ06}. We find that $6$ such series exist:
\begin{align*}
	X_1(t) &= i+2t^2+4t^3+(36-2i)t^4+ \LandauO(t^5), \\
  X_2(t) &= %\bar{X}_1(t) =
           -i+2t^2+4t^3+(36+2i)t^4+ \LandauO(t^5), \\
	X_3(t) &= \sqrt{t}+t+\frac{3}{2}t^{3/2} + 3t^2 + \frac{51}{8}t^{5/2} + 14t^3 + \LandauO(t^{7/2}), \\
	X_4(t) &= -\sqrt{t}+t-\frac{3}{2}t^{3/2} + 3t^2 -
                 \frac{51}{8}t^{5/2} + 14t^3 + \LandauO(t^{7/2}), \\
  	X_5(t) &= i \sqrt{t}-it^{3/2}+2i t^{5/2} + t^3 - 4i t^{7/2} + 2t^4 + \LandauO(t^{9/2}),\\
  X_6(t) &=% \bar{X}_4(t) =
           -i \sqrt{t}+it^{3/2}-2i t^{5/2} + t^3 + 4i t^{7/2} + 2t^4 + \LandauO(t^{9/2}).
\end{align*}
Note that the coefficients of $X_1$ and $X_2$ (resp.~$X_5$ and $X_6$)
are conjugates of one another.  As discussed in~\cite{BMJ06}, each of these
series $X$ also satisfies
\beq\label{dx}
\Pol_{x}(S(X), R_0, R_1, \DSA, \DSB, t, X)        = 0, 
\eeq
where $\Pol_x$ is the derivative with respect to the last variable of $\Pol$, and of course
\beq\label{original}
\Pol(S(X), R_0, R_1, \DSA, \DSB, t, X)       = 0.
\eeq
Using this, we can easily identify two of the series $X_i$: indeed,
eliminating $\DSA$ and $\DSB$ between the three equations~\eqref{dx0},~\eqref{dx}, and~\eqref{original} gives a polynomial equation between $S(X)$, $R_0$, $R_1$, $t$, and $X$, which factors. Remarkably, its simplest non-trivial factor only
involves $t$ and $X$, and reads
\vspace{-1mm}
\beq\label{X34}
  X^2-t(1+X)^2(1+X^2).
\eeq
By looking at the first terms of the $X_i$'s and at the other
factors, one concludes that the above equation holds for $X_3$ and $X_4$, which are thus explicit. The other four series $X_i$ satisfy another equation in $S(X)$, $X$, $R_0$, $R_1$, and $t$, which we will not use.

      Let $D(x_1,\ldots,x_4,t,x)$ be the discriminant of $\operatorname{Pol}(x_0,\ldots,x_4,t,x)$ with respect to $x_0$. According to~\cite[Thm.~14]{BMJ06},
      each $X_i$ is a {\emph{double root}} of       $D(R_0,R_1, \DSA,\DSB,t,x)$, seen  as a polynomial in $x$.
Hence this polynomial, which involves four unknown series $R_0, R_1, \DSA,\DSB$, has (at least)
$6$ double roots. This seems  more information than we need.
However, we shall see that there is some redundancy in the $6$
series $X_i$, which comes from the properties of $\Pol$ %, and hence $D$,
that we used at the end of Section~\ref{sec:cat} to write it in a compact form.
 
We first observe that $D$ factors as 
\[
  D(R_0,R_1, \DSA,\DSB,t,x)= 27x^2t^2(1+x+x^2)^2 \Delta(x) 
  D_1(R_0,R_1, \DSA,\DSB,t,x),
\]
where $\Delta(x)$ is defined by~\eqref{Delta-def}, and $D_1$ has
degree $24$ in $x$. It is easily checked that none of the $X_i$'s are
roots of the prefactors, so they are double roots of $D_1$. But
we observe that $\bx^{12}D_1$ is symmetric in $x$ and $\bx$.  That is,
\[
  D_1(R_0,R_1, \DSA,\DSB,t,x)= x^{12}D_2(R_0,R_1,  \DSA,\DSB,t,x+1+\bx), 
\]
for some polynomial $D_2(x_1, \ldots, x_4,t,y)\equiv D_2(y)$  of degree 12 in
$y$. Since each $X_i$ is a double root of $D_1$, each series
$Y_i:=X_i+1+1/X_i$, for $1\le i\le 6$,  is a double root of $D_2$. The series $Y_i$, for $2\le i \le 6$, are easily seen from their first 
terms to be distinct, but the first terms of $Y_1$ and $Y_2$
suspiciously agree: one suspects (and rightly so), that $X_2=1/X_1$,
and carefully concludes that $D_2$ has (at least) $5$ double roots in
$y$. Moreover, since $X_3$ and $X_4$ satisfy~\eqref{X34}, the
corresponding series $Y_3$ and $Y_4$ are the roots of $  1+t=tY_i^2$,
that is, $Y_{3,4}=\pm \sqrt{1+1/t}$. The other roots start as follows:
\[ 
  Y_2= 1+4t^2+8t^3 + \LandauO(t^4),\qquad
  Y_{5,6}= \mp \frac i{\sqrt t} + 1+ t^2 \pm i t^{5/2} + \LandauO(t^3).
\]
 This is not yet the end of the story: indeed, $D_2$ appears to be 
almost symmetric in $y$ and $1/y$.  
More precisely, we observe that
\[
  D_2(R_0,R_1,  \DSA,\DSB,y)= y^6 D_3\left(R_0,R_1,  \DSA,\DSB,
    ty+\frac{t+1}y\right),
\]
for some polynomial $D_3(R_0,R_1,  \DSA,\DSB,t,z)\equiv D_3(z)$ of degree $6$ in $z$. It follows that each series $Z_i:= tY_i+(1+t)/Y_i$, for $2\le i \le 6$, is a root of $D_3(z)$, and even a double root, unless $tY_i^2=1+t$, which precisely occurs for $i=3,4$. One finds $Z_{3,4}= \pm 2\sqrt{t(1+t)}$,
\[ 
  Z_2= 1+ 2t -4 t^2 +\LandauO(t^3),\qquad 
  Z_{5,6}= 2t+2t^3 +\LandauO(t^4).
\]
Since $Z_5$ and $Z_6$ seem indistinguishable,  we   conclude that $D_3(z)$ has (at least) two double roots $Z_2$ and $Z_5$, and a factor $(z^2-4t(1+t))$ coming from the simple roots at $Z_3$ and $Z_4$.
We 

\pagebreak
%to allow a nicer page split on the next pages

\noindent can thus write
\[
  D_3(z)= \sum_{i=0}^6 d_i z^i=  \left({z}^{2}-4\,t(1+t) \right) (\alpha z^2 +\beta z +\gamma)^2,
\]
where the $d_i$ are explicit in terms of $R_0$, $R_1$, $B_1$, and $B_2$. We can determine $\alpha$, $\beta$, and $\gamma$ in terms of the $d_i$ by matching the three monomials of highest degree, and this gives:
\[
  D_3(z)= \sum_{i=0}^6 d_i z^i= {\frac { \left({z}^{2}-4\,t(1+t) \right) \left( 8\,{z}^{2}d_{6}^{2}+4\,zd_{{5}}d_{{
  6}}+16\,{t}^{2}d_{6}^{2}+16\,td_{6}^{2}+4\,d_{{4}}d_{{6}}-d_{5}^{2} \right) ^{2}}{64\,d_{6}^{3}}}.
\]
Extracting from this identity the coefficients of $z^0, \ldots, z^3$ gives four polynomial
relations between  the coefficients~$d_i$, resulting in four polynomial
relations between the four series $R_0$, $R_1$, $\DSA$, and $\DSB$.
We give below the degrees and number of terms in each of them.

\begin{table}[ht]
    \centering
  \begin{tabular}{l@{\hskip 8mm}ccccc|c}
    \toprule
  Degree in &$R_0$& $R_1$& $\DSA$& $\DSB$ &  $t$ &Number of terms\\
\midrule
Eq.~$1$ & $5$ & $3$ & $1$ & $1$ & $7$  & $72$  \\
Eq.~$2$ & $6$ & $4$ & $2$ & $2$ & $7$  & $132$ \\
Eq.~$3$ & $5$ & $5$ & $2$ & $2$ & $9$  & $192$ \\
Eq.~$4$ & $6$ & $6$ & $3$ & $3$ & $10$ & $276$ \\
               \bottomrule
\end{tabular}
\vskip 4mm
\caption{Properties of the four polynomial equations defining the four main  unknown series $R_0$, $R_1$, $B_1$, and $B_2$.}
\label{tab:system}
\end{table}
\vspace{-3mm}

We will now  check that the solution of this system is unique if we add the conditions $R_0=\LandauO(t^3)$, $R_1=\LandauO(t^2)$, $B_1=\LandauO(t^2)$, $B_2=\LandauO(t^2)$, which are directly deduced from the  definitions of $R(x)$, $B_1$, and $B_2$ in~\eqref{eq:RSdef} and~\eqref{eq:B1B2def}. We write accordingly $R_0=t^3 \tilde R_0$, $R_1=t^2 \tilde R_1$, $B_1=t^2 \tilde B_1$, $B_2=t^2 \tilde B_2$ in the system, divide each equation by a power of $t$ so that it becomes non-trivial  at $t=0$ (and, as it happens, linear in each series at this point). We finally form linear combinations of these four equations so that the system, evaluated at $t=0$, is triangular. We refer again to our \Maple \ sessions for details.

As explained at the beginning of this subsection, we have  been able to derive directly from this system polynomial equations (of degree $24$) for  $R_0$ and $R_1$ by successive eliminations, but not for the other two series. At the end we resorted to a guess-and-check approach.

%=============================================
\subsection{Guess-and-check}
\label{sec:guess}
% =============================================

The functional equation~\eqref{eq:kernelA1} defining $A(x,y)$ encodes a simple recurrence for the numbers $a_{i,j}(n)$ that count (weighted) walks of length $n$ by the positions of their endpoints $(i,j) \in \Cc$:
\begin{align*}
	a_{i,j}(n+1) &= a_{i-1,j-1}(n) + a_{i-1,j}(n) + a_{i-1,j+1}(n) + a_{i,j-1}(n) \\ & \quad  + a_{i,j+1}(n)  + a_{i+1,j-1}(n) 
	             + a_{i+1,j}(n) + a_{i+1,j+1}(n) ,
\end{align*} 
with $a_{i,j}(n)=0$ for $(i,j) \not\in \Cc$ and initial conditions $a_{0,0}(0)=2/3$, $a_{-2,0}(0)=a_{0,-2}(0)=1/3$, and $a_{i,j}(0)=0$ otherwise.
We implemented this recurrence in the programming language~$C$ using modular arithmetic and the Chinese remainder theorem to compute
these numbers up to $n=2000$ (this effectively bounds $i$ and $j$ to $2000$ as well, since $a_{i,j}(n)=0$ if $i>n$ or $j>n$).
For this purpose, we used approximately $100$ primes of size $\approx 2^{64}$, and we actually computed $3A(x,y)$ rather than $A(x,y)$, as it has integer coefficients.

The series $R_0$, $R_1$, $\DSA$, and $\DSB$ are related to $A(x,y)$ as follows. First, observe that by~\eqref{eq:RSdef} it holds that $R_0= t^2 A_{-1,0}$ and $R_1=t^2 A_{-2,0}$.
Second, for $\DSA$ and $\DSB$ defined in~\eqref{eq:B1B2def}, we also start from~\eqref{eq:RSdef}, which implies that
$S'(\zeta) = t M(0,\zeta) + t\zeta M_{y}(0,\zeta)$ where $M(0,y) = \sum_{j \geq 0} A_{-1,j} y^j$. 
In order to compute $M(0,\zeta)$ we used $\zeta^2 = -1-\zeta$, with $\zeta=(-1+i\sqrt 3)/2$, which implies that $6M(0,\zeta) = \alpha_1 + i\sqrt{3}\alpha_2$ with $\alpha_1,\alpha_2 \in \Z[[t]]$. Hence, the initial coefficients of the series $\alpha_1$ and $\alpha_2$ may be computed using modular arithmetic. The same holds for $\zeta M_y(0,\zeta)$, which then allows to reconstruct the coefficients of $B_1$ and $B_2$. Then we were able to guess polynomial equations satisfied by 
$R_0$, $R_1$, $\DSA$, and $\DSB$ using the \texttt{gfun} package in {\Maple}~\cite{gfun}. 
We refer for full details to  the accompanying \Maple\ worksheet.

Of course, the equations obtained for $R_0$ and $R_1$ coincide with those that we derived from the system of the previous subsection.
Details on the corresponding equations are shown in Table~\ref{tab:guessed}. We note that the degree $24$ equation for $B_2$ is in fact a degree~$12$ equation for $B_2^2$.

\begin{table}[ht!]
\centering

  \begin{tabular}{cccc}
 \toprule
    Generating function & Degree in $GF$ & Degree in $t$  & Number of terms \\
    \hline	
    $R_0$ &  $24$ & $36$ & $323$ \\
    $R_1$ &  $24$ & $36$ & $623$
       \\
    $\DSA$ &  $12$ & $24$ & $229$ \\
		$\DSB$ &  $24$ & $60$ & $477$ \\
   \bottomrule
  \end{tabular}
  \vskip 4mm
	\caption{Properties of the guessed polynomial equations for the four main unknown series $R_0$, $R_1$, $B_1$, $B_2$.}
 	\label{tab:guessed}
      \end{table}
\vspace{-2mm}

We now have to check that the guessed series satisfy the system obtained in the previous subsection. This turns out to be much easier once the algebraic structure of these series is elucidated. We explain in Appendix~\ref{app:subextensions}  how this can be done. We believe that this can be of interest to readers handling algebraic series of large degree. After this step, one obtains  expressions for $R_0$, $R_1$, $B_1$, and $B_2$ in terms of the series $v$ and $w$ of Section~\ref{sec:univariate}. We have not tried a direct check of the system based on the four guessed equations of Table~\ref{tab:guessed}.

\begin{prop}\label{prop:explicit-univariate}
  Let  $u,v, w \in \qs[[t]]$ be the series defined  in
  Section~\ref{sec:univariate} by~\eqref{u-def},~\eqref{v-def},
  and~\eqref{w-def}, respectively. Then the four series that occur in
  the equation in one catalytic variable defining $S(x)$ are:
  \allowdisplaybreaks
      \newcommand{\incnl}{\\[1.5mm]}
      % increase the distance to next line in the following formula to make it better readable
   \begin{align*}
  R_0 &= \frac{t}{2} \left(  \frac{\aw(1+2\av)}{1+4\av-2\av^3}-1\right), \incnl
    R_1 &= \frac{1}{6}\left( 1+2t + \frac{(1-2t)(1+2\av)(16\av^6+24\av^5+7\av^4-24\av^3-30\av^2-10\av-1)}{\aw(\av^4+8\av^3+6\av^2+2\av+1)(1+4\av-2\av^3)} \right), \incnl
    \DSA &= \frac{3\av^2 (1-8t) (1+4\av+\av^2)(\av^2-1)(1+2\av)}{2(1-3\av^2-4\av^3)^3(1+4\av-2\av^3)},\incnl
 	\DSB &= \frac{(1+2v)(1-2t)}{2w(v^4+8v^3+6v^2+2v+1)^2(2v^3-4v-1)} 
 	\left(4tv^{12}+68tv^{11}+16(22t+1)v^{10} \right. \\
      &\quad \left. +12(67t+2)v^9+5(192t-5)v^8+8(61t-10)v^7 
        -(286t+41)v^6         -2(394t-33)v^5    \right. \\
      &\quad \left.
        -(738t-113)v^4 -4(97t-17)v^3  -(126t-19)v^2 
        -2(12t-1)v -2t \right).    
\end{align*}
\end{prop}
\begin{proof}
 It suffices to check that the four series above satisfy the initial conditions $R_0=\LandauO(t^3)$, $R_1=\LandauO(t^2)$, $B_1=\LandauO(t^2)$, $B_2=\LandauO(t^2)$, and the system of 4 polynomial equations established in Section~\ref{sec:system}, the properties of which are summarized in Table~\ref{tab:system}. The first point is straightforward. Then we take each equation of the system in turn, replace the four unknown series by the above expressions, take the numerator of the resulting equation (which is a polynomial in $t$, $v$, and~$w$), and reduce it first modulo  Equation~\eqref{w-def} defining $w$ over $\qs(v)$. In each case, we note that the remainder does not involve $w$, an encouraging sign. Then we reduce further modulo  Equation~\eqref{algv}  defining $v$ over $\qs(t)$. In each case, we find zero, so that the system holds for the above values of $R_0$, $R_1$, $B_1$, and $B_2$. This completes the proof.
\end{proof}

Note that this proves in particular the announced expressions~\eqref{Cm10-expr} and~\eqref{Mm10-expr}  for the series $M_{0,0}=R_0/t^2$ and $M_{1,0}=R_1/t^2$; see~\eqref{eq:RSdef}. We claim that at this stage, we have proved the algebraicity of the series $P(x,y)$ and $M(x,y)$. Recall that by definition, walks in $\Cc$ ending in the first quadrant (resp.\ at negative abscissa) have \gfs
\[
  \frac 1  3 Q(x,y) + P(x,y), \qquad \left(\hbox{resp.} -\frac 1 3
    \bx^2 Q(\bx,y) + \bx M(\bx,y)
  \right).
\]

\begin{coro}
  The series $P(x,y)$ and $M(x,y)$ are algebraic over $\Q(t,x,y)$.
\end{coro}
\begin{proof}
  We work our way backwards starting from the 4 univariate algebraic series of Proposition~\ref{prop:explicit-univariate}. Since $S(x)=tx M(0,x)$ satisfies a cubic equation 	$\Pol(S(x), R_0, R_1, \DSA, \DSB, t, x) = 0$,
  where the polynomial $\Pol$
  has non-zero leading coefficient in its first variable, $S(x)$ and $M(0,x)$ are algebraic of degree at most $72$. We will see that this bound is tight. It then follows from~\eqref{eq:3Ms2} that $R(x)=t^2M(x,0)$ is algebraic as well. We now return to~\eqref{eqMMM-king}, which expresses $M(x,y)$: since $M_{0,0}=R_0/t^2$ and $M_{1,0}=R_1/t^2$, we conclude that $M(x,y)$ is algebraic. We finally use the relation~\eqref{Psol-king} between $P(x,y)$ and $M(x,y)$ to conclude that $P(x,y)$ is algebraic.   
\end{proof}

In the next subsection, we determine the degree of all algebraic series of interest, and give closed form expressions for $S(x)$ and $R(x)$ in terms of the already defined series $v$ and $w$, and a ``simple'' cubic extension of $\Q(t,v,x)$.

%======================================================
\subsection{Back to \texorpdfstring{${S(x)}$}{S(x)} and \texorpdfstring{${R(x)}$}{R(x)}}
\label{sec:bivariate-sol}
%=====================================================

In this subsection, we prove that $S(x)$ and $R(x)$ belong to the same cubic extension of $\Q(t,w,x)$, and describe this extension in (reasonably) compact terms.
We give  two descriptions of this extension by rational parametrizations (in fact, a third one  hides in Appendix~\ref{app:P0}). Remarkably, they define cubic extensions of $\Q(t,v,x)$ rather than $\Q(t,w,x)$.
The first one is in terms of the variable $\tilde y:= t(x+\bx+1)/(1-2t)$ and involves $v$ but not $t$. The second one, however, involves the original variable $x$, and now $t$ and $v$.

More precisely, let  $\Pun\equiv \Pun(x)$ be the unique series of the form $\Pun=xt^2+\LandauO(t^3)$  satisfying
\beq\label{P1-param}
  \tilde y = \frac\kappa {\Pun} \frac{\num(\Pun)}{\num(r_1/\Pun)},
\eeq
where 
\[
\num(U)=   
U+{v}^{2}w^2 -{\frac {
{v}^{4} w^2 \left( v^2-1 \right)  \left( {v}^{2}+v+1
 \right) }{U}},
\]
with $w^2=1+4v-4v^3-4v^4$ as before, and
\begin{align}
  \kappa&={\frac { \left( {v}^{3}-3\,v-1 \right) ^{2} \left( {v}^{2}+v+1
      \right) {v}^{2}}{{v}^{4}+8\,{v}^{3}+6\,{v}^{2}+2\,v+1}}, \notag \\
  \quad 
    r_1&=-{v}^{3}w^2 \left( {v}^{2}+v+1 \right)  \left( {v}^{3}-3\,v-1 \right). \label{kappa-r1}
\end{align}
Let $\Ptwo\equiv \Ptwo(x)$ be the unique series $\Ptwo=\bx+\LandauO(t)$ that satisfies
\beq\label{xP2}
  x=\Ptwo\, \frac{M(\Ptwo)}{M(1/\Ptwo)},
\eeq
where $M(U)= 1/U+\alpha+\beta U$, with
\[
  \alpha= v -{\frac { \left( {v}^{3}-3\,v-1 \right) \beta}{{v}^{2}+v+1}}
\]
and
\[
   \beta = \frac{\left( {v}^{2}+v+1 \right)
    \left( \left( 2\,{v}^{5}+15\,{v}^{4}+20\,{v}^{3}
+16\,{v}^{2}+6\,v+1 \right) t+ v \left( {v}
^{3}-3\,v-1 \right)\right)
}
{t \left( {v}^{4}+8\,{v}^{3}+6\,{v}^{2}+2\,v+1 \right)  \left( 2\,{v}^{3
}+3\,{v}^{2}+6\,v+1 \right) 
}.
\]
The series $\Pun$ and $\Ptwo$ generate the same cubic extension of $\Q(t,v,x)$. In particular,
\beq\label{P12}
  \frac 1 \Pun = a\left( \Ptwo+ \frac 1 \Ptwo + {\frac {{v}^{2}+4\,v+1}{{v}^{2}+v+1}}\right)
\eeq
with
\[
  a={\frac {\beta\,t \left( {v}^{4}+8\,{v}^{3}+6\,{v}^{2}+2\,v+1 \right) 
}{ {v}^{3}(1-2t) \left( v^2-1 \right)   \left( {v}^{2}+v+1 \right) 
 \left( {v}^{3}-3\,v-1 \right)  }}.
\]
One can also express $\Ptwo$ as an element of $\Q(t,v,x,\Pun)$ by combining~\eqref{xP2} and~\eqref{P12}. Finally, one can check that $\Pun$ and $\Ptwo$ have degree $36=12\times 3$ over $\Q(t,x)$. Therefore, we have $\Q(t,v,x,\Pun)=\Q(t, x, \Pun)=\Q(t,x,\Ptwo)$.

\begin{prop}\label{prop:RhatShat} 
 Let $v$ and $w$ be the series of $\Q[[t]]$ defined by~\eqref{v-def} and~\eqref{w-def}. Let $\Pun(x)$ and $\Ptwo(x)$ be defined above. The series $R(x)=t^2M(x,0)$ and $S(x)=txM(0,x)$ are algebraic of degree $72$ over
  $\qs(x,t)$ and belong to $\qs(t,x,\aw,\Pun)=\qs(t,x,\aw,\Ptwo)$.
	More precisely, the series
\beq\label{St-def}
          \St(x)= \frac 1{x-\bx} \left( (x+1+\bx ) S(x)- \frac{R_0} t + \frac {2x} 3  + \frac\bx 3\right)
\eeq
          and
  \beq\label{eq:Rhat}
    \Rh(x)=	3(x+1+\bx)R(x) + 3R_1  
    + \frac{1 - t\bx(x+\bx)(x+1)^2}{t(x+1+\bx)}(R_0-t\bx) + t(1+\bx^2)
  \eeq
  belong respectively to $\qs(t,x,\Pun (x))$
  and  $\aw\,\qs(t,x,\Pun (x))$. In particular,  
\[
\tilde S(x)+\frac 1 3 =-
\frac{{v}^{2}w^2 (1+2v)\left( {v}^{2}+4\,v+1 \right) ^{2}}{2v^3-4v-1} 
\frac{1 }{D(\Pun) D(r_1/\Pun)}
\]
where $r_1$ is given by~\eqref{kappa-r1} and 
\[
D(U)=(v+1) U + v w^2(v^2-1) + (v-1) \frac{r_1} U.
\]
Recall that $R_0$ and $R_1$ lie in $\Q(t,w)$, and are given by Proposition~\ref{prop:explicit-univariate}.
\end{prop}

\begin{proof}
  We  return to the cubic equation that defines $S(x)$, written in the form~\eqref{eqSc} in terms of~$z$ and $\St(x)$ and we replace  $R_0$, $R_1$, $\DSA$, $\DSB$ by their expressions in terms of $t$, $v$, and $w$. 
Then we observe that only even powers of $w$ occur: hence, using the defining equation~\eqref{w-def} of $w$, we obtain a cubic equation for $\tilde S(x)$ involving only $t$, $v$, and of course the variable $z$ defined by~\eqref{z-def}. This  equation has degree % $1$ in $t$,
$2$ in~$z$.  We lower the degree in $t$ to $1$ using the minimal polynomial~\eqref{algv} of $v$.  Now the coefficient of $z^2$ does not involve $t$, the coefficient of $z^1$ is a multiple of $(1-2t)$, and the coefficient of $z^0$ is a multiple of $(1-8t)$. But observe that the minimal equation of $v$ can also be written as
\[
  \frac{1-8t}{(1-2t)^2}=  {\frac { \left( {v}^{2}+4\,v+1 \right)  \left( 4\,{v}^{3}+3\,{v}^{2}-1
 \right) ^{3}}{ \left( 4\,{v}^{4}+4\,{v}^{3}-4\,v-1 \right)  \left( {v
    }^{4}+8\,{v}^{3}+6\,{v}^{2}+2\,v+1 \right) ^{2}}}.
\]
  This gives a cubic equation for $\tilde S(x)$, with coefficients in $\Q(v,\tz)$ where
\beq\label{zt-def}
\tz= \frac z{1-2t}= \frac{yt}{1-2t} + \frac{1+t}{y(1-2t)},
\eeq
where as before $y=x+1+\bx$. It is remarkable that this equation does \emph{not}  involve $t$. Its genus (in $\tz$ and $\tilde S$) is found to be zero and thus this equation admits a rational parametrization. We give one in Appendix~\ref{app:P0} (see~\eqref{zt-P0}), in terms of a series denoted by $\Pzero(x)$, for which we have
\[
  \tilde S(x)+\frac 1 3 =
-  {\frac {{v}^{2} \left( v^2-1 \right)  \left( 2 v+1 \right)   \left( {v}^{2}+4 v+1 \right) ^{2}}{ \left( 2 {v}^{3}-4 v-
   1 \right)
 %\left( -4 {v}^{8}-8 {v}^{7}+4 {\Pzero}^{2}{v}^{4}-11 {v}^{6}+4 {\Pzero}^{2}{v}^{3}+14 {v}^{4}-4 {\Pzero}^{2}v+8 {v}^{3}-{\Pzero}^{2}+{v}^{2} \right)
\left(w^2{\Pzero}^{2}+{v}^{2}
 \left( v^2-1 \right)  \left( 2 v+1 \right)  \left( 2 {v}^{3}+3 {v}^{2}+6 v+1 \right)\right) 
}}.
\]
But it may be better to parametrize our extensions in terms of $x$ than $\tilde z$. Let us first get back to $y=x+1+\bx$, or rather to    $\tilde y = {yt}/({1-2t})$, and observe that $\tilde z$ can be written as
\beq\label{zt-yt}
  \tilde z= \tilde y + \frac {t(1+t)}{(1-2t)^2}\frac 1 {\tilde y}= \tilde y + \frac q{\tilde y}
\eeq
where
\[
 q= -{\frac {v \left( {v}^{2}+v+1 \right)  \left( {v}^{3}-3\,v-1 \right) ^{3}}
     { %\left( 4\,{v}^{4}+4\,{v}^{3}-4\,v-1 \right)
  w^2  \left( {v}^{4}+8\,{v
}^{3}+6\,{v}^{2}+2\,v+1 \right) ^{2}}}
.
\]
This means that $\tilde S(x)$ also satisfies a cubic equation with coefficients 
in $\Q(\tilde y, v)$, again not involving $t$.
This equation is also found to have genus $0$ (in $\tilde y$ and $\tilde S$) and can be parametrized rationally by introducing the series $\Pun$ defined by~\eqref{P1-param}. Indeed, if, in the equation relating $\tilde y $ and $\tilde S$,  we replace $\tilde y$ by its expression in terms of $\Pun$, the equation factors into a linear term in $\tilde S$, and a quadratic one. Provided we choose the correct determination of $\Pun$, given by $\Pun=xt^2+\LandauO(t^3)$, then the term that vanishes is the linear one, and this gives the expression of $\tilde S$ stated in the proposition.
Observe that replacing $\Pun$ by $r_1/\Pun$ in~\eqref{P1-param}
replaces $\tilde y $ by $q/\tilde y$ (because
$\kappa^2=q r_1$), and thus leaves~$\tilde z $ unchanged; see~\eqref{zt-yt}. Analogously, the series $\Pzero(x)$ that parametrizes the equation in $\zt$ and $\tilde S$ (see Appendix~\ref{app:P0}) is invariant by this transformation, and reads 
\[
  \Pzero=\frac{1-v^2}{w^2}\left(\Pun +v^2w^2 +\frac{r_1}{\Pun}\right).
\]

One can actually go even further, as the equation that relates the original variable $x$ and the series $\Pun$ (now with coefficients in $\Q(t,v)$) also has genus zero. It can be parametrized by introducing the series $\Ptwo$ defined by~\eqref{xP2}. Indeed, if we replace, in the equation relating~$x$ and~$\Pun$, the variable $x$ by its expression in terms of $\Ptwo$, we observe again a factorization, which leads to~\eqref{P12} once the correct determination of $\Ptwo$ is chosen.

One readily checks that $\Pun$ and $\Ptwo$ (and $\Pzero$ as well) have degree $36$ over $\Q(t,x)$.

Thus $\tilde S(x)$ belongs to $\Q(t,x,\Pun)=\Q(t,x,\Ptwo)$, while $S(x)$, which involves $R_0$ and hence~$w$, belongs to $\Q(t,w,x,\Pun)=\Q(t,w,x,\Ptwo)$ and has degree at most $72$. To prove that this bound is tight, one can eliminate $w$ and $v$ in the equation defining $S(x)$. It is enough to do it for $x=2$, for instance, as we find that $S(2)$ has degree $72$.

\medskip

We now wish to determine the series $R(x)=t^2M(x,0)$, which is expressed in terms of $S(x)$ and $S(\bx)$ in~\eqref{Rh-Sh}. Equivalently, 
\[
  \hat R (x)= \frac{(x-\bx)\sqrt{\Delta(x)}}{x+1+\bx } \left( \St(x) +2 \St(\bx)\right).
\]
We could of course eliminate $\St(x)$ and $\St(\bx)$  to determine a polynomial equation satisfied by $\Rh(x)$ over $\qs(t,x,v)$, but there is an algebraic structure in the above equation, which will 
save us these calculations. Let us denote $\Pol(s)=s^3+ps+q$ the monic minimal polynomial of $\St(x)$ over $\Q(\zt,v)$. One of its root is of course $s_1=\St(x)$, another one is $s_2=\St(\bx)$ (because $\zt$ is invariant under $x\mapsto \bx$) and the third one is $s_3=-\St(x)-\St(\bx)$ (because there is no quadratic term in $\Pol$). Hence $  \St(x) +2 \St(\bx)=s_2-s_3$. It is not hard to see that, if we denote by $\delta(\zt)=-4p^3-27q^2$ the discriminant of $\Pol(s)$, and choose its square root so that
\[
  \sqrt{\delta(\zt)} = (s_1-s_2)(s_1-s_3)(s_2-s_3),\]
then
\[
  \sqrt{\delta(\zt)}   (s_2-s_3)=-6p s_1^2 + 9q s_1 -4p^2.
\]
Hence
\beq\label{Rh-expr}
  \hat R (x)=\frac{(x-\bx)}{x+1+\bx  }\sqrt{\frac{\Delta(x)}{\delta(\zt)}}
  \left( 9q\St(x) -6p\St(x)^2-4p^2\right),
\eeq
for some $p, q \in \Q(\zt,v)$. Hence the proof of the proposition will be complete if we prove that ${\Delta(x)}/(w^2{\delta(\zt)})$ is a square in $\Q(t,v,x)$. 
After several reductions, described in our \Maple\ session, we obtain
\beq\label{ratio:delta}
  \sqrt{\frac{\Delta(x)}{\delta(\zt)}}=
  \frac{w\Delta(x)^2 (v^4+8v^3+6v^2+2v+1)^2(2v^3-4v-1)^3(x-\bx)^3}
  {y^2(ty^2-t-1)(v^2+4v+1)^2(v^2-1)(2v+1)(1-2t)^2P(\zt)}
\eeq
where we denote as before $y=x+1+\bx$ and
{
\newcommand{\shrink}[1]{\hspace{-0.1em}#1\hspace{-0.1em}}
\newcommand{\myp}{\shrink{+}}
\newcommand{\mym}{\shrink{-}}
\begin{align*}
    P(\zt) &=
- w^2 {\zt}^{2} \left( {v}
^{4}\myp8 {v}^{3}\myp6 {v}^{2}\myp2 v\myp1 \right) ^{2}\\
&\quad - ( v\mym1 ) 
 \left( 8 {v}^{7}\myp16 {v}^{6}\myp40 {v}^{5}\myp72 {v}^{4}\myp85 {v}^{3}\myp53
 {v}^{2}\myp13 v\myp1 \right)   \left( {v}^{4}\myp8 {v}^{3}\myp6 {v}^
{2}\myp2 v\myp1 \right)\zt\\ 
&\quad + 2 v \left( 2 {v}^{11}\myp2 {v}^{10}\myp12 {v}^{9}\myp18
 {v}^{8}\myp23 {v}^{7}\myp22 {v}^{6}\myp5 {v}^{5}\mym29 {v}^{4}\mym57 {v}^{3}\mym
40 {v}^{2}\mym11 v\mym1 \right) .
\end{align*}
}%
From this point on, we can combine~\eqref{Rh-expr} and~\eqref{ratio:delta} with  the various parametrizations (by $\Pzero$, $\Pun$, or $\Ptwo$) introduced above to write closed-form expressions for $\Rh(x)$. We give one in Appendix~\ref{app:P0} in terms of $\Pzero$; see~\eqref{Rhat-P0}.
The degree of $R(x)$ is clearly $72$ at most. We determine it at $x=2$ by elimination of $\Pzero$, $w$, and $v$, and find it to be $72$; hence the bound is tight.
\end{proof}

\begin{figure}[!ht]
\hskip -25mm   \scalebox{1}{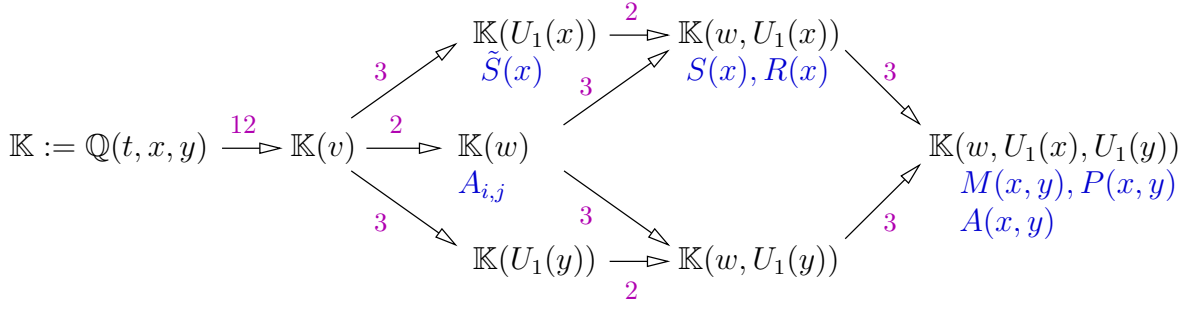}
\caption{Structure of the various fields involved in the solution of king walks in $\Cc$. We have indicated the degrees, and where the main series lie.}
\label{fig:algSR}
\end{figure}

\begin{proof}
  [End of the proof of Theorem~\ref{theo:king}.] 
  We have proved all statements of this theorem, except for the
  degrees of the trivariate \gfs \ $M(x,y)$, $P(x,y)$, and $A(x,y)$. It
  is clear from~\eqref{eqMMM-king},~\eqref{Psol-king}, and~\eqref{A-PM}
  that they belong to   $\KK(w, \Pun(x), \Pun(y))$,
  where $\KK=\Q(t,x,y)$
  and hence that they have degree at most $72\times 3=216$ over
  $\KK$. We check (by specializing $x$, $y$, and even $t$ to real
  values where all series converge, like $x=3$, $y=2$, and $t=1/100$) that there is no unexpected degree reduction.
\end{proof}
We get the final picture of the algebraic extensions shown in Figure~\ref{fig:algSR}.

\begin{comment}

\begin{figure}[!ht]
	\centering
	\scalebox{1.1}{
	\begin{tabular}{ccccc}
		&& $\KK(\av,\aw,\St(x))$ && \\
	    & $\stackrel{3}{\hookNEarrow}$ && $\stackrel{3}{\hookSEarrow}$ &\\
$\KK = \Q(t,x,y) \stackrel{24}{\hookrightarrow}\KK(\av,w)$ &&&& $\KK(\av,\aw,\St(x),\St(y))$\\
		& $\stackrel{3}{\hookSEarrow}$ && $\stackrel{3}{\hookNEarrow}$ &\\
		&& $\KK(\av,\aw,\St(y))$ &&
	\end{tabular}	
	}
	\caption{Define $\KK = \Q(t,x,y)$. We have $R(x) \in \KK(\av,\aw,\St(x))$ and $M(x,y) \in \KK(\av,\aw,\St(x),\St(y))$. Thus, $M(x,y)$ and $A(x,y)$ are of degree $216$ over $\Q(t,x,y)$.}
	\label{fig:algSR}
\end{figure}
\end{comment}

%=======================================================================
\subsection{Some interesting univariate series}
\label{sec:proofcoro}
%=======================================================================
In this subsection we examine various univariate series of interest, like those that are involved in the enumeration of all walks in $\Cc$, or of walks ending on the boundaries of $\Cc$. We also prove the results of Proposition~\ref{cor:coeffs} dealing with walks ending at a specific point, and the asymptotic results of Corollary~\ref{coro:asy}.

\begin{prop}\label{prop:univariate_series}
  The series $R(1)=t^2M(1,0)$ is algebraic of degree $24$  over $\Q(t)$ and belongs to $\Q(t,w)$. More precisely,
  \[
  R(1)+\frac t 3 =-\frac{\numsmall} {3w(2v^3-4v-1)(v^4+8v^3+6v^2+2v+1)(2v^3+3v^2+6v+1)},
\]
where
\begin{align*}
  \numsmall&=v \left( v+1 \right)  \left( 2 {v}^{3}+4 {v}^{2}+5 v+1 \right) 
 \left( 4 {v}^{6}+3 {v}^{5}-8 {v}^{4}-6 {v}^{3}+12 {v}^{2}+11 v+
   2 \right)\\
&\quad +\left( 96 {v}^{10}+272 {v}^{9}+446 {v}^{8}+384 {v}^{7}+3 {v}^{6
}-464 {v}^{5}-553 {v}^{4}-298 {v}^{3}-87 {v}^{2}-14 v-1 \right) t
.
\end{align*}
The series $S(1)=tM(0,1)$ is algebraic of degree $48$ over $\Q(t)$ and belongs to a quadratic extension of $\Q(t,w)$. More precisely,
\beq\label{S-T}
  S(1)+\frac 1 2 =w\T,
\eeq
where $\T=1/2+ \LandauO(t)$ has degree $2$ over $\Q(t,v)$, and satisfies~\eqref{T-def} (in Appendix~\ref{sec:appendixC}).
\\
The series $M(1,1)$ and $P(1,1)$ are algebraic of degree $48$ and belong to  $\Q(t,w,\T)$. \\
The series $A(1,1)$ and $A_{0,0}$  are algebraic of degree $24$ and belong to $\Q(t,w)$. More precisely,
\[
  A(1,1)+\frac 1 {3t}= -\frac{w \times \numsmall'}{3t(1-2t)(4v^3+3v^2-1)^2(2v^3-4v-1)(2v^3+3v^2+6v+1)},\]
with
\begin{align*}
  \numsmall'&=2(4v^3+3v^2+4v+1)(4v^3+3v^2-1)^2t\\
 &\quad +(v+1)(16v^9+72v^8+94v^7+86v^6+3v^5+61v^4+68v^3+24v^2+7v+1)  ,
\end{align*}
while
\[
  A_{0,0}=P_{0,0} = \frac{2 R_1}{t^2},
\]
where $R_1$ is given in Proposition~\ref{prop:explicit-univariate}.
\end{prop}
\begin{proof}
  We begin with the series $S(1)$: we set $x=1$ in the cubic equation~\eqref{eq:Pol1} satisfied by $S(x)$, and observe that the equation factors. The factor that vanishes is quadratic in $S(1)$. (The fact that $S(1)$ is quadratic can also be seen from~\eqref{eq:2Ms}.) Then we replace $R_0, R_1, B_1, B_2$ by their expressions from Proposition~\ref{prop:explicit-univariate}. We then reduce the degree of $t$ and $w$ in this equation by taking remainders (in $t$ and $w$) modulo~\eqref{algv} and~\eqref{w-def}. The coefficient of $w$ in this equation has a factor $(1+2S(1))$, which suggests to write~\eqref{S-T}. Now $\T$ is quadratic over $\Q(t,v)$, but is found not to belong to $\Q(t,w)$. Its minimal equation over $\Q(t,v)$ can be written as~\eqref{T-def}.

  Now in order to determine $R(1)$, we set $x=1$ in the square of~\eqref{eq:3Ms2}, and perform similar reductions as for $S(1)$. For $M(1,1)$, we use the defining equation of $M(x,y)$ (see~\eqref{eqMMM-king}), of course at $x=y=1$, and obtain
  \[
    (1-8t) M(1,1)=\frac 1 3-\frac {R_1+3R(1)}{2t} +(1-8t) \frac{S(1)}{2t},
  \]
  from which the properties stated in the proposition easily follow. We then combine the above expression of $M(1,1)$ with~\eqref{Psol-king} to obtain
  \[
     (1-8t) P(1,1)=\frac 2 3-\frac {R_1+3R(1)}{t} -(1-8t) \frac{S(1)}{t}.
   \]
   Since $A(x,y)$ is given by~\eqref{eq:Asep}, we then find
  \[
     (1-8t) A(1,1)=\frac 4 3-2\frac {R_1+3R(1)}{t}.
   \]
We observe that the series $S(1)$ is not involved in this expression, and therefore $A(1,1)$ has degree $24$ only. Finally, we obtain from~\eqref{Psol-king} that $A_{0,0}=P_{0,0}= 2 M_{1,0}= 2 R_1/t^2$, which thus also has degree $24$.  
 \end{proof}

Let us now prove Proposition~\ref{cor:coeffs}, which deals with walks ending at a specific point.
 
 \begin{proof}[Proof of Proposition~\ref{cor:coeffs}.]
   According to~\eqref{sol-king} and~\eqref{Psol-king}, it suffices to prove that all series $M_{i,j}$ belong to $\Q(t,w)$.

   Let us first prove this when $i=0$ or $j=0$, that is, for the coefficients of the series $S(x)=txM(0,x)$ and $R(x)=t^2M(x,0)$. For $S(x)$, we write $S(x)=xT(x)$, and observe that the cubic equation~\eqref{eq:Pol1} satisfied by $S(x)$, with coefficients in $\qs(t,x,R_0, R_1, B_1, B_2)$, reads
   \[
     3 t (R_0^2+R_0 t+t^2) (T(x)-R_0/t)= x\,  \widetilde \Pol(t,x,T(x),R_0, R_1, B_1, B_2),
   \]
   for some polynomial $\widetilde \Pol$.   This implies that  $T_0=S_1=R_0/t$, as we already know from the definitions of $R(x)$ and $S(x)$, and then, by induction on $i$, that the series $S_i$ belong to $\Q(t,w)$ (because the series $R_0, R_1, B_1, B_2$ do). It then follows that the coefficients of $R(x)$ also belong to this field, using~\eqref{eq:Rhat},~\eqref{Rh-expr}, and~\eqref{ratio:delta}.

   We finally return to the equation~\eqref{eqMMM-king} that defines $M(x,y)$. It reads
   $K(x,y)M(x,y)= F(t,x,y)$, where $F(t,x,y)$ is a Laurent series in $x$ and $y$, having coefficients in $\Q(t,w)$ as we have just proved. We extract the coefficient of $x^i y^j$ in this equation, for $i,j\ge 0$, and thus obtain a linear expression $tM_{i+1, j+1}$ in terms of series $M_{k, \ell}$, where  $k\le i+1$ and $\ell \le j+1$, one equality being strict, and series of $\Q(t,w)$. We then conclude by induction on $i+j$.

   The fact that $C_{i,j}$ is transcendental (except for $i=-1$ or $j=-1$), follows from the fact that $Q_{i,j}$ is transcendental for $i,j\ge 0$, because its coefficients grow like $8^n n^{-3}$, which contradicts algebraicity.
 \end{proof}
 
 We finally prove the asymptotic results of Corollary~\ref{coro:asy}.

 \begin{proof}[Proof of Corollary~\ref{coro:asy}.] We apply here the principles of the \emm singularity analysis, of algebraic series~\cite[Sec.~VII.7]{flajolet-sedgewick}. The series $u$ defined by~\eqref{u-def} is found to have radius of convergence $1/8$, and a unique singularity of minimal modulus, at $t=1/8$.
   Moreover, as $t$ approaches $1/8$ from below, $u$ has the following Puiseux expansion:
   \[
     u=     \frac 1 3-\frac 2 9\,{6}^{1/3}(1-8t)^{1/3}+\frac 1 {27}{6}^{2/3}(1-8t)^{2/3}
   +\frac 1{27}(1-8t) + \LandauO\left((1-8t)^{4/3}\right).
\]
Then the series $v$ defined in~\eqref{v-def}, seen as a series in $u$, has a radius of convergence larger than $u_c:=1/3$, and is thus analytic at $u_c$. At this point it attains the value $v_c\approx 0.455\ldots$,
which is the only real root of $4v^3+3v^2-1$. As $t$ approaches $1/8$ from below, one finds
\[
v=   v_c-\frac 1 3\,v_c \left( 1+2\,v_c \right) {6}^{1/3}(1-8t)^{1/3}+{\frac { \left( 8\,{v_c}^{2}+11\,v_c+2 \right) }{54}}(1-8t) + \LandauO\left((1-8t)^{4/3}\right).
\]
Finally, the series $w$, seen as a series in $v$, is analytic at $v_c$, where it is equal to $w_c:=\sqrt{3v_c^2+12v_c+3}/2$. As $t$ approaches $1/8$ from below, one finds
\[
w= w_c-\frac 2 9{6}^{2/3}  \,v_c\,w_c\left( 1+2\,v_c
 \right)(1-8t) ^{2/3}+ \LandauO\left((1-8t)\right).
\]
More terms of the singular expansions of these three series are available in our \Maple\ session.  We plug these expansions in the expressions of $A(1,1)$ and $A_{0,0}$ given in the previous proposition and obtain
\beq\label{A11sing}
  A(1,1)= -  \frac{2^5 6^{1/3} w_c\left(28 v_c^2+61 v_c-86\right)}{3^3 101 (1-8t)^{2/3}} + cst+ \LandauO\left((1-8t)^{1/3}\right),
\eeq
\begin{multline*}
   A_{0,0}=cst
 %-{\frac {24032\,{\it w_c}\,{\it v_c}}{2727}}+{\frac {4864\,{\it w_c}\,{{         \it v_c}}^{2}}{2727}}-{\frac {13208\,{\it w_c}}{2727}}+{\frac{80}{3}}
 -{\frac {2^9\,{6}^{2/3} { w_c}\, \left( 6716\,{{ v_c}}^{2}+2165\,{
          v_c}-1582 \right)  \left(1 -8\,t \right) ^{2/3}}{3^4 101^2}}
 %
% +{\frac { \left( 7019264\,{\it w_c}\,{{\it v_c}}^{2}-38297824\,{\it w_c}\,{\it v_c}   -20368984\,{\it w_c}+118984464 \right){2478843}}
 + cst \left(1 -8\,t \right) 
\\+\frac {2^8\,{6}^{1/3} { w_c}\, \left( 344660\,{{ v_c}}^{2}+688535\,{ v_c }-718546 \right) \left(1 -8\,t \right) ^{4/3}}{
3^5 101^3}
+\LandauO\left(  \left(1 -8\,t \right) ^{5/3}\right),
\end{multline*}
where each symbol $cst$ stands for a real constant that may vary from place to place, but has no implication on the asymptotic behaviour of the coefficients of our series. The series we are really interested in are
\[
  C(1,1)= A(1,1)-\frac 1 3 Q(1,1)
\]
and
\[
  C_{0,0}=A_{0,0} +\frac 1 3 Q_{0,0}.
\]
Recall from~\cite[Thm.~VI.1]{flajolet-sedgewick} that for $\alpha \not \in \{-1, -2, \ldots\}$, it holds that
\[
  [t^n] (1-8t)^{-\alpha-1} = \frac {8^n n^\alpha} {\Gamma(\alpha+1)} + \LandauO (8^nn^{\alpha-1} ).
\]
In particular, the $n$th coefficient in $A(1,1)$ grows like $8^n n^{-1/3}$, while the estimate corresponding to the remainder is in $8^n n^{-4/3}$. Moreover, it is proved in~\cite{Bostan-etal-2017-qp,MelczerMishna2016} that
\[
  [ t^n] Q(1,1)= \frac 8{3\pi} \frac{8^n}{n} + \LandauO\left(\frac{8^n}{n^2}\right),
\]
so that $Q(1,1)$ contributes to the second order term in the asymptotic behaviour of the number $c(n)$ of $n$-step walks in $\Cc$. We then  compute the minimal polynomial over $\Q$ of the constant occurring in the first term of~\eqref{A11sing}, and put the two contributions together to obtain the first part of the corollary.

Now consider the series $C_{0,0}$. Since the coefficient of $t^n $ in $Q_{0,0}$ grows like $8^n/n^3$ (see~\cite{Bostan-etal-2017-qp,DenisovWachtel15}), the first two terms in the expansion of $c_{0,0}(n)$ come from the above expansion of $A_{0,0}$, and this yields the second part of the corollary.
 \end{proof}

%%%%%%%%%%%%%%%%%%%%%%%%%%%%%%%%%%%%%%%%%%%%%%%%%%%%%%%%%%%%%%ù
 \section{Combinatorial proofs of some identities on square lattice walks}
 \label{sec:combi}
 %%%%%%%%%%%%%%%%%%%%%%%%%%%%%%%%%%%%%%%%%%%%%%%%%%%%%%%%%%%%%%ù

%generic walk from starting point to end point
\newcommand{\Walkse}[3]{#1^{#2}_{#3}}
%Q starting point to end point
\newcommand{\Qse}[2]{\Walkse{Q}{#1}{#2}}
%C starting point to end point
\newcommand{\Cse}[2]{\Walkse{C}{#1}{#2}}
%A starting point to end point
\newcommand{\Ase}[2]{\Walkse{A}{#1}{#2}}
%A starting point to end point
\newcommand{\Hse}[2]{\Walkse{H}{#1}{#2}}

As already observed in~\cite[Sec.~7.1]{Bousquet2016} for the simple and diagonal models, the first two equations of Theorem~\ref{theo:king}, combined with the $x/y$-symmetry of our step set, imply that for $i,j \ge 0$,
   \[
     C_{i,j}=Q_{i,j}+C_{-i-2, j}+ C_{i,-j-2}.
   \]
As suggested in\cite{Bousquet2016}, this can be proved using the reflection principle. This is what we do in this section. Further, we establish identities of this type for more general starting points and endpoints, and all Weyl models of Table~\ref{tab:weyl}. We begin in Section~\ref{sec:combi4} with the four models having a group of order $4$, and develop in Section~\ref{sec:combi-general} a general setting.

%======================================================
\subsection{A group of order \texorpdfstring{$4$}{4}: simple, diagonal, king, and diabolo walks}
\label{sec:combi4}
%======================================================

As shown in Table~\ref{tab:weyl} there are four step sets associated with
the Weyl group  $A_1\times A_1$, of order $4$.
Mimicking the action of this group on $\rs^2$, we decompose the three-quarter plane $\Cc$ into three disjoint parts: 
\begin{align*}
	\Qc &= \{ (i,j) : i \geq 0 \text{ and } j  \geq 0 \}   &&\text{(the first quadrant)},  \\
	\Lc &= \{ (i,j) : i \leq -1 \text{ and } j  \geq 0 \}  &&\text{(the left quadrant)},   \\
	\Bc &= \{ (i,j) : i \geq 0 \text{ and } j  \leq -1 \}  &&\text{(the bottom quadrant)}.
\end{align*}
As before, let $C_{i,j}$ (resp.\ $Q_{i,j}$) be the number of walks confined to $\Cc$ (resp.\ $\Qc$) ending at $(i,j)$. 
More generally, for any starting point $(a,b)$ we write $\Qse{a,b}{i,j}$ (resp.~$C^{a,b}_{i,j}$) for the length \gf\ of walks confined to $\Qc$ (resp.~$\Cc$), starting from $(a,b)$ and ending at $(i,j)$.
A step set~$\cS$ is called \emph{vertically symmetric} (or \emm v-symmetric,) if for all $(i,j) \in \cS$ one has $(-i,j) \in \cS$; 
it is called \emph{horizontally symmetric} (or \emm h-symmetric,) if for all $(i,j) \in \cS$ one has $(i,-j) \in \cS$.
The four models that we consider in this subsection are the only v- and h-symmetric models among all small step models.

\begin{figure}[htb]
 \centering
 \includegraphics[height=4.2cm]{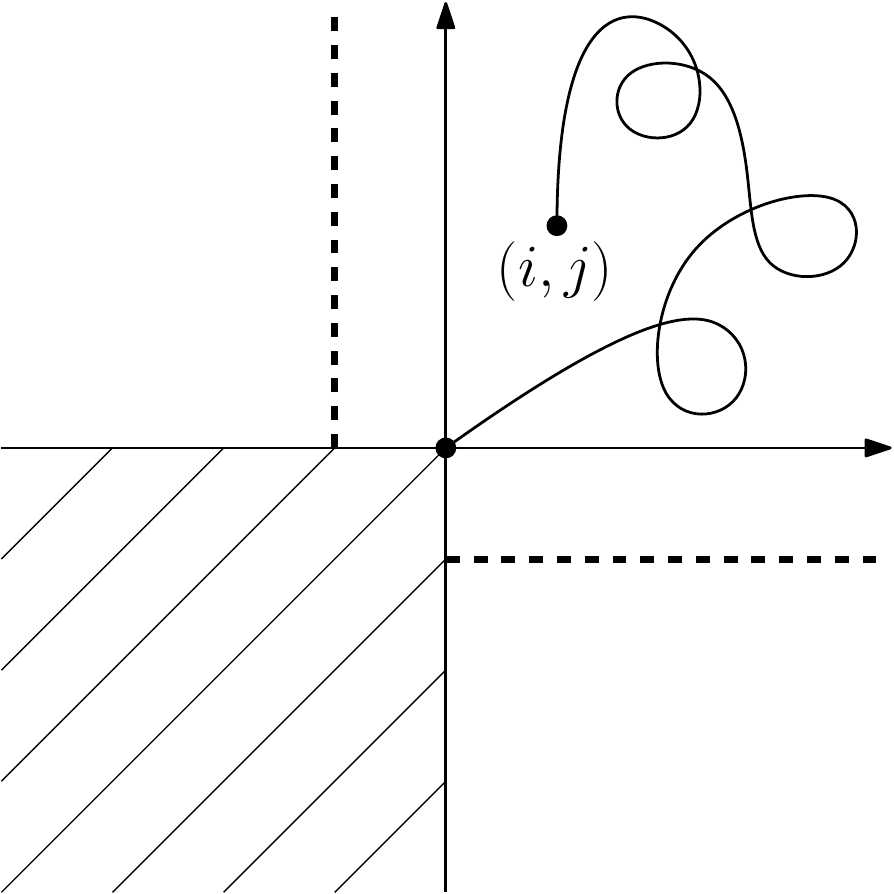}
	\qquad
 \includegraphics[height=4.2cm]{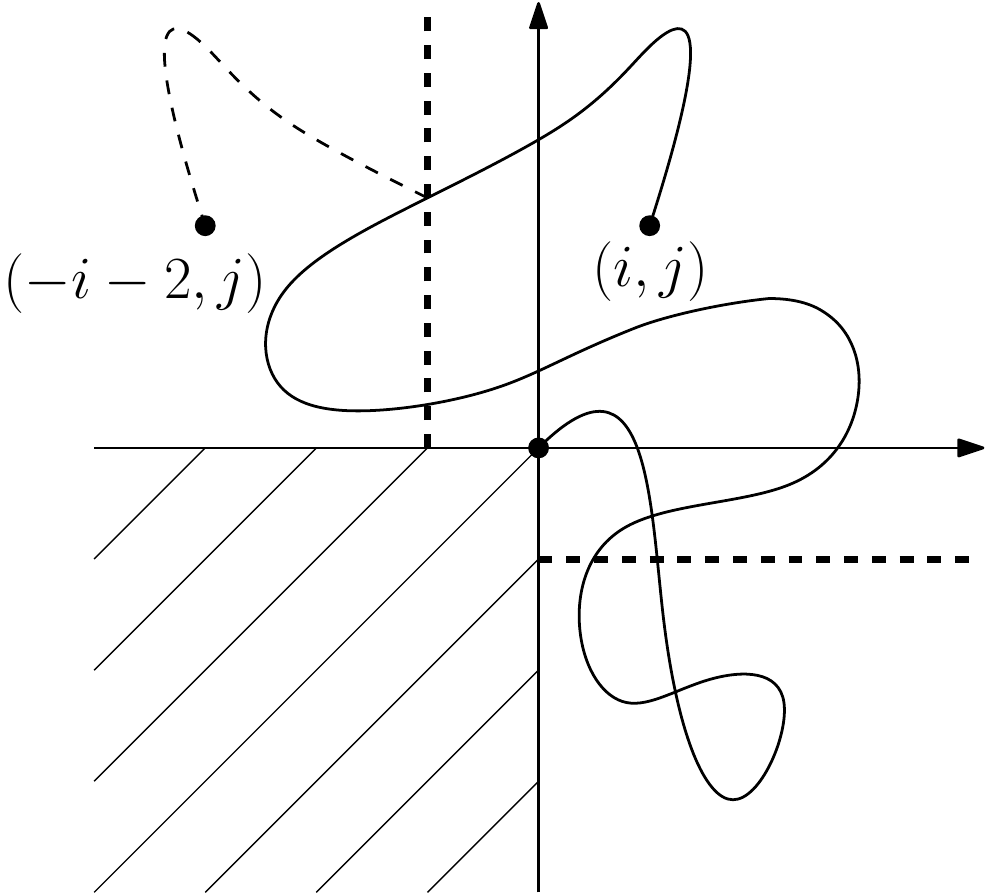}
	\qquad
 \includegraphics[height=4.2cm]{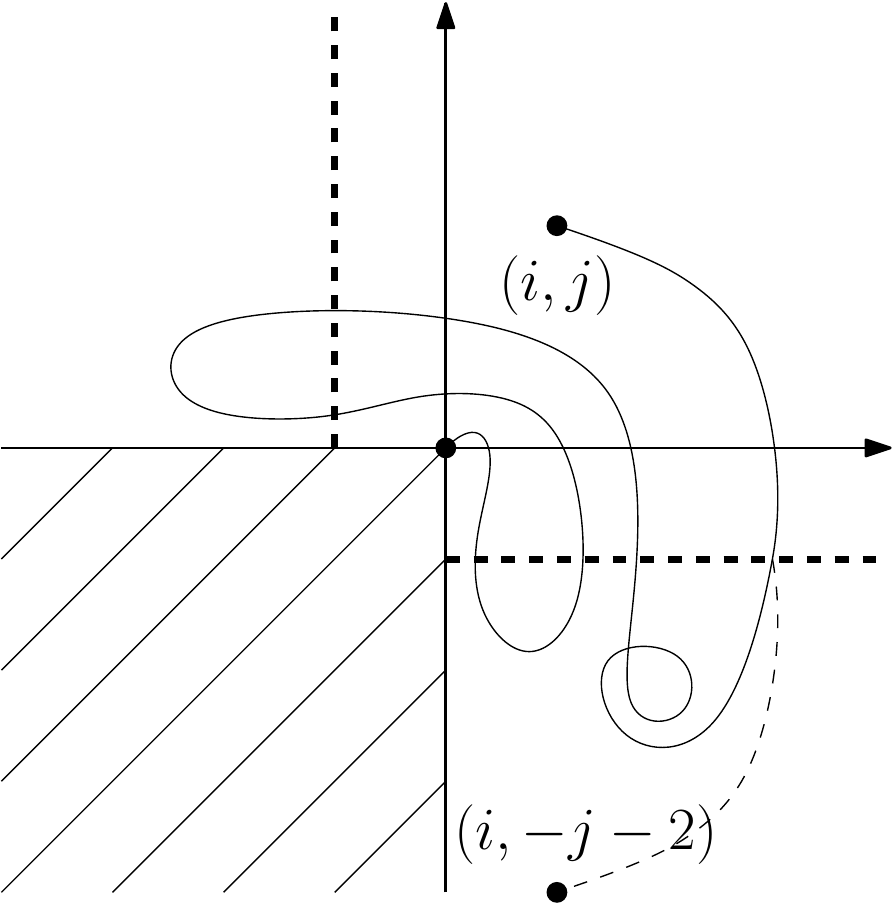}
 \caption{By the reflection principle, walks in the three-quarter plane $\Cc$ from $(0,0)$ to $(i,j)$ with $i,j \geq 0$ are in bijection with the union of three sets of walks: walks in $\Cc$ ending at $(-i-2,j)$, walks in $\Cc$ ending at $(i,-j-2)$, and walks staying completely in the first quadrant~$\Qc$, ending at $(i,j)$. For more such identities see Proposition~\ref{prop:bijection}.}
 \label{fig:bijection}
\end{figure}

\begin{prop}
	\label{prop:bijection}
	Let $\cS$ be one of the four v- and h-symmetric small step models, and let $(a,b)$ be a starting point in $\Cc$.
          For $(i,j) \in \Qc$  we have 
       \begin{align*}
	\Cse{a,b}{i,j} &= \Cse{a,b}{-i-2,j} + \Cse{a,b}{i,-j-2} +
		\begin{cases}
			 \Qse{a,b}{i,j} & \text{ if } a,b \geq 0, \\
			0  & \text{ if } a = -1 \text{ or } b=-1, \\
			 - \Qse{-a-2,b}{i,j} & \text{ if } a < -1, \\
		 - \Qse{a,-b-2}{i,j} & \text{ if } b < -1.
		\end{cases}
	\end{align*}	
	Furthermore, there exists an explicit bijection proving  each of these identities.
  \end{prop}
\begin{proof}
	The proof idea is to suitably reflect the walks along the lines $x=-1$ and $y=-1$ which directly results in bijections for the claimed identities.
        We  fix  an endpoint $(i,j) \in \Qc$.
        
	First, for a starting point $(a,b) \in \Qc$ we partition the walks confined to $\Cc$ into three classes as shown in Figure~\ref{fig:bijection}: a walk either always stays in the first quadrant and is therefore counted by $\Qc_{i,j} ^{a,b}$,
        or it leaves the first quadrant. 
	In the latter case it either touches the line $x=-1$ or $y=-1$.
	We cut the walk at the last point $(k,\ell)$ where this happens and reflect the second part of the walk,  going from $(k,\ell)$ to $(i,j)$, along this line. As $\Sc$         is v- and h-symmetric we get a walk in $\Cc$ with steps in $\Sc$  ending either at $(-i-2,j)$ or $(i,-j-2)$. 
        The reverse bijection is analogous. One key point here is that a walk from $(a,b)$ to $(-i-2,j)$ (say) will necessarily  touch the line $x=-1$,
         and will touch it after any visit to the line $y=-1$.

         Second, if the starting point $(a,b)$ is on the line $x=-1$ or $y=-1$ then the same argument applies, with $Q_{i,j}^{a,b}=0$ because  no path can be entirely in the first quadrant $\Qc$.

Third, if $a<-1$,  the path starts left of the line $x=-1$, and thus cannot be contained in the first quadrant either. Moreover, a difficulty arises when defining the reverse construction: a walk starting from $(a,b)$ and ending at $(-i-2,j)$ may not touch the line $x=-1$, and thus cannot be reflected along this line (there is no such problem with walks ending at $(i,-j-2)$). But these walks are in essence walks in a quadrant: reflecting them  along the line $x=-1$ gives walks   from $(-a-2,b)$ to $(i,j)$ confined to the first quadrant $\Qc$.
	
	Fourth, for $b<-1$ the reasoning is analogous.
\end{proof}

The above proposition implies in particular the three formulas given in \cite[Sec.~7.1]{Bousquet2016}:
for $i,j \geq 0$ we have for any v- and h-symmetric step set and the three starting points $(0,0)$, $(-1,0)$, and $(-2,0)$:
\begin{align*}	
	\Cse{0,0}{i,j} &= \Cse{0,0}{-i-2,j} + \Cse{0,0}{i,-j-2} + Q_{i,j}^{{0,0}}, \\
	\Cse{-1,0}{i,j} &= \Cse{-1,0}{-i-2,j} + \Cse{-1,0}{i,-j-2}, \\
	\Cse{-2,0}{i,j} &= \Cse{-2,0}{-i-2,j} + \Cse{-2,0}{i,-j-2} -Q_{i,j}^{{0,0}}. 
\end{align*}

Let us reformulate Proposition~\ref{prop:bijection} in terms of trivariate (rather than univariate) \gfs. For $(a,b) \in \Cc$,  let $C^{a,b}(x,y)$ denote the \gf \ of walks in $\Cc$ that start from $(a,b)$: 
\beq\label{Cab-def}  
    C^{a,b}(x,y)= \sum_{(i,j)\in \Cc} \Cse{a,b}{i,j} x^i y^j.
\eeq 
  We also define (uniquely) series $P^{a,b}(x,y)$, $L^{a,b}(x,y)$, and $B^{a,b}(x,y)$ in $\qs[x,y][[t]]$ by
  \beq\label{Cab-split}
    C^{a,b}(x,y)= P^{a,b}(x,y)+ \bx L^{a,b}(\bx,y)+ \by B^{a,b}(x,\by).
  \eeq
  Then  Proposition~\ref{prop:bijection} can be reformulated as follows.
  \begin{prop} \label{prop:bijection-3v}
    Let $\Sc$ be one of the four v- and h-symmetric small step models. For $(a,b)\in~\Cc$,  the above defined series are related by
    \begin{align*}
      P^{a,b}(x,y) &= \bx \left( L^{a,b}(x,y)-L^{a,b}(0,y)\right) +\by \left( B^{a,b}(x,y)-B^{a,b}(x,0)\right) \\
      &+
      \begin{cases}
        Q^{a,b}(x,y) & \text{if } a,b \ge 0, \\
        0   &\text{if } a=-1 \text{ or } b=-1, \\
        - Q^{-a-2,b}(x,y) & \text{if } a<-1, \\
        - Q^{a,-b-2}(x,y) & \text{if }b<-1. 
      \end{cases}
    \end{align*}
      \end{prop}
      \begin{proof}
    We multiply the identities of Proposition~\ref{prop:bijection}  by $x^i y^j$ and sum over all $i,j\ge 0$.        
      \end{proof}

Now we will use these results to generalize Equation~\eqref{Psol-king} to the four models under consideration.
First, we define the generating function $A(x,y)$ as in
Proposition~\ref{prop:A-def-gen}, or equivalently by~\eqref{AC-def-king}.
It satisfies the following functional equation:
\beq\label{eq:diaboloAfunceq}
 K(x,y)A(x,y)= \frac{2+\bx^2+\by^2}3 
-t\by H_-(x) A_{-,0}(\bx) -t\bx V_-(y)A_{0,-}(\by) -t\bx\by A_{0,0} \mathbbm 1_{(-1,-1)\in \cS}.
\eeq
Hence $A(x,y) = \sum_{(i,j) \in \Cc} A_{i,j} x^i y^j$ can be interpreted as  the \gf\ of walks starting from $(0,0)$, $(-2,0)$, or $(0,-2)$ with weights $2/3$, $1/3$, and $1/3$, respectively.
In particular, if we now define the series $P(x,y), L(x,y), B(x,y) \in \qs[x,y][[t]]$  by
\beq\label{A-split-diab}
 A(x,y)= P(x,y) + \bx L(\bx,y) + \by B(x, \by),
\eeq
(observe that $B(x, \by)= L(\by,x)$ for an $x/y$-symmetric model), we have
\beq\label{linear}
        P(x,y) = \frac 1 3 \left(2  P^{0,0}(x,y) + P^{-2,0}(x,y) + P^{0,-2}(x,y)\right),
\eeq
and analogously for the series $L$ and $B$.
    Then Proposition~\ref{prop:bijection-3v} implies the following generalization of Equation~\eqref{Psol-king}.

\begin{coro}
\label{coro:PLBgroup4}
In the case of simple, diagonal, king, or diabolo walks, the power series $P(x,y)$, $L(x,y)$, and $B(x,y)$ defined in~\eqref{A-split-diab} obey the following identity
%\beq\label{P-diabolo}
\[
	 P(x,y)=\bx \big( L(x,y)-L(0,y)\big) +\by \big( B(x,y)-B(x,0) \big).
\]
 %\eeq
\end{coro}

\begin{proof} 
  Applying Proposition~\ref{prop:bijection-3v} for $(a,b) = (0,0)$, $(-2,0)$, and $(0,-2)$ to~\eqref{linear} makes all contributions of the series $Q$ vanish and shows the claim.
\end{proof}

\begin{remark}
Let us  define  the  series $A(x,y)$ as in~\eqref{eq:diaboloAfunceq}, but with weights $w_0$,
  $w_x$, and $w_y$ for walks starting from $(0,0)$, $(-2,0)$, and
  $(0,-2)$ respectively (rather than $2/3$, $1/3$, $1/3$). This only changes the initial term
  in~\eqref{eq:diaboloAfunceq},  and Corollary~\ref{coro:PLBgroup4} still holds, provided that  $w_0=w_x+w_y$.
\end{remark}

\medskip
In the next  subsection, we give a higher level explanation of what happens here, more in the spirit of Gessel's and Zeilberger's proof of the reflection principle in~\cite{gessel-zeilberger}, and thus obtain  statements that are valid for all Weyl models.

%======================================================
\subsection{A general result for Weyl models}
\label{sec:combi-general}
%======================================================

\newcommand\sdot[1]{\smash{\overset{\mathbin{\vcenter{\hbox{\raisebox{-2mm}{\hspace{0.2mm}\huge.}}}}}{#1}}}

We now consider one of the seven Weyl models $\cS$ of Table~\ref{tab:weyl}, with a group $G$ of order $2d$, $d\in\{2, 3, 4\}$. Recall the definition of this group from Section~\ref{sec:group}, and the definition of the length $\ell(g)$ and sign $\varepsilon_g$ of $g \in G$. This group acts on steps, seen as elements of the vector space~$\zs^2$: for $g\in G$, the corresponding element $\vec g$ sends $(i,j)$ to $(k,\ell)$ if $g(x^i y^j)=x^k y^\ell$ (recall that we have defined $g(F(x,y)):=F(g(x,y))$ for any rational function $F(x,y)$). By construction of $G$, the set of steps $\cS$ is invariant under this action of $G$. The group $G$ also acts on points of the plane, that is, on the \emm affine space, $\zs^2$, by $\sdot g (a,b)=(c,d)$ where $x^cy^d=\bx\by \, g(x^{a+1}y^{b+1})$ (the shift by $xy$ is a bit unfortunate, and would be avoided by considering the positive quadrant $\{(a,b): a >0, b>0\}$ rather than the non-negative quadrant $\Qc$).

For $g \in G$, we denote $\Qc_g= \sdot g (\Qc)$. The $2d$ domains $\Qc_g$, for $g \in G$, are disjoint; see Figure~\ref{fig:reflection}. For $(a,b)\in \Qc$, the orbit of $(a,b)$ under the affine action of $G$ consists of $2d$ distinct points of the plane. In particular, the orbit sum $\OS(x^{a+1}y^{b+1})= \sum_g \vareps_g g(x^{a+1}y^{b+1})$ is non-zero. The union of the $2d$ domains $\Qc_g$ does not cover the whole plane. For the  points $(a,b)$  that are not in this union, the orbit of $(a,b)$ under the affine action of $G$ has cardinality less that $2d$, and in fact $\OS(x^{a+1}y^{b+1})= 0$. The complement of $\cup_g \Qc_g$ is the union of $d$ lines (also called \emm walls, to match  the terminology of~\cite{gessel-zeilberger}), defined,  for each  $g\in G$ such that $\vareps_g=-1$ (i.e., $\ell(g)$ is odd), by  $W_g=\{(a,b)\in \zs^2 : \sdot g(a,b)=(a,b)\}$.  The lines are  dashed  in our figures, and  correspond to the reflection axes once the steps are straightened (as in Table~\ref{tab:weyl}). For instance, in all cases we have $W_\phi=\{(a,b):a=-1\}$ and $W_\psi=\{(a,b):b=-1\}$.  Any two of the lines $W_g$ intersect at the point $(-1,-1)$. An important property is that a walk that is not entirely contained in a domain $\Qc_g$ must touch one of these lines.

\begin{figure}[htb]
 \centering
 \includegraphics[width=0.35\textwidth]{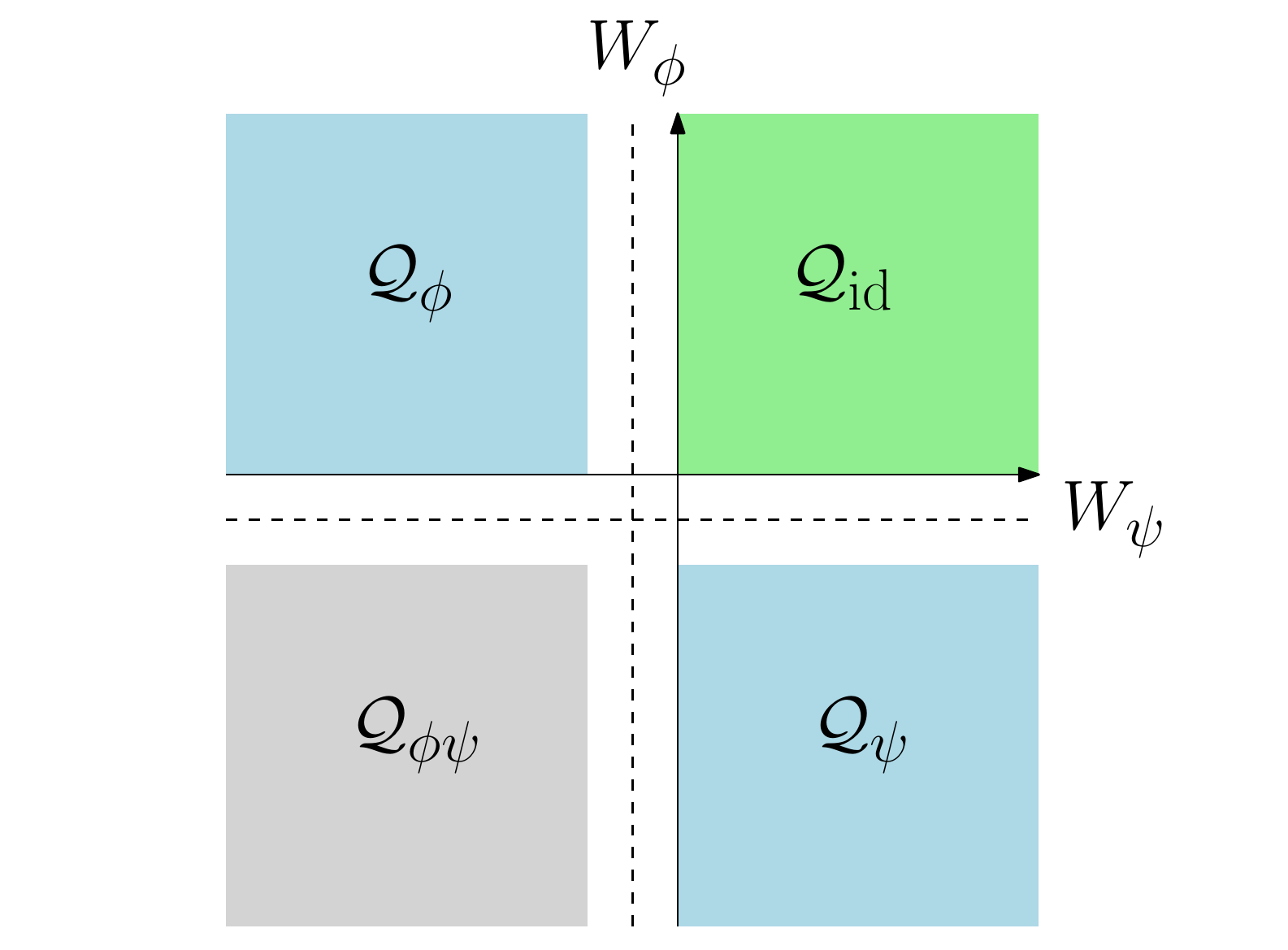}
%\quad
\hskip -6mm \includegraphics[width=0.35\textwidth]{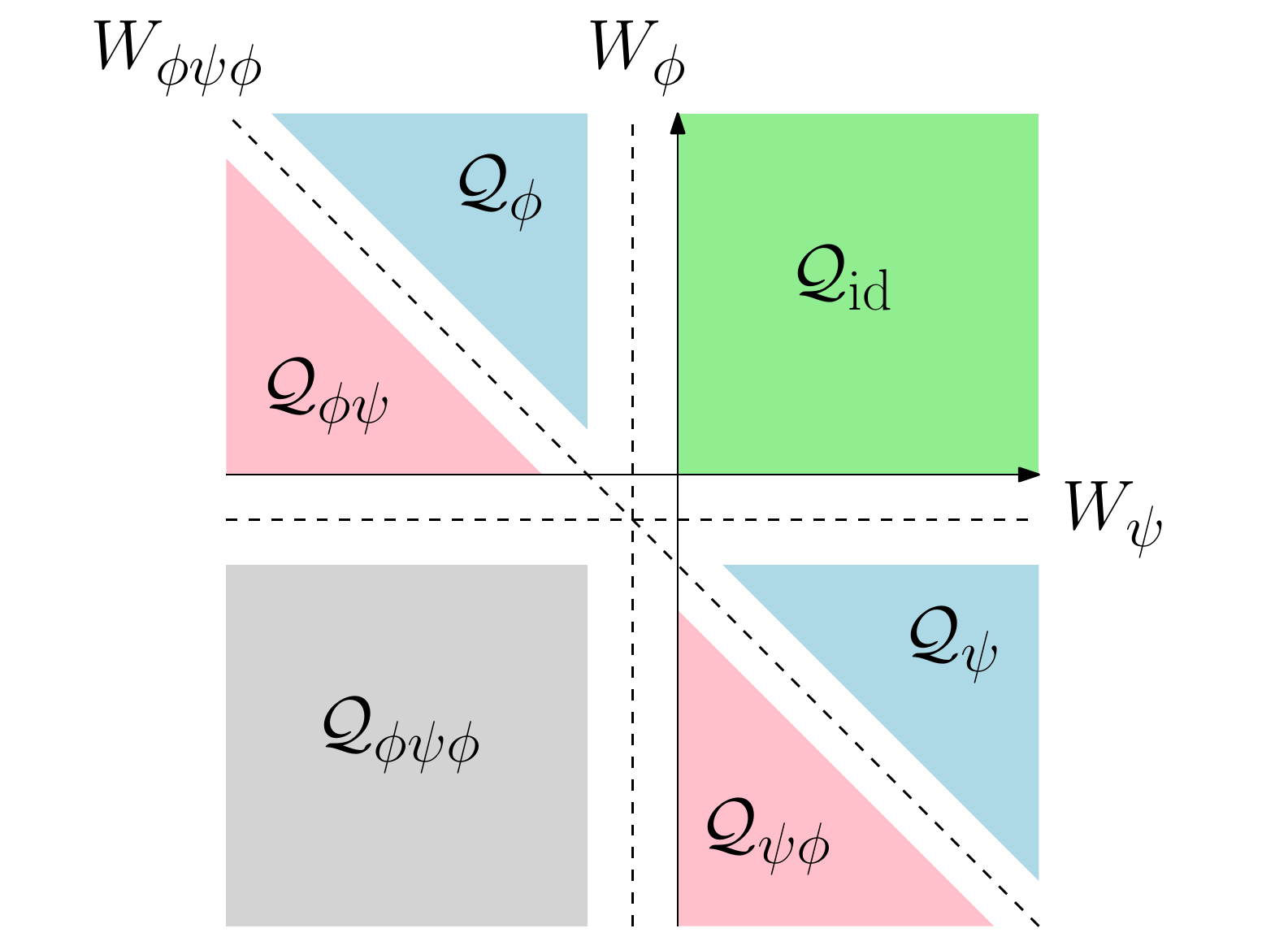}
%\quad
\hskip -6mm  \includegraphics[width=0.355\textwidth]{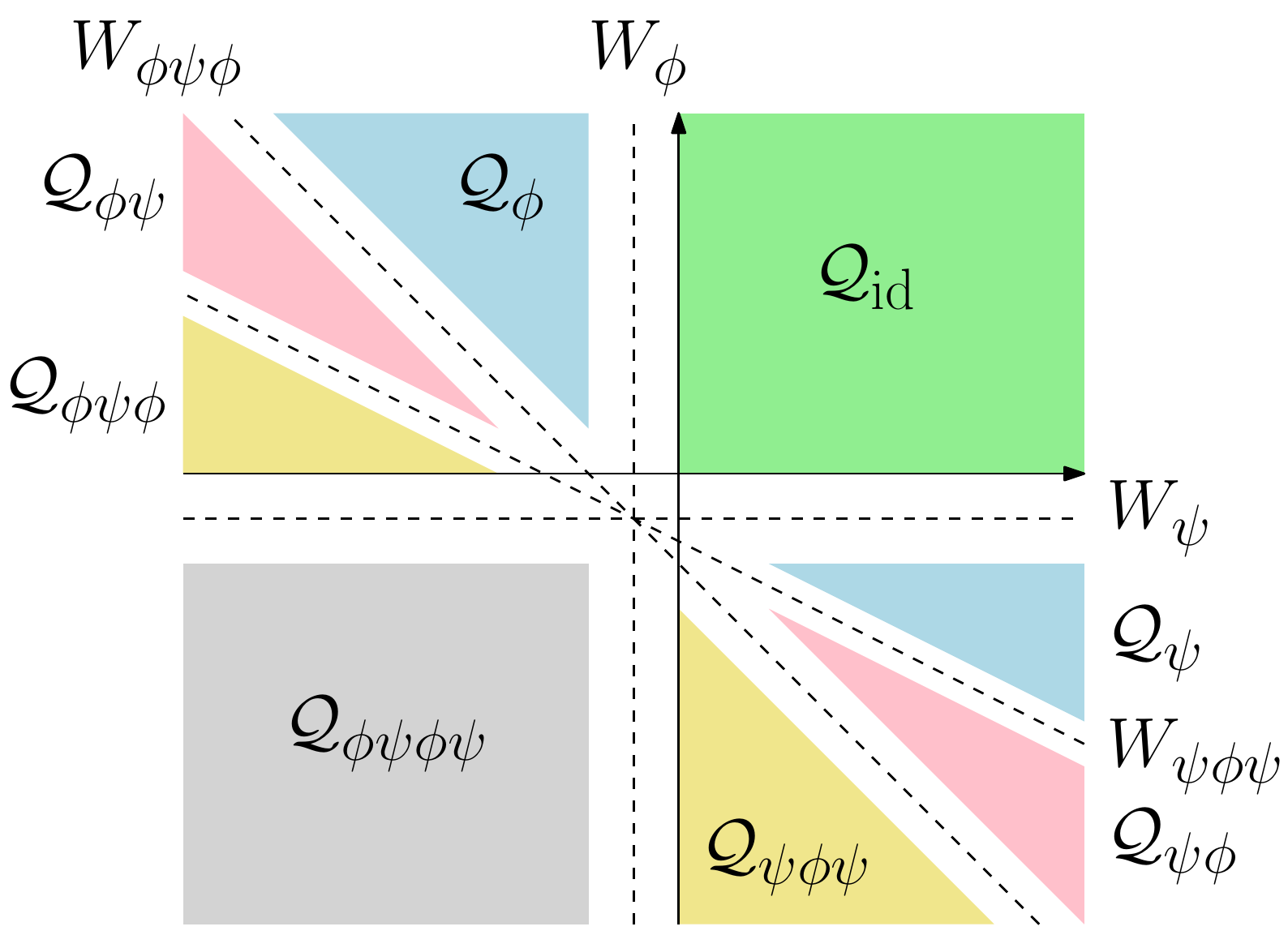}
 \caption{The $2d$ domains $\Qc_g$ for $g \in G$, where $G$ has order  $2d=4,6, 8$. They are separated by the $d$ walls $W_g$, for $g\in G$ such that $\vareps_g=-1$.}
 \label{fig:reflection}
\end{figure}

We adopt the same notation $C^{a,b}_{i,j}$ and $C^{a,b}(x,y)$ as in the previous subsection. The generalization of Proposition~\ref{prop:bijection} reads as follows.

\begin{prop}\label{prop:bij3v-general}
  Let $\Sc$ be one of the Weyl models of Table~\ref{tab:weyl}. Let $2d$ be the order of the associated group $G$. Let $\omega=\phi\psi\psi\cdots$ (with $d$ generators) be the only element of length $d$ in~$G$. For any starting point $(a,b) \in \Cc$ and any endpoint $(i,j)\in \Qc$, we have:
  \[
    \sum_{g\in G\setminus\{\omega\}} \vareps_g\,  C^{a,b}_{\sdot g(i,j)} =
    \begin{cases}
      0 & \text{if } (a,b) \not \in \bigcup_g \Qc_g, \\
      \vareps_h \, Q_{i,j} ^{\sdot h(a,b)} & \text{if } \sdot h(a,b) \in \Qc \text{ for } h \in G.
    \end{cases}
    \]
  \end{prop}
  \begin{proof}
    Recall that for $(i,j)\in \Qc$, the $2d$ endpoints  $\sdot g(i,j)$ are distinct. Hence the left-hand side of the above identity counts walks in $\Cc$, starting from $(a,b)$ and ending at one of the $2d-1$ points in the (affine) orbit of $(i,j)$ that are not in the negative quadrant, with a sign that depends on the domain $\Qc_g$ where the walk ends. Observe that for each walk the parameters $(a,b)$, $(i,j)$, and $g$ are uniquely determined.
    We will define a (partial) sign-reversing involution $\iota$ on these walks. The idea is sketched in Figure~\ref{fig:tandem-reflection-walks} for a group of order $6$, that is, for tandem or double tandem walks.

    Let $w$ be such a walk. If it does not intersect any of the walls, then $\iota(w)$ is undefined. In this case, the starting point $(a,b)$ of $w$ must be in one (and exactly one) of the domains $\Qc_g$, say in $\Qc_{h^{-1}}$ (so that $\sdot h(a,b)\in \Qc$). Then the endpoint of $w$ must be in $\Qc_{h^{-1}}$ as well, and applying $\sdot h$ to the walk $w$ (seen as a sequence of vertices) sends $w$ to a walk joining $\sdot h(a,b)$ to $ (i,j)$ in $\Qc$. Hence the signed number of walks that do not intersect any wall is given by  the right-hand side of the identity.

    Now assume that $w$ intersects one of the $d$ walls, and write $w=(w_0, \ldots, w_n)$ where the $w_i$'s are points of $\Cc$. Consider the largest $m$ such that $w_m$ is on one of the walls $W_h$. Note that the group element $h$ is uniquely defined, because the walls only intersect at $(-1,-1)$, which is not in~$\Cc$. Moreover, we have $m<n$ because the final point $\sdot g(i,j)$ is not on a wall. More generally, all points $w_{m+1}, \ldots, w_n=\sdot g(i,j)$  lie in $\Qc_g$. Now, form the walk $\iota(w):=(w_0, \ldots, w_m=\sdot h (w_{m}), \sdot h (w_{m+1}),\ldots, \sdot h (w_{n}))$. Note that the points $\sdot h (w_{m+1}),\ldots, \sdot h (w_{n})$ lie in the domain $\Qc_{hg}$. The new walk has still steps in $\cS$, because $\Sc$ is invariant under the (vectorial) action of $G$.

    Let us prove that it lies in the three-quadrant cone $\Cc$. This holds obviously for the first $m$ steps. If this were not true for the rest of the walk, then either the step $(w_m, \sdot h (w_{m+1}))$ would be one of the two forbidden steps joining $(-1,0)$ to $(0,-1)$ (but this is impossible because $\sdot h (w_{m+1})$ is not on a wall), or all points $\sdot h (w_{m+1}),\ldots, \sdot h (w_{n})$ would be in the domain $\Qc_\omega$. But this is not possible either since $w_m=\sdot h (w_{m})$ would then have both coordinates negative.

    Since $\iota(w)$  ends at $\sdot h\circ \sdot g(i,j)$, its sign is $-\vareps_g$ (because $h$ has odd sign). Its last visit to a wall is clearly $w_m \in W_h$, so $\iota \circ \iota(w)=w$ and we have indeed constructed a sign reversing involution of walks that visit at least one wall. This concludes the proof.
  \end{proof}

\begin{figure}
  \makebox[\textwidth][c]{
    \includegraphics[width=0.27\textwidth]{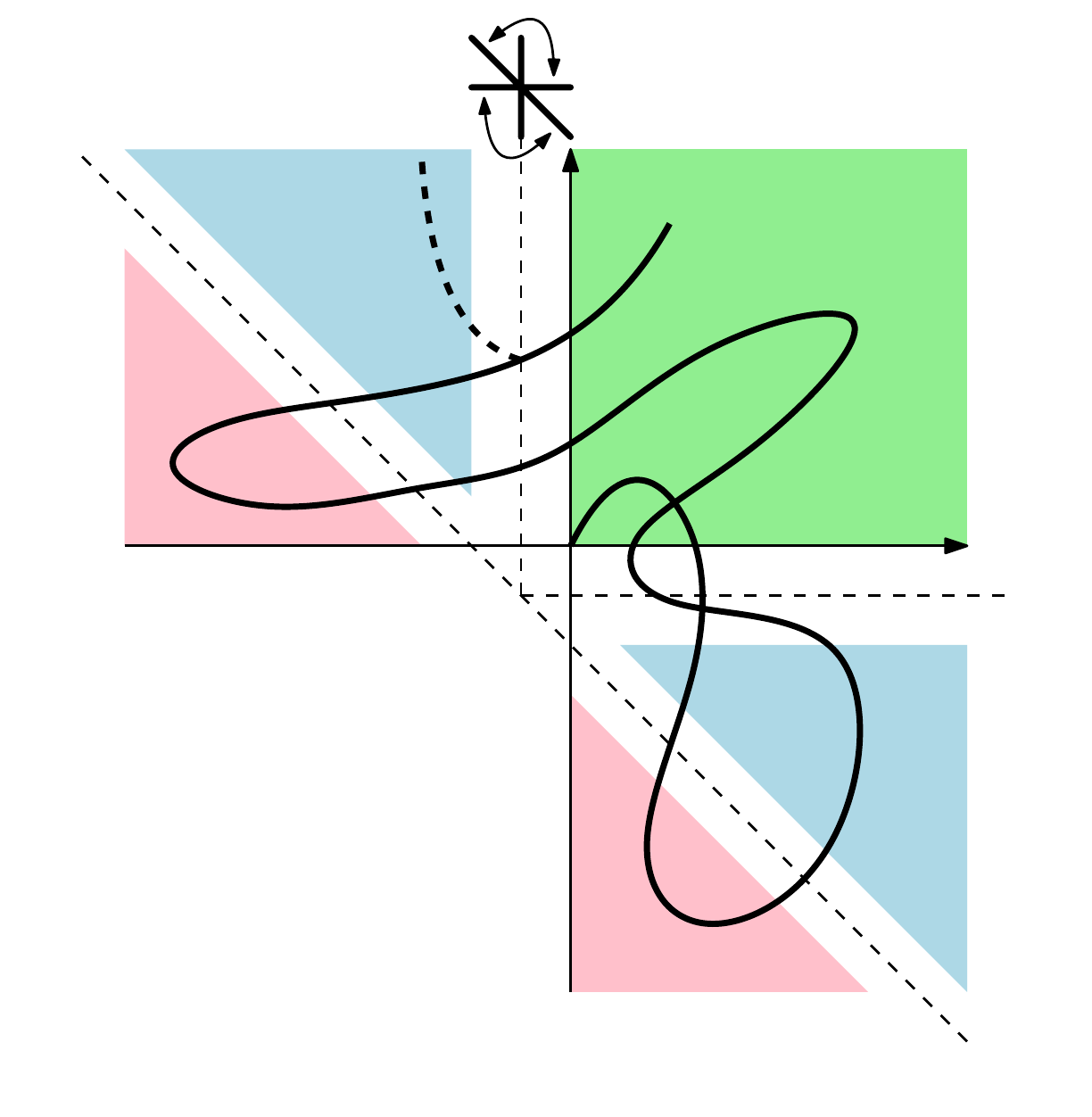}
    \hskip -1.5mm
    \includegraphics[width=0.27\textwidth]{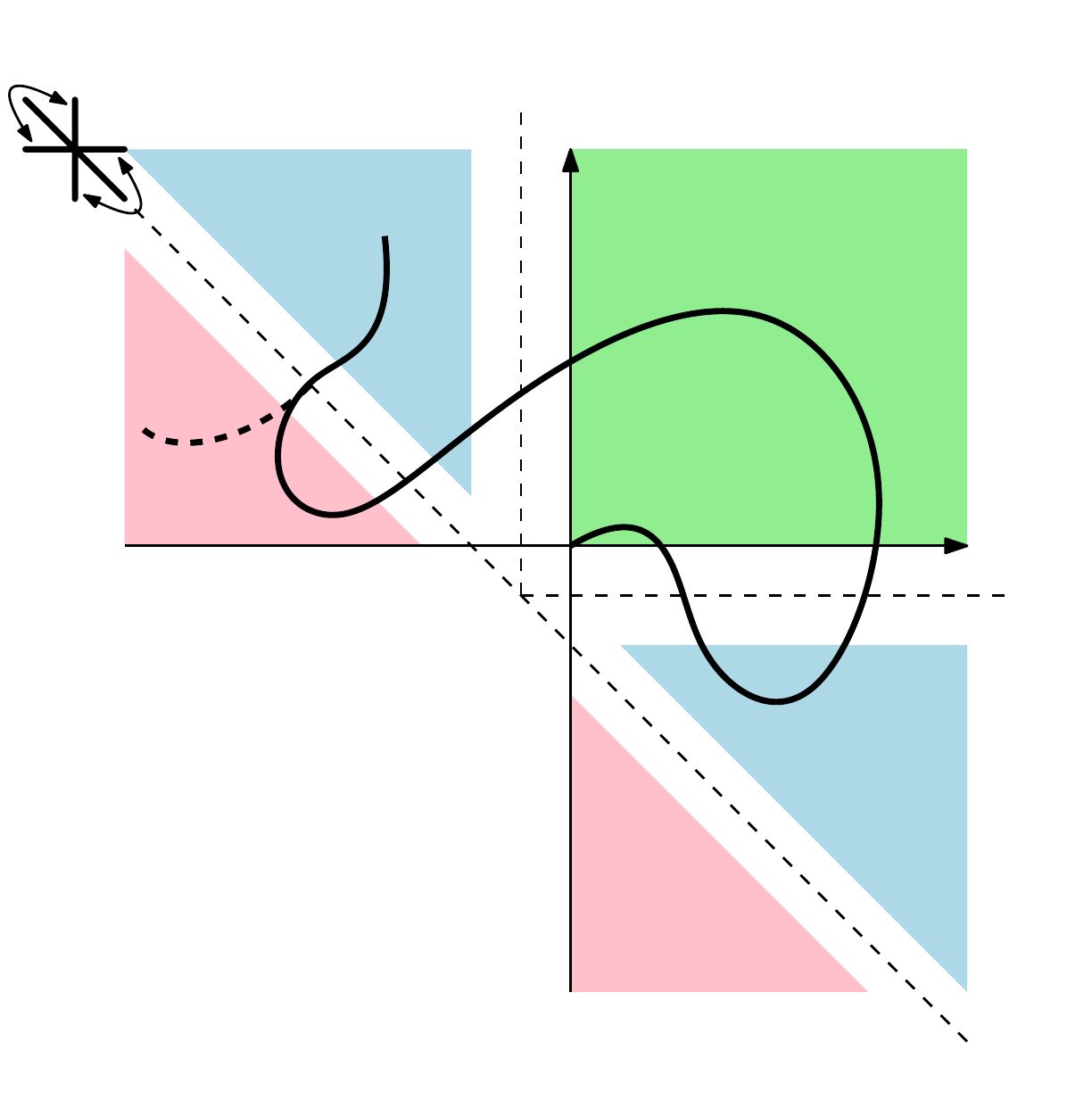}
    \hskip -1.5mm
    \includegraphics[width=0.27\textwidth]{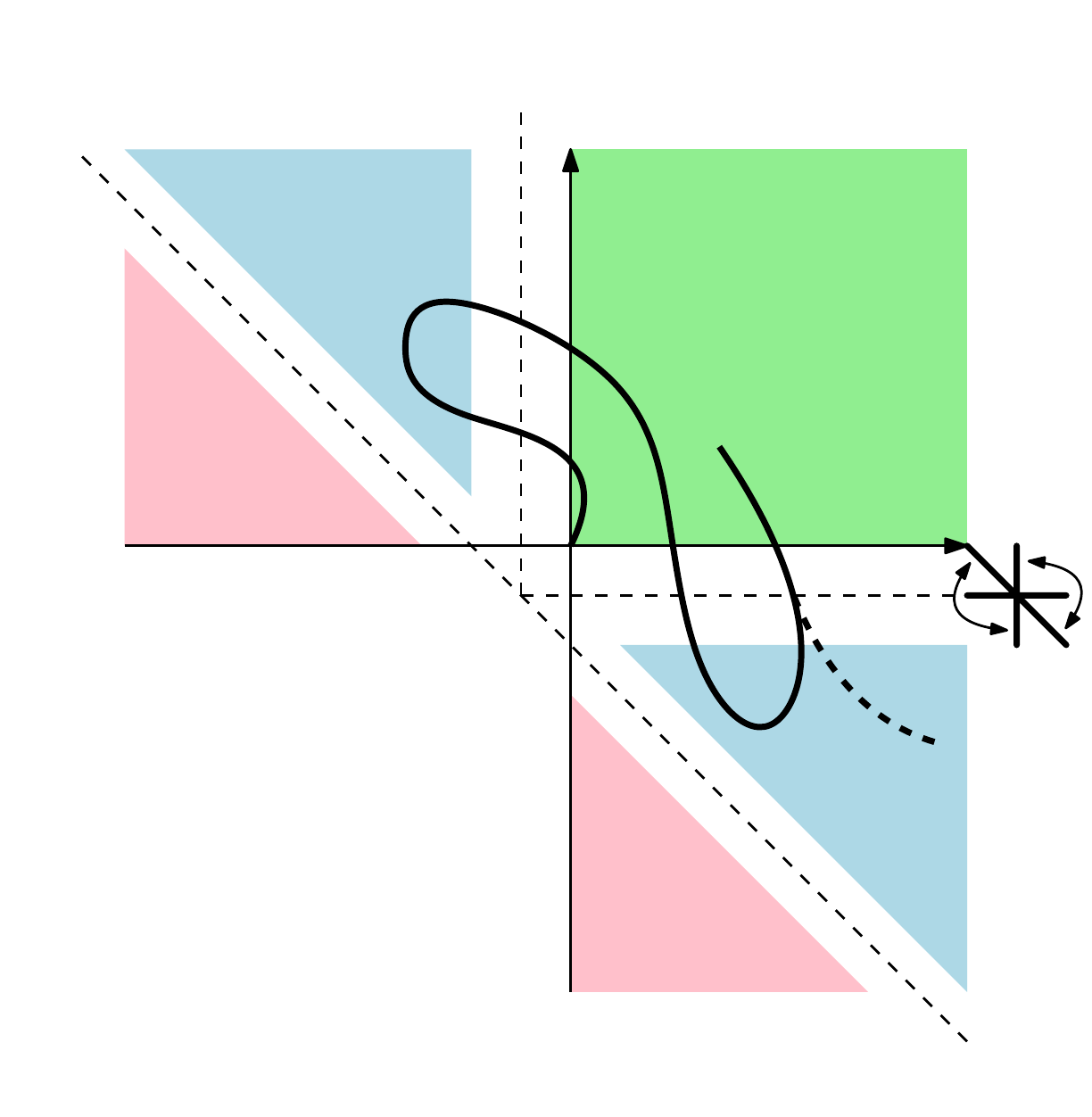}
    \hskip -1.5mm
    \includegraphics[width=0.27\textwidth]{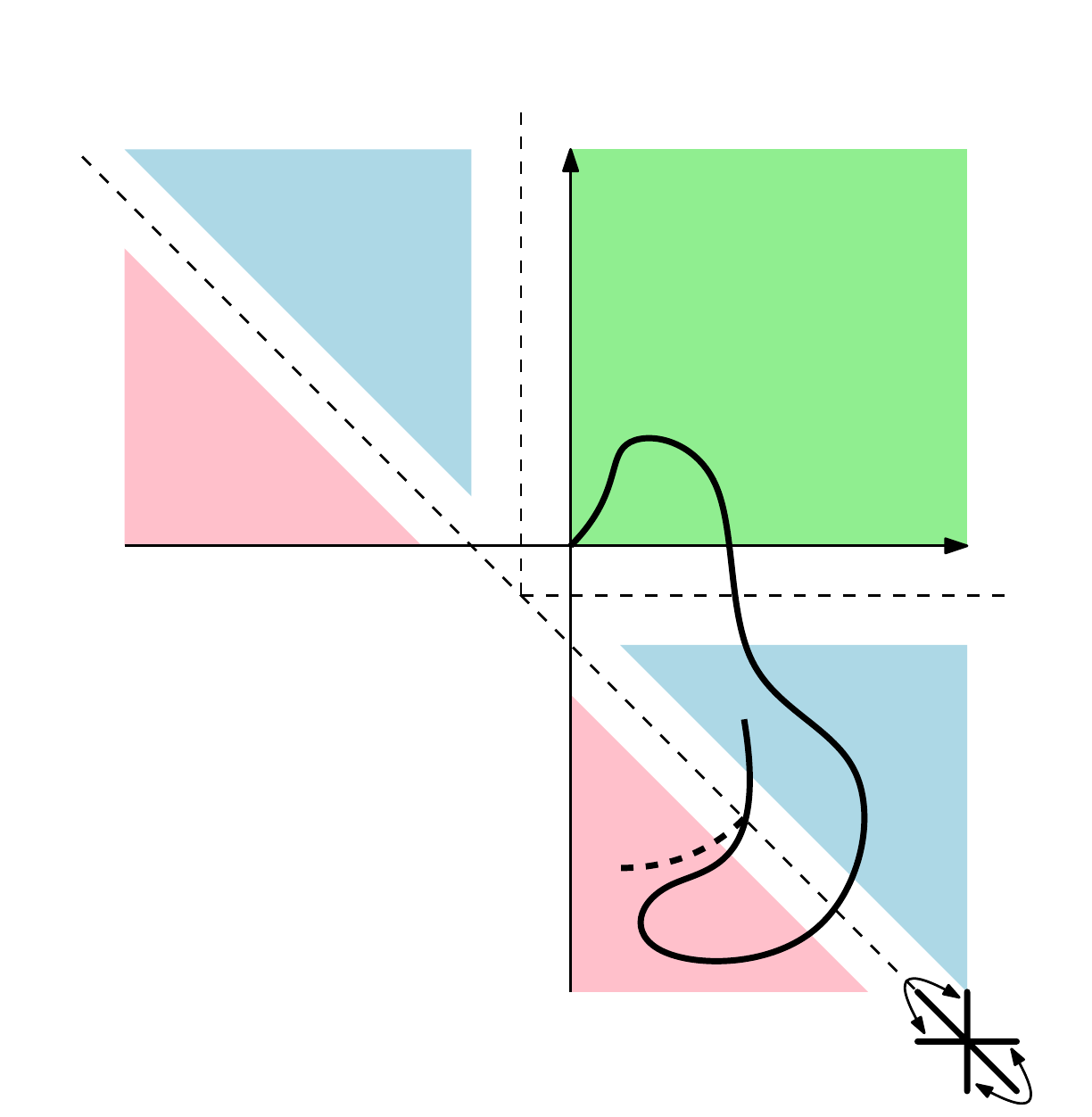}}
  \caption{The involution of Proposition~~\ref{prop:bij3v-general} for walks starting at $(0,0)$ and a group of order $6$. The above eight walks capture all possible values of the pair $(g,h)$, where $\Qc_g$ is the domain in which the walk ends and $W_h$ the last visited wall. The figure also shows the action of $h$ on the steps.}
 \label{fig:tandem-reflection-walks}
\end{figure}

\pagebreak
% to make a nicer split and avoid an orphan

Let us now reformulate the above proposition in terms of trivariate \gfs, as in Proposition~\ref{prop:bijection-3v}. Given $(a,b)\in \Cc$ and $g\in G$, the \gf\ of walks in $\Cc$ starting from $(a,b)$ and ending in $\Qc_g$ reads
  \begin{align*}
    \sum_{(k, \ell) \in \Qc_g} C^{a,b}_{k, \ell} x^k y^\ell
   & =
     \sum_{(i, j) \in \Qc} C^{a,b}_{\sdot g (i,j)}\bx \by \, g\!\left(x^{i+1} y^{j+1} \right)
     \\
 & = \bx \by \, g(xy) P_g^{a,b}(g(x,y)),
  \end{align*}
  where  $P_g^{a,b}(x,y) =  \sum_{(i, j) \in \Qc} C^{a,b}_{\sdot g (i,j)} x^i y^j$ is a series in $\qs[x,y][[t]]$. For instance, when $G$ has order $4$, it follows from~\eqref{Cab-split} that
  \begin{align*}
  P_\id^{a,b}(x,y) &= P^{a,b}(x,y),\\
  P_\phi^{a,b}(x,y)&=\bx\left(L^{a,b}(x,y)- L^{a,b}(0,y)\right), \\
    P^{a,b}_\Psi(x,y)&= \by \left( B^{a,b}(x,y)-B^{a,b}(x,0)\right).
  \end{align*}
  The generalization of Proposition~\ref{prop:bijection-3v} reads as follows.
  
  \begin{prop}  \label{prop:bijection-3v-gen} Let $\Sc$ be one of the Weyl models of Table~\ref{tab:weyl}. Let $2d$ be the order of the associated group $G$. Let $\omega=\phi\psi\psi\cdots$ (with $d$ generators) be the only element of length $d$ in~$G$. For any starting point $(a,b) \in \Cc$, we have:
    \[
      \sum_{g \in G \setminus\{\omega\}} \vareps_g P_g^{a,b}(x,y)=
    \begin{cases}
      0 & \text{if } (a,b) \not \in \bigcup_g \Qc_g, \\
      \vareps _h Q^{\sdot h(a,b)}(x,y) & \text{if } \sdot h(a,b) \in \Qc \text{ for } h \in G.
    \end{cases}
    \]
  \end{prop}
  \begin{proof}
    Multiply the identity of Proposition~\ref{prop:bij3v-general} by $x^i y^j$ and sum over $i, j \ge 0$.
  \end{proof}

  Our final result deals with the series $A(x,y)$ defined in Proposition~\ref{prop:A-def-gen}. Observe that the orbit sum of $xy$ can be written in terms of the affine orbit of $(0,0)$:
  \[
    \bx \by \OS(xy) =\sum_{\substack{h \in G \\[1mm] (c,d)=\sdot h (0,0)}}
     \vareps_h \, x^c y^d.
  \]
 Hence, from the functional equation~\eqref{A-func-eq}   satisfied by $A(x,y)$, we conclude that this series counts weighted walks in $\Cc$, and more precisely, that
  \[
    A(x,y)= \frac{2d-2}{2d-1}\,  C^{0,0}(x,y) - \frac 1 {2d-1} \sum_{h \in G\setminus\{\id, \omega\}} \vareps_h \, C^{\sdot h(0,0)}(x,y).
  \]
  For $g\in G$, let us denote by $\bx \by \, g(xy) P_g(g(x,y))$ the contribution in $A(x,y)$ of (weighted) walks ending in $\Qc_g$. As before, this notation is designed so that $P_g(x,y) \in \qs[x,y][[t]]$. Then
  \beq\label{P-Pab}
    P_g(x,y)=  \frac{2d-2}{2d-1}\,  P^{0,0}_g(x,y) - \frac 1 {2d-1} \sum_{h \in G\setminus\{\id, \omega\}} \vareps_h \, P^{\sdot h(0,0)}_g(x,y).
  \eeq
We can now state   the generalization of Corollary~\ref{coro:PLBgroup4}.
  \begin{coro}
    The above defined series $P_g(x,y)$ are related by
    \[
      \sum_{g \in G \setminus\{\omega\}} \vareps_g P_g(x,y)=0.
    \]
  \end{coro}
  \begin{proof}
    We first use the identity~\eqref{P-Pab}, and then Proposition~\ref{prop:bijection-3v-gen}. This gives:
    \begin{align*}
      \sum_{g \in G \setminus\{\omega\}} \vareps_g P_g(x,y)
      &=  \frac{2d-2}{2d-1}\,   \sum_{g \in G \setminus\{\omega\}} \vareps_g P^{0,0}_g(x,y) - \frac 1 {2d-1} \sum_{h \in G\setminus\{\id, \omega\}} \vareps_h
        \sum_{g \in G \setminus\{\omega\}} \vareps_g P^{\sdot h(0,0)}_g(x,y)
      \\
      &= \frac{2d-2}{2d-1}\,  Q^{0,0}(x,y)  - \frac 1 {2d-1} \sum_{h \in G\setminus\{\id, \omega\}} \vareps_h \vareps_{h^{-1}} Q^{0,0}(x,y)
      \\
  &    =0,
    \end{align*}
    since $\vareps_h \vareps_{h^{-1}} =1$ and $G$ has order $2d$.
  \end{proof}

   \begin{remark}
   It is easy to see that the results of this section hold as well if we allow steps between the points $(0,-1)$ and $(-1,0)$, as in~\cite{Budd2020Winding,ElveyPrice2020Winding}. 
  \end{remark}

%%%%%%%%%%%%%%%%%%%%%%%%%%%%%%%%%%%%%%%%%%%%%%%%%%%%%%
\section{Final comments}
\label{sec:final}
%%%%%%%%%%%%%%%%%%%%%%%%%%%%%%%%%%%%%%%%%%%%%%%%%%%%%%

The first question raised by this paper is whether all seven models of Table~\ref{tab:weyl} actually obey the pattern described in Conjecture~\ref{conj:Weyl}. Does $C(x,y)$ differ from the linear combination of series $Q(\cdot, \cdot)$ given in Proposition~\ref{prop:A-def-gen} by an algebraic series? This is now proved for three of these seven models.

In terms of techniques, one can of course try to extend the approach of this paper to the other four Weyl models. Another idea would be to try to use the technique, based on \emm invariants,, that has been used recently~\cite{mbm-tq-kreweras} to solve the first three (algebraic) models of Table~\ref{tab:zero}. In fact, it is shown in~\cite{mbm-tq-kreweras} that this approach also works for the simple and diagonal models. Can it be adapted to the four unsolved Weyl cases?  
to Gessel's model (number four in Table~\ref{tab:zero})?

Next to these $7+4=11$ models, there remain $12$ models with finite and non-monomial group, as shown in Table~\ref{tab:non-monomial}.
The non-monomial group action when applied to power series, prevents the efficient extraction of (positive/negative) parts.
For this reason the methods of this paper become even more complicated, and probably new approaches have to be developed.

Another question is whether the \gf\ for walks in other cones -- possibly larger than $2\pi$, as in~\cite{Budd2020Winding,ElveyPrice2020Winding} -- may  satisfy a similar algebraicity phenomenon; that is,  decompose into a simple D-finite series with the same orbit sum and an algebraic one.

We conclude with a sketch of the solution of the king model in which we
allow moves from $(0,-1)$ to $(-1,0)$ and back, as in~\cite{Budd2020Winding,ElveyPrice2020Winding}.

\subsection*{Allowing steps between \texorpdfstring{$(-1,0)$}{(-1,0)} and \texorpdfstring{$(0,-1)$}{(0,-1)} in king walks}
As already mentioned in this paper, 
in two recent references dealing with the winding number of plane lattice walks~\cite{Budd2020Winding,ElveyPrice2020Winding}, it seems more natural to count walks in which all vertices lie in $\Cc$, but not necessarily all edges: that is, one allows  steps form $(-1,0)$ to $(0,-1)$, and vice versa. It is natural to ask whether this choice leads to simpler series. This is why we have re-run our \Maple\ sessions on this variant of the king model. The first steps of the derivation, until the determination of the series $R_0$, $R_1$, $B_1$, and $B_2$ (as in Proposition~\ref{prop:explicit-univariate}) appear to be a bit simpler, but this stops being the case as soon as we return to the series $R(x)$ and $S(x)$. Let us give a few details.

First, the only changes in our basic functional equations are a term $t(\bx+\by) C_{-1,0}$ in the right-hand side of~\eqref{eq:kernelC1}, and a term $t\bx M_{0,0}$ in the right-hand side of~\eqref{eqM-king}. The new series $\Sh(x)$ is obtained from~\eqref{eq:STcardanosubs} by deleting the term in $R_0$. We still denote $\St(x)=(x+\bx+1)\Sh(x)/(x-\bx)$, and then the equation in one catalytic variable that we have to solve reads, with the same notation as in~\eqref{eqSc}:
\begin{align*}
0 &=  27  \left( 2 t+z+1 \right)  \left( 10 t-3 z+1 \right) {\St(x)}^{3}+
  \big(  \left( 54-243 R_1+54 t \right) R_0 -243 R_1^{2}+243 R_1 t \\
& \quad -81 zR_1 -45 {t}^{2}+18 zt-27 B_1+27 B_2+81 R_1-54 t+18 z-9 \big) \St(x)\\
&\quad -81 {R_0}^{2}+ \left( 81 t+27-27 z-162 R_1 \right) R_0-81 R_1^{2}+81 R_1 t-27 zR_1\\ 
&\quad +5 {t}^{2}+10 zt-3 {z}^{2}+9 B_1+9 B_2+27 R_1-6 t+4 z-2=0.
\end{align*}
The system defining $R_0$, $R_1$, $B_1$, and $B_2$ is also a bit more compact, and in fact we can  derive polynomial equations for each individual series without having to guess them first.  They are now all of degree $24$ (while $B_1$ had degree $12$ in the first setting), and are found to belong to $\Q(t,w)$. For instance, we now have
\[
  R_0= \frac{v(1-2t) }{v^4+8v^3+6v^2+2v+1}\left( 1+2v+ \frac {2v^3-4v-1}{2w}\right).
\]

Then we get back to $S(x)$, and that is where things become after all a bit more complicated than in the first setting. For instance, $\St(x)$ has now degree $72$ rather than $36$. It can be written as $\St_0(x)+w \St_1(x)$, where both series $\St_i(x)$, now of degree $36$, belong to $\Q(v,\Pzero(x))$ and hence to $\Q(v,\Pun(x))$, where $\Pzero$ and $\Pun$ are series defined in Appendix~\ref{app:P0} and Section~\ref{sec:bivariate}, respectively. Finally, both series $R(x)$ and  $S(x)$ are found to belong to $\Q(t,w, x, \Pun(x))$ and have degree $72$ over $\qs(t,x)$.

%%%%%%%%%%%%%%%%%%%%%%%%%%%%%%%%%%%%%%%%%%%%%%%%%%%%%%
 \section*{Acknowledgements}
%%%%%%%%%%%%%%%%%%%%%%%%%%%%%%%%%%%%%%%%%%%%%%%%%%%%%%
We are extremely thankful to Mark van Hoeij, who helped us a lot in finding a simple description of our algebraic series of high degree. Our warm thanks also go to Bruno Salvy for  his  help with several \Maple\ problems that we met.

%%%%%%%%%%%%%%%%%%%%%%%%%%%%%%%%%%%%%%%%%%%%%%%%%%%%%%

%\clearpage

\appendix
%========================================================
\section{From large polynomial equations to simple sub-extensions}
\label{app:subextensions} 
% ========================================================

\newcommand{\Pmab}{P}

In this section we  explain how to derive a ``simple'' expression for a series $F$, similar to those of Proposition~\ref{prop:explicit-univariate}, from a large polynomial equation satisfied by $F$, like the 
polynomial equations for the series $R_0$, $R_1$, $B_1$, and $B_2$ that we have guessed in Section~\ref{sec:guess}; see Table~\ref{tab:guessed}. More specifically, we describe how to find subextensions over $\qs(t)$ of the fields $\Q(t,R_0)$, \ldots , $\Q(t,B_2)$, and ``simple'' series in these extensions.  For this section, we have greatly benefited from the help of Mark van Hoeij (\url{https://www.math.fsu.edu/~hoeij/}). We also refer to the appendix that he wrote in~\cite{BoKa08}. The final picture is shown in Figure~\ref{fig:algtvw}.

\begin{figure}[!ht]
	\begin{tabular}{ccccccccccc}
		&& $\Q(t)$ & $\stackrel{4}{\hookrightarrow}$ & $\Q(t,\xzero)$ & $\stackrel{3}{\hookrightarrow}$ & $\Q(t,\av)$ & $\stackrel{2}{\hookrightarrow}$ & $\Q(t,\aw)$ \\
		& $\stackrel{\circled{2}}{\hookNEarrow}$ && $\stackrel{2}{\hookNEarrow}$ && $\stackrel{2}{\hookNEarrow}$ && $\stackrel{2}{\hookNEarrow}$ & \\
		$\Q(s)$ & $\stackrel{\circled{4}}{\hookrightarrow}$ & $\Q(\xzero)$ & $\stackrel{\circled{3}}{\hookrightarrow}$ & $\Q(\av)$ & $\stackrel{2}{\hookrightarrow}$ & $\Q(\av,\aw)$ &&
	\end{tabular}	
	\caption{Algebraic  structure of the fields $\Q(t,B_1)=\Q(t,v)$ and $\Q(t,R_1)=\Q(t,B_1)=\Q(t,B_2)=\Q(t,w)$. The numbers above the arrows give the degree of the extensions; circles mark rational parametrizations. We set $s=\frac{t(1+t)}{(1-8t)}$.	}
	\label{fig:algtvw}
\end{figure}

We begin with the simplest series, $\DSA$, of (conjectured) degree
$12$. We denote by
 $\Pmab(F)$ its guessed monic minimal polynomial with coefficients in $\qs(t)$.

\medskip

%=========================================================
\subsection{Finding sub-extensions}
\label{sec:sub_extensions}
%=========================================================
In principle, the  \texttt{Subfields} command of
  {\sc Maple} can determine all subextensions of $\qs(t,B_1$) of a prescribed degree. But we were unable to use it successfully with the variable $t$. Instead, we used it for several specific values of $t$. For instance, for $t=1$ the polynomial $\Pmab(F)$ is irreducible over $\qs$, and the command \texttt{evala(Subfields(subs(t=1,P(F)),d)}, with $d=2, 3, 4, 6$, shows the existence of a subfield of degree $4$ over $\qs$, generated by a number~$u$ that satisfies
  \[
199974741\,{u}^{4}-76156920\,{u}^{3}-34589883726\,{u}^{2}+248642276448\,u-521380624943
,
\]
but  of no subfield of degree $2$, $3$, or $6$.
By repeating this calculation with several fixed rational values of~$t$, one conjectures that the extension
$\Q(t,\DSA)$ indeed possesses a subfield $\GK=\Q(t,\xzero)$ of degree $4$ over $\Q(t)$. For each fixed $t$,
{\sc Maple} gives a generator $\xzero$, but it is not canonical. How
can we then construct~$u$ for a generic $t$?

If $\Q(t,B_1)$ has indeed a subfield $\GK$ of degree 4 over $\Q(t)$, then $\Pmab(F)$
factors over $\GK$ into the form  $P_3(F) P_9(F)$, where  $P_3$ (resp.\ $P_9$) is a  monic polynomial of degree $3$ (resp.\ $9$) with coefficients in $\GK$. This factorization should be reflected in the factorization of $P(F)$ over $\qs(t,B_1)$, which should then be of the form
\begin{align*}
	\Pmab(F) =  (F-\DSA)  Q_2(F) Q_9(F),
\end{align*}
where indices still indicate degrees. This time the polynomials $Q_2$ and $Q_9$ should have  coefficients in $\qs(t,B_1)$, and we would then have
\[
  P_3(F)=(F-B_1)Q_2(F).
\]
If we can compute this factorization using \Maple\ (see below what to do otherwise), we thus obtain an expression of the
minimal monic polynomial of $B_1$ over $\GK$, namely $P_3(F)$, as a polynomial in $F$ with explicit coefficients in $\qs(t,B_1)$.
Now, let us write
\[
  P_3(F)= F^3+p_2 F^2+p_1 F+p_0.
\]
By eliminating~$B_1$ from the expressions of the $p_i$'s (using the equation $P(B_1)=0$), we obtain the minimal monic polynomial of each $p_i$
over $\qs(t)$, say $M_i(p_i)=0$, where $M_i(p)$ has coefficients in $\qs(t)$. Since $p_i$ must belong to $\GK$, each $M_i(p)$  should be of degree at most $4$ in $p$. Conversely, for each $M_i$ of degree $4$ (if any), we can take $p_i$ as a generator of $\GK$ over~$\qs(t)$.

If the command \texttt{factor(P(F),RootOf(P(B1),B1))} fails, as happened for us, we can perform this factorization for several rational values of  $t$. The above procedure then gives  the value of the minimal monic polynomial $M_i(p)$ at this specific  value of $t$. Since the coefficients of $M_i$ are rational functions in $t$, we then reconstruct the value of this polynomial for a generic $t$ by rational interpolation. In practise, we were able to reconstruct the minimal polynomial $M_2(p)$, of degree $4$ in $p$, from its values obtained for $t=3, \ldots, 30$ (we start at $t=3$ because $P(F,t)$ is reducible for $t=2$). At this stage, we can conjecture that $\Q(t,B_1)$ has a subextension of degree $4$ generated by a root of $M_2(p)$, namely $p_2$. We denote by $u_1:=p_2$ this first generator of $\GK$. Note that we have not identified $p_2$ but just its minimal polynomial over $\qs(t)$.

\subsection{Finding ``simple'' generators}

However, the polynomial $M_2(p)$ is still too big for our taste. In particular, its numerator, denoted $N_2(p,t)$, is  a polynomial in $p$ and $t$, of degree $12$ in $t$.  Using the \texttt{algcurve} package, we find that $N_2(p,t)$ has genus $2$, and the \texttt{is$\_${hyperelliptic}} command tells us that it is hyperelliptic. This implies that the equation $N_2(p,t)=0$ can be written as $g^2=\Pol(f)$ where $f$ and $g$ are rational functions in $p$ and $t$, and conversely, $p$ and $ t$ can be expressed rationally in terms of $f$ and $g$. The command \texttt{Weierstrassform} determines such a pair $(f,g)$.

Next, we compute the minimal polynomials of  $f$ and $g$ over $\qs(t)$, in the hope that they are simpler than $M_2$. This is indeed the case, and we finally take $u_2:=g$ as a new generator of~$\GK$. The coefficients of its minimal polynomial over $\Q(t)$ are found to have several common factors. This leads us to introduce a new generator $u_3$, which only differs from $u_2$ by a factor of $\qs(t)$, and satisfies
\[
  9\,u_3^{4}-4\,{\frac { \left( 112\,{t}^{2}+120\,t-1 \right) 
 \left( 16\,{t}^{2}+72\,t-7 \right) }{ \left( 4\,t+1
 \right) ^{4}}}u_3^{3}+30\,u_3^{2}-12\,u_3+1 =0.
\]
Remarkably, this can be rewritten so that $u_3$ and $t$ are separated:
\[
  {\frac { \left( 3\,u_3^{2}+6\,u_3 -1 \right) ^{2}}{u_3^{3}}}=64\,{\frac { \left( 16\,{t}^{2}+24\,t-1 \right) ^{2}}{
 \left( 4\,t+1 \right) ^{4}}}.
\]
Hence, one of the square roots of $u_3$, denoted $u_4$, has also degree $4$ (and thus generates the field $\GK$) and satisfies
\[
  {\frac {3\,u_4^{4}+6\,u_4^{2}-1}{u_4^{3}}}={8\,
\frac {16\,{t}^{2}+24\,t-1}{ \left( 4\,t+1 \right) ^{2}}}
\]
or equivalently,
\[
\frac { \left( u_4+1 \right) ^{3} \left( 3\,u_4-1
 \right) }{16 \left( u_4-1 \right) ^{3} \left( 3\,u_4+1
 \right) }=-{\frac {t \left( 1+t \right) }{1-8\,t}}.
\]
Finally, with $u_5:=(u_4-1/3)/(u_4+1)$, we have reached
\beq\label{u5-eq}
  \frac {u_5}{(1+u_5)(1-3u_5)^3}= \frac{t(1+t)}{1-8t},
\eeq
where we recognize Equation~\eqref{u-def-alt} satisfied by the series $u$ of Section~\ref{sec:univariate}.
Note that this equation also shows the existence of a non-trivial
subfield of $\Q(t)$ and $\Q(u)$, namely $\Q(s)$ with $s : = \frac{t(1+t)}{1-8t}$, which we, however, have not used; see Figure~\ref{fig:algtvw}.

%===========================================
\subsection{Proving the guessed sub-extension}
%===========================================
At this stage, we  suspect that the field $\qs(t,B_1)$ contains a field $\GK=\qs(t,u_5)$, where $u_5$ is one of the roots of~\eqref{u5-eq}. In order to check this, and identify the correct root $u_5$, we factor the (guessed) minimal polynomial of $B_1$, denoted $P(F)$ above, using the command \texttt{factor(P(F),RootOf(Alg(u5),u5))}, where $\Alg(u)$ is the minimal polynomial of $u_5$. Actually a bug in the version of \Maple\ that we use forces us to have a monic polynomial instead of $\Alg(u)$, which is why we consider in practise $u'= u_5/(27t(1+t))$ instead of $u_5$. Then the factorization works, and tells us that $P(F)$ has indeed a factor $P_3(F)$ of degree $3$ with coefficients in $\qs(t,u_5)$. By expanding $P_3(B_1)$ around $t=0$ for each of the  roots of~\eqref{u5-eq}, we see that $u_5$ must be the root $u$ defined in Section~\ref{sec:univariate} as the only solution of~\eqref{u5-eq} that is a formal power series in $t$.

We have now proved that for the  (guessed) series $B_1$, the field $\qs(t,B_1)$ admits indeed $\qs(t,u)$ as  a subextension of degree $4$.

%=======================================================
\subsection{Construction of the series \texorpdfstring{$v$}{v}}
% =======================================================
We would now like to find in $\Q(t,B_1)$ a series $v$ that is also cubic above $\Q(t,u)$ (like~$B_1$), but satisfies a simpler equation, and, why not, an equation that does not involve $t$.
To investigate this, we now look at $B_1$ as an algebraic element over $\Q(u)$.
We construct its minimal monic polynomial $\tilde P(F)$ over $\Q(u)$, of degree $6$ in $F$,   by eliminating $t$ between $P_3(F)$ and the minimal polynomial 
of $u$. We now repeat the procedure of Section~\ref{sec:sub_extensions}, but with $\tilde P$ and $u$ rather than $P$ and $t$. The \texttt{Subfields} command, used for specific values of $u$, suggests that $\Q(t,B_1)=\Q(u,B_1)$ contains an extension of $\Q(u)$ of degree $2$ (which is $\Q(t,u)$), and another of degree $3$, say $\tilde \GK$, above which $B_1$ should have degree $2$. We then factor $\tilde P(F)$ over $\Q(B_1)$ for various values of~$u$, and observe the following pattern:
\[
  \tilde P(F)= (F-B_1) \tilde Q_1(F)\tilde  Q_2(F)  \hat Q_2(F),
\]
where indices indicate the degree.
Thus, the minimal polynomial of $B_1$ over $\tilde \GK$ should be
\[
  \tilde P_2(F)=(F-B_1) \tilde Q_1(F)=F^2+\tilde p_1 F+\tilde p_0.
\]
We reconstruct it again by rational interpolation in~$u$. Its two coefficients $\tilde p_1$ and $\tilde p_0$ are found indeed to have degree $3$ over $\Q(u)$. In particular, $\tilde p_1$ has  a cubic minimal polynomial  $\tilde M_1(p)$, of degree $21$ in $u$. By observing the repeated factors in the coefficients of $\tilde M_1(p)$, we introduce a series $v_1$ that differs of $p_1$ by a multiplicative factor, and satisfies
\begin{align*}
 \left( 3\,{u}^{3}-15\,{u}^{2}+9\,u+21 \right) v_1^{3}+ \left( 
9\,{u}^{4}-72\,{u}^{3}+126\,{u}^{2}+36\,u+9 \right) v_1^{2} &
\\ - \left( 18\,{u}^{5}-36\,{u}^{4}-99\,{u}^{3}-53\,{u}^{2}-7\,u-1
 \right) v_1+3\,{u}^{2} \left( 1+u \right) ^{4}
&=0.
\end{align*}
The degree in $u$ has reduced to $6$.

Now in the field $\Q(u,v_1)$, we would like to find an even simpler generator than $v_1$. The above curve is found to have genus $0$, so we have a rational parametrization this time, which \Maple\ can compute.
Since this parametrization looks pretty big, one can first  use  the \texttt{NormalBasis} package\footnote{Source code available online: \url{https://www.math.fsu.edu/~hoeij/files/NormalBasis/}.} of van Hoeij and Novocin~\cite{vanHoeijNovocin2005}, which gives a new generator $v_2$ satisfying an equation that is cubic in $v_2$ (of course) and in $u$, and then  parametrize this simpler equation with the \texttt{parametrization} command.
This is how we obtained Equation~\eqref{v-def}.
We then check that $\tilde P(F)$ actually factors over $\Q(u,v)=\Q(v)$, with one factor $\tilde P_2(F)$  of degree $2$, and that the root of~\eqref{v-def} such that this factor of degree $2$ vanishes is the one with constant term zero. Now we have proved the existence of a subfield $\Q(u,v)=\Q(v)$ in $\Q(u,B_1)=\Q(t,B_1)$.
 
 \subsection{Expression of \texorpdfstring{$B_1$}{B1}}
 We return to the minimal polynomial of $B_1$ over $\Q(v)$, namely $\tilde P_2$ and factor it over $\Q(t,v)$ using \texttt{factor(P2(F),RootOf(alg(t,v),t))}, where $alg(t,v)$ is the minimal polynomial of $v$ over $\qs(t)$, which has degree $12$ in $v$ but only $2$ in $t$. This gives us the expression of $B_1$ in Proposition~\ref{prop:explicit-univariate}.

 \subsection{Expression of \texorpdfstring{$R_0$}{R0} and construction of \texorpdfstring{$w$}{w}} 
 We return to the guessed minimal polynomial of $R_0$ over $\qs(t)$, which has degree $24$. We use the minimal polynomial~\eqref{algv} of $v$, and  the first terms of $R_0$, to obtain the minimal polynomial of $R_0$ over $\Q(v)$, which has degree $4$.  This polynomial further factors over $\Q(t,v)$, and we obtain an equation of degree $1$ in $t$,  of the form
 \[
   c_2(v) R_0^2+ t c_1(v)R_0 +c_0(t,v)=0.
 \]
 This suggests to look at the quadratic equation satisfied by $R_0/t$, which is found to have coefficients in $\Q(v)$. We solve it, which leads us to  introduce the series $w$ defined by~\eqref{w-def}, and we finally obtain the expression for $R_0$ stated  in Proposition~\ref{prop:explicit-univariate}.

  \subsection{Expressions of \texorpdfstring{$R_1$}{R1} and \texorpdfstring{$B_2$}{B2}}
  We return to the guessed minimal polynomial of $R_1$, of degree $24$ over $\Q(t)$, and derive as above an equation of degree $4$ over $\Q(v)$.
  This equation factors into four linear terms in $\Q(t,v,w)$, and this gives us the expression for $R_1$ stated in Proposition~\ref{prop:explicit-univariate}.

  We apply the same steps to the minimal polynomial of $B_2$. Recall that it has degree $12$ in $B_2^2$. As a result, the minimal polynomial of $B_2$ over $\Q(v)$ is found to be bi-quadratic.

%%%%%%%%%%%%%%%%%%%%%%%%%%%%%%%%%%%%%%%%%%%%%%%%%%%%%%%%%%%%%%%%
\section{Another parametrization for \texorpdfstring{$S(x)$}{S(x)} and \texorpdfstring{$R(x)$}{R(x)}}
\label{app:P0}
%%%%%%%%%%%%%%%%%%%%%%%%%%%%%%%%%%%%%%%%%%%%%%%%%%%%%%%%%%%%%%%%
In Section~\ref{sec:bivariate-sol} we gave two parametrizations for $S(x)$ and $R(x)$, in terms of series $\Pun$ and~$\Ptwo$. Here we give another one in terms of a series denoted $\Pzero$. We have mentioned it in the proof of Proposition~\ref{prop:RhatShat}.

The series $\tilde S(x)$ defined by~\eqref{St-def} satisfies over $\Q(\tilde z, v)$ (where $\tilde z$ is defined by~\eqref{zt-def}) a cubic equation, which can be  parametrized rationally by introducing the unique series $\Pzero$ such that $\Pzero= \bx t +\LandauO(t^2) $ and
\beq\label{zt-P0}
  \tilde z= \frac{\numsmall}{({v}^{4}+8\,{v}^{3}+6\,{v}^{2}+2\,v+1) \den},
\eeq
where
\begin{align*}
    \numsmall&=w^2 {\Pzero}^{3}+v w^2\left( {
        v}^{3}+3 v+2 \right)  {\Pzero}^{2}\\
   &\quad -v \left( 2 v+1 \right)  \left( 4 {v}^{8}+4 {v}^{7}+14
 {v}^{6}+19 {v}^{5}+7 {v}^{4}-22 {v}^{3}-32 {v}^{2}-11 v-1
\right) \Pzero\\
&\quad -{v}^{3} \left( v^2-1 \right) 
 \left( {v}^{5}+{v}^{4}+6 {v}^{3}+8 {v}^{2}+11 v+3 \right)  \left( 
   2 v+1 \right) ^{2}
\end{align*}
and
\beq\label{denP0}
  \den= w^2 \Pzero^{2}+v w^2\left( 2
 v+1 \right)   \Pzero-{v}^{2} \left( v^2-1 \right)  \left( {v}^{2}+v+1 \right)  \left( 2 v+1 \right) ^{2}.
\eeq
We have denoted, as usual, $w^2=1+4v-4v^3-4v^4$. Then we have
\[
  \tilde S(x)+\frac 1 3 =
-  {\frac {{v}^{2} \left( v^2-1 \right)  \left( 2 v+1 \right)   \left( {v}^{2}+4 v+1 \right) ^{2}}{ \left( 2 {v}^{3}-4 v-
   1 \right)
 %\left( -4 {v}^{8}-8 {v}^{7}+4 {\Pzero}^{2}{v}^{4}-11 {v}^{6}+4 {\Pzero}^{2}{v}^{3}+14 {v}^{4}-4 {\Pzero}^{2}v+8 {v}^{3}-{\Pzero}^{2}+{v}^{2} \right)
\left(w^2\Pzero^{2}+{v}^{2}
 \left( v^2-1 \right)  \left( 2 v+1 \right)  \left( 2 {v}^{3}+3 {v}^{2}+6 v+1 \right)\right) 
}}.
\]
We can also express $\Rh(x)$ in terms of $\Pzero$:
\beq\label{Rhat-P0}
  \Rh(x)=-\frac{y(1-2t)^2(1+2v)N_1(\Pzero)N_2(\Pzero)N_3(\Pzero)}
  {w(ty^2-t-1)(2v^3-4v-1)({v}^{4}+8\,{v}^{3}+6\,{v}^{2}+2\,v+1)^2 \den^2}
  \eeq
  where $y=x+1+\bx$, $\den$ is given by~\eqref{denP0} and
\begin{align*}
     N_1(U)&= w^2U^2+2v(2v+1)w^2U-v(3v^5+3v^4+2v^3+3v+1)(2v+1)^2, \\
     N_2(U)&=w^2U^2+2v^2w^2(v^2-1)U-v^3(8v^3+12v^2+15v+4)(v^2-1)^2,\\
     N_3(U)&= w^2U^2-(v^2+4v+1)^2U+v^2(v^2-1)(2v+1)(2v^3+3v^2+6v+1).
 \end{align*}

  %%%%%%%%%%%%%%%%%%%%%%%%%%%%%%%%%%%%%%%%%%%%%%%%%%%%
  \section{Another quadratic extension of \texorpdfstring{$\Q(t,v)$}{Q(t,v)}}\label{sec:appendixC}
  %%%%%%%%%%%%%%%%%%%%%%%%%%%%%%%%%%%%%%%%%%%%%%%%%%%%
This extension is different from $\Q(t,w)$, and  is involved in the description of the series $S(1)$ in Proposition~\ref{prop:univariate_series}. The series $\T$ defined by $S(1)+1/2=w\T$ has degree $24$ over $\Q(t)$, degree~$2$ over $\Q(t,v)$, and satisfies:
  \begin{align}
	\label{T-def}
	\begin{aligned}
  0 &= {\T}^{2}+{\frac { \left( 2v+1 \right) \T}{3(2\,{v}^{3}-4v-1)}} \\
    & \quad+ \frac{\numsmall''}{12w^2(1-2t)(2v^3+3v^2+6v+1)^2(2v^3-4v-1)^2(4v^3+3  v^2-1)},
	\end{aligned}
\end{align}
with
\begin{align*}
    \numsmall''&=
    2(4v^3+3v^2-1)(64v^{12}+576v^{11}+2336v^{1}0+5136v^9+6896v^8
  \\
&\quad +4652v^7-832v^6  -4756v^5-4495v^4-2300v^3-682v^2-108v-7)t
    \\ &\quad -128v^{15}-608v^{14}-1312v^{13}-2624v^{12}-4560v^{11}-6808v^{10}-5476v^9
    \\ &\quad +2088v^8+10500v^7+11309v^6+5096v^5+559v^4-220v^3-45v^2+4v+1.
\end{align*}

\newpage
	
\bibliographystyle{abbrv}
\bibliography{Bibliography}
\label{sec:biblio}

\end{document}